\documentclass[numbers,webpdf,imanum]{ima-authoring-template}%
\usepackage{xcolor}
\usepackage{todonotes}
\usepackage{cleveref}

\DeclareMathOperator{\tr}{tr}
\newcommand{\revise}[1]{{#1}}

\newtheorem*{assumption}{Assumptions}
\newcommand{\vertiii}[1]{{\left\vert\kern-0.25ex\left\vert\kern-0.25ex\left\vert #1
		\right\vert\kern-0.25ex\right\vert\kern-0.25ex\right\vert}}
\newcommand{\tnorm}[2]{\vertiii{#1}_{#2}}

\usepackage{amsthm}
\usepackage{pgfplots}
\usepackage{pgfplotstable}
\usepackage{subcaption}
\usepackage{tikz}
\usepackage{booktabs}

\usepackage{cleveref}
\newtheorem{corollary}{Corollary}

\definecolor{seabornblue}{rgb}{0.2980392156862745, 0.4470588235294118, 0.6901960784313725}
\definecolor{seaborngreen}{rgb}{0.3333333333333333, 0.6588235294117647, 0.40784313725490196}
\definecolor{seabornred}{rgb}{0.7686274509803922, 0.3058823529411765, 0.3215686274509804}
\graphicspath{{Fig/}}

\theoremstyle{thmstyletwo}%
\newtheorem{theorem}{Theorem}%  meant for continuous numbers
\newtheorem{remark}{Remark}%
\newtheorem{lemma}{Lemma}%

\numberwithin{equation}{section}

\begin{document}

\DOI{DOI HERE}
\copyrightyear{2023}
\vol{00}
\pubyear{2023}
\access{Advance Access Publication Date: Day Month Year}
\appnotes{Paper}
\copyrightstatement{Published by Oxford University Press on behalf of the Institute of Mathematics and its Applications. All rights reserved.}
\firstpage{1}

%\subtitle{Subject Section}

\title[FEM discretization of the smectic density equation]{Finite-element discretization of the smectic density equation}

\author{Patrick E. Farrell\ORCID{0000-0002-1241-7060}
\address{\orgdiv{Mathematical Institute}, \orgname{University of Oxford}, \orgaddress{\street{Woodstock Rd}, \postcode{OX2 6GG}, \country{UK}}}}
\author{Abdalaziz Hamdan\ORCID{0000-0003-0287-2630}
\address{\orgdiv{Department of Mathematics}, \orgname{Imperial College London}, \orgaddress{\street{Exhibition Rd}, \postcode{SW7 2AZ}, \country{UK}}}}
\author{Scott P. MacLachlan \ORCID{0000-0002-6364-0684}
\address{\orgdiv{Department of Mathematics and Statistics}, \orgname{Memorial University of Newfoundland}, \orgaddress{\street{Elizabeth Ave}, \postcode{A1C 5S7}, \state{NL}, \country{Canada}}}}

%\authormark{Author Name et al.}

\corresp[*]{Corresponding author: \href{email:a.hamdan@imperial.ac.uk}{a.hamdan@imperial.ac.uk}}

\received{Date}{0}{Year}
\revised{Date}{0}{Year}
\accepted{Date}{0}{Year}

%\editor{Associate Editor: Name}

\abstract{\revise{The fourth-order PDE that models the density variation of smectic A liquid crystals presents unique challenges in its (numerical) analysis beyond more common fourth-order operators, such as the classical
    biharmonic.  While the operator is positive definite, the equation has a ``wrong-sign'' shift, making it somewhat more akin to an indefinite Helmholtz operator, with lowest-energy modes consisting of plane waves. As a result, for large shifts, the natural continuity, coercivity, and inf-sup constants degrade considerably, impacting standard error estimates.}  In
	this paper, we analyze and compare three finite-element formulations for such
	PDEs, based on $H^2$-conforming elements, the $C^0$ interior penalty
	method, and a mixed finite-element formulation that explicitly introduces approximations to the gradient of the
	solution and a Lagrange multiplier.
	The conforming method is simple but is impractical to apply in three dimensions; the interior-penalty method works well in two and three dimensions but has lower-order convergence and (in preliminary experiments) seems difficult to precondition; the mixed method uses more degrees of freedom, but works well in both two and three dimensions, and is amenable to monolithic multigrid preconditioning. \revise{Our analysis reveals different behaviours of the error bounds with the shift parameter and mesh size for the different schemes.}
	Numerical
	results verify the finite-element convergence for all
	discretizations, and illustrate the trade-offs between the three schemes.}

% \boxedtext{
% \begin{itemize}
% \item Key boxed text here.
% \item Key boxed text here.
% \item Key boxed text here.
% \end{itemize}}

\maketitle

\section{Introduction}\label{sec1}
Recent years have seen significant and successful efforts in developing numerical models of various liquid crystalline materials \cite{pevnyi2014modeling, MR3439774, MR3504546, MR3403140,MR1618476,MR3033068,MR2596543,Ball-apala}.
In these models, equilibrium states of liquid crystals usually correspond to minimizers of a given energy functional, which can be directly discretized using finite-element (or other) variational techniques.  Smectic A liquid crystals are characterized by their natural propensity to form layers with periodic variation in the density of the liquid crystal aligned with the orientation of the molecules.  While some models make use of a complex order parameter as a model of the energy of liquid crystals \cite{gennes-1972-article}, several recent papers have proposed models based on a real-valued density variation \cite{pevnyi2014modeling, doi:10.1080/15421406.2015.1030571, xia2021structural}.  For example, Pevnyi et al.~\cite{pevnyi2014modeling} propose a model
\begin{equation}\label{eq:PSS}
	E(u,\vec{\nu})=\int_{\Omega} \frac{m_1}{2} u^2+\frac{m_2}{3}u^3 +\frac{m_3}{4}u^4+B\left|\nabla\nabla u+q^2\vec{\nu}\otimes\vec{\nu}u\right|^2+\frac{K}{2}|\nabla \vec{\nu}|^2,
\end{equation}
where $\Omega\subset \mathbb{R}^{d}$, for $d\in\{2,3\}$, $u:\Omega\rightarrow \mathbb{R}$ represents the variation in the density of the liquid crystal from its average density, $\vec{\nu}: \Omega \to \mathbb{R}^d$ is the unit-length director of the liquid crystal (the local axis of average molecular alignment), and $m_1,\ m_2,\ m_3,\ q,\ K$, and $B$ are real-valued constants determined by the liquid crystal under consideration.  Of these, the smectic wavenumber, $q$, is notable because it prescribes a preferred wavelength for the solution of $2\pi/q$.  Here, and in what follows, we use $|\boldsymbol{T}|^2 = \boldsymbol{T}:\boldsymbol{T}$ to denote the Frobenius norm squared of a tensor $\boldsymbol{T}$ (of any order), defined as the sum of squares of the entries in $\boldsymbol{T}$ at a given point in $\Omega$. 

It is well-known that representing the orientation of the liquid crystal with a vector-valued director cannot represent certain defects of the liquid crystal~\cite{ball2017}.
In \cite{xia2021structural}, Xia et al.~adapted \eqref{eq:PSS} to make use of a tensor-valued order parameter in place of the director field, proposing
\begin{eqnarray}\label{Xia-et-all}
	E(u,\boldsymbol{Q})&=&\int_{\Omega} \frac{m_1}{2} u^2+\frac{m_2}{3}u^3 +\frac{m_3}{4}u^4+B\left|\nabla\nabla u+q^2\left(\boldsymbol{Q}+\frac{I_d}{d}\right)u\right|^2\notag\\
	&+&\frac{K}{2}|\nabla \boldsymbol{Q}|^2+f_n(\boldsymbol{Q})\label{E},
\end{eqnarray}
where $\boldsymbol{Q}: \Omega \to \mathbb{R}^{d \times d}_{\text{sym}}$ is the \revise{symmetric, traceless} tensor-valued order parameter, $I_d$ is the $d \times d$ identity
matrix, and $f_n(\boldsymbol{Q})=-l\tr(\boldsymbol{Q}^2)+ l\left(\tr(\boldsymbol{Q}^2)\right)^2$ for $d=2$ and $f_n(\boldsymbol{Q}) = -l\tr(\boldsymbol{Q}^2)-\frac{l}{3}\tr(\boldsymbol{Q}^3)+ \frac{l}{2} \left(\tr(\boldsymbol{Q}^2)\right)^2$ for $d = 3$. Here, the penalty parameter, $l$, and the functions $f_n(\boldsymbol{Q})$ are chosen so that the minimizer of $\int_\Omega f_n(\boldsymbol{Q})$ is of the form $\boldsymbol{Q} = \vec{\nu}\otimes\vec{\nu} - \frac{I_d}{d}$, and are included in the energy to weakly enforce the rank-one condition implied by \eqref{eq:PSS}. A related model was proposed by Ball \& Bedford~\cite{doi:10.1080/15421406.2015.1030571}.
%While there remain many open questions about the physical values of the constants $m_1,m_2,m_3,q,K$, and $B$
An important feature of these models is the energetic competition between the term encouraging alignment of the orientation and density variation, scaled by $B$, and the deformation of the director field, scaled by $K$. The Euler--Lagrange equations for either of these functionals \eqref{eq:PSS} or \eqref{E} naturally lead to a coupled system of PDEs, with a fourth-order operator acting on $u$ and a second-order operator acting on $\vec{\nu}$ or $\boldsymbol{Q}$. While the discretization of the vector or tensor Laplacian is relatively standard, the fourth-order PDE involving $u$, which we refer to as the ``smectic density equation,'' is of a type not previously studied in the literature.
%\spm{I'm slightly annoyed by the typography of having $B$ be a constant, but $T$ be a tensor.  Should we switch to a lower-case letter for $B$ (breaking the physics notation) or use calligraphic $\mathcal{T}$ for the tensor?  Also, do we want to change $\vec{n}$ in the above for something else, so that it doesn't conflict with our use of $\vec{n}$ below for the outward normal vector?}
%\spm{I'm not sure what the proper tensor terminology is here for a contraction over all indices.  Writing $T:T$ for matrices (2-tensors) is reasonable and (I think) indisputably means what we want it to, but I'm less sure for a 3-tensor like $\nabla Q$.}

Motivated by such examples, we consider the discretization of a simplification of \eqref{E} with suitable boundary conditions, given in variational form as
\begin{equation} \label{eqn:our_problem}
	\min_{v\in \mathcal{V} \subset H^2(\Omega)}\frac{B}{2}\int_{\Omega}\left|\nabla\nabla v+q^2\boldsymbol{T}v\right|^2+\frac{m}{2}\int_{\Omega}v^2-\int_{\Omega}fv,
\end{equation}
where, motivated by the above, $\boldsymbol{T}$ is a bounded $d\times d$ tensor, with $|\boldsymbol{T}|^2 \leq\mu_1$ for $\mathcal{O}(1)$ constant $\mu_1$, while $m\ \text{and}\ q$ are $\mathcal{O}(1)$ positive constants, and $0< B\leq 1$. \revise{In \eqref{Xia-et-all}, $Q$ is symmetric, but we do not require that $T$ be symmetric in \eqref{eqn:our_problem}.} Sufficiently smooth extremizers of this energy must satisfy its Euler--Lagrange equations, which yield the fourth-order smectic density equation,
\begin{equation}\label{eq:mod_biharmonic}
	B\nabla\cdot\nabla\cdot\big(\nabla\nabla u+q^2\boldsymbol{T}u\big)+Bq^2\boldsymbol{T}:\nabla\nabla u+(Bq^4\boldsymbol{T}:\boldsymbol{T}+m)u=f. 
\end{equation}
\revise{When $T = \vec{\nu}\otimes\vec{\nu}$, we note that the functions $u(\vec{x}) = ce^{\pm \imath q\vec{\nu}\cdot\vec{x}}$ are in the nullspace of the term multiplied by $B$ in~\eqref{eqn:our_problem} and, consequently, near-nullspace modes of the differential operator in~\eqref{eq:mod_biharmonic}.  This highlights the difficulty in developing accurate numerical methods for this problem, similar to those encountered for the second-order indefinite Helmholtz equation.  We also highlight the positive sign on the second-order terms in~\eqref{eq:mod_biharmonic} that is opposite what would be expected on an $H^2$-elliptic PDE, where second-order terms typically have negative signs.}

We consider three finite-element formulations for this fourth-order problem, with a particular focus on the treatment of the boundary conditions that arise naturally from the transition from the variational to strong forms. Two formulations are based on discretizing the variational first-order optimality condition for \eqref{eqn:our_problem}, using either $H^2$-conforming or $C^0$ interior penalty (C0IP) methods.  The third formulation is based on mixed finite-element principles, introducing the gradient of the solution as an explicit variable constrained using a Lagrange multiplier, and leverages standard discretizations for the Stokes problem in order to achieve inf-sup stability. Both the C0IP and mixed approaches are quite general, in the sense that we achieve high-order convergence when using high-order elements if the solution is sufficiently smooth. However, both the $H^2$-conforming and C0IP formulations provide slightly suboptimal convergence, as shown later.  Complications to achieving optimal convergence come from the use of Nitsche's method to weakly enforce essential boundary conditions and weakly imposing $C^1$ continuity in the C0IP approach.
% Removed this sentence for being too much information for the intro:
%%%% The nonsymmetric versions of Nitsche's method for the Dirichlet boundary conditions in the $H^2$-conforming method and  weakly imposing the $C^1$ conformity in the C0IP method lead to suboptimal convergence rates in the $L^2$ norm. The gradient approximation in the mixed formulations is suboptimal because of the "unbalanced" approximation properties, which are of $\mathcal{O}(h^{k+1})\ \text{for}\ DG_{k}(\Omega,\tau_h) \ \text{in the}\ L^2 \ \text{norm}, \ \text{and}\ \mathcal{O}(h^{k+2}) \ \text{for}\ CG_{k+2}(\Omega,\tau_h)$ in the $H^1$-norm.  
%\spm{We should probably revisit this claim of suboptimality now that we have finalized results for all cases.}

Finite-element methods for fourth-order $H^2$-elliptic problems have been extensively studied. These include conforming methods, such as the use of Argyris elements, nonconforming methods \cite{MR2373954,MR2207619,MR0520174}, $C^0$ interior penalty methods (C0IP, which are also nonconforming) \cite{MR2142191,MR3022211,MR345432,MR2298696}, and mixed-finite element methods, including two-field~\cite{malkus1978mixed, cheng2000some,MR0520174}, three-field \cite{MR3580405,farrell2021new}, and four-field discretizations \cite{li2017stable,behrens2011mixed, MR3852718}.  While conforming methods are attractive in two dimensions, the natural extension of Argyris elements to $\mathbb{R}^3$ requires the use of ninth-order polynomials on each element~\cite{zhang2009}, which is prohibitively expensive in comparison to low-order methods.  While a C0IP method was used in \cite{xia2021structural} and is analyzed herein, preliminary experiments showed that it is difficult to develop effective preconditioners for this discretization, motivating the consideration of alternate approaches.  Thus, we also propose a mixed finite-element discretization of \eqref{eqn:our_problem} that does not require growth in polynomial order in three dimensions, and which we expect to be more amenable to the development of effective preconditioners, similar to those in \cite{farrell2021new}.

An additional challenge in considering the models of smectic LCs in \cite{pevnyi2014modeling, doi:10.1080/15421406.2015.1030571, xia2021structural} is that of proper treatment of the boundary conditions.  In particular, these models typically include only natural BCs on the density variation, $u$, and do not strongly impose Dirichlet boundary conditions, such as the ``clamped'' boundary conditions that are commonly considered for the biharmonic problem.  Indeed, the case of clamped boundary conditions, where the Hessian and Laplacian weak forms of a fourth-order operator are equivalent, has been extensively studied \cite{MR2142191,malkus1978mixed,cheng2000some,li2017stable,behrens2011mixed, MR3852718,MR3533246, MR3667017, Pauly-Zulehner}.  \revise{For instance, in \cite{MR3533246}, the 2D Hessian is translated into three second-order elliptic problems that are solved sequentially using the preconditioned conjugate gradient method and multigrid preconditioners as an iterative solver for the three elliptic systems.  As noted therein, the extension of this approach to free boundary plate-bending problems is not straight-forward.} A central question in this work is how to treat the more general forms of boundary conditions that arise when moving from the variational form in \eqref{eqn:our_problem} to the strong form in \eqref{eq:mod_biharmonic}, summarized in \eqref{55}-\eqref{!!!!} below.  To our knowledge, existing results in the literature treat the cases of clamped boundary conditions, Cahn--Hilliard boundary conditions (a special case of those in \eqref{!!!!} when $q=0$) \cite{MR3022211}, and simply supported boundary conditions (a special case of those in \eqref{55}) \cite{MR3051409}, but not the case of full Neumann
boundary conditions.  In the case of simply supported boundary conditions, existence of minimizers of \eqref{E} when $q\geq 0$ as well as error estimates for its discretization using $C^0$ interior penalty methods when $q=0$ were obtained in \cite{xia2021variational}. Here, we prove well-posedness of \eqref{eqn:our_problem} and provide error estimates for its discretization using Argyris elements, C0IP, and a mixed finite-element method. %\revise{We point out that these methods are also applicable for the coupled multiphysics system \eqref{Xia-et-all} when $\boldsymbol{Q}$ is also discretized, but slow convergence of Newton's method has been noticed. This case, in addition to a speed up methodology for Newton's method, will be presented in detail in future work.}

This paper is organized as follows. A brief summary of the tools needed for the finite-element analysis is presented in~\Cref{sec:prelim}. The continuum analysis, including the weak forms and uniqueness theory, is presented in~\Cref{sec:continuum}. In~\Cref{sec:discrete}, we present the conforming, C0IP, and mixed finite-element methods, and analysis of both well-posedness and error estimates for these methods.  Finally, numerical experiments to compare the different finite-element methods are presented in~\Cref{sec:numerics}.

\section{Preliminaries}\label{sec:prelim}
In all that follows, we assume $\Omega \subset \mathbb{R}^d$, $d \in \{2, 3\}$, to be a bounded simply connected polytopal domain with Lipschitz boundary.
We recall the standard Sobolev spaces
\begin{eqnarray*}
	H^1_\Gamma(\Omega)&=&\left\{u\in H^1(\Omega)\middle| \ u=0 \ \text{on}\ \Gamma\right\},\\
	H_\Gamma(\mathrm{div};\Omega)&=&\left\{\vec{v}\in H(\mathrm{div};\Omega)\middle| \ \vec{v}\cdot\vec{n}=0 \ \text{on}\ \Gamma\right\},
\end{eqnarray*}
where $\Gamma\subset\partial\Omega$ and $\vec{n}$ is the outward unit normal to $\Gamma$. \revise{We define the divergence of a tensor $\boldsymbol{T}$ row-wise as $[\nabla\cdot T]_i = \sum_{j}\partial_j T_{ij}$}.
Let $\{\mathcal{T}_h\}$ be a quasiuniform family of triangulations of $\Omega$, with $\text{diam}(\tau) \leq h$ for all $\tau\in\mathcal{T}_h$, and let $CG_{k}(\Omega,\mathcal{T}_h),\ DG_{k}(\Omega,\mathcal{T}_h)$, and $RT_k(\Omega,\mathcal{T}_h)$ be the standard continuous Lagrange, discontinuous Lagrange, and Raviart--Thomas approximation spaces of degree $k$, respectively, on mesh $\mathcal{T}_h$. We also define ${CG}_{k}^\Gamma(\Omega,\mathcal{T}_h) = CG_{k}(\Omega,\mathcal{T}_h) \cap H^1_\Gamma(\Omega)$ and ${RT}_{k}^\Gamma(\Omega,\mathcal{T}_h)={RT}_{k}(\Omega,\mathcal{T}_h)\cap H_{\Gamma}(\text{div};\Omega)$. We use $\|u\|_0$ and $\|u\|_{0,\tau}$ to denote the $L^2$ norm of $u$ over all of $\Omega$ and over a single element, $\tau$, respectively. Similar definitions are used, as needed, for other standard Sobolev norms. 
\begin{remark}
  \revise{
    The assumption of quasiuniform meshes is essential for the conforming and mixed methods presented below, as they rely on
    \cite[Theorem 4.5.11]{MR2373954} and \cite[Lemma 2.7]{MR3660776}, respectively, both of which require the meshes to be quasiuniform.  For the C0IP discretization,this can be weakened to an assumption of shape-regularity, but we leave this generalization for future work.}   	
\end{remark}
\begin{remark}\label{Remark4.2}
	In what follows, we use $C$ to represent a generic positive constant
	that can depend on the domain, shape regularity of the triangulation, $\mathcal{T}_h$, and the polynomial degree $k$ of the finite-element space, but not on the mesh parameter, $h$, nor the smectic wavenumber, $q$, and may be different in different instances. Where needed, we will use $\big\{C_i\big\}$ to denote different arbitrary constants in the same expression. 
\end{remark}
To prove well-posedness, we will make use of a standard \revise{inverse} estimate, where $\|u\|_{0,\kappa}^2 = \int_{\kappa}u^2ds$, for $\kappa\subseteq\partial\Omega$ or $\kappa\subseteq\partial\tau$.
\begin{theorem}\label{Trace-discrete}
		Let $q\geq 1$.  There exists a constant $C>0$ such that for any \revise{quasiuniform} triangulation $\mathcal{T}_h$ and any element $\tau\in\mathcal{T}_h$, for all  $u\in H^{1}(\tau)$,
		\begin{eqnarray}\label{trace-q}
			\|u\|^2_{0,\partial \tau}\leq C \left(\left(\frac{h}{q}\right)\|\nabla u\|^2_{0,\tau}+ \left(\frac{h}{q}\right)^{-1}\|u\|_{0,\tau}^2 \right).
		\end{eqnarray}
	\end{theorem}
	\begin{proof}
		These inequalities are direct results of the multiplicative trace inequality~\cite[Lemma 12.15]{MR4242224} and Young's inequality. 
	\end{proof}
	\begin{corollary}\label{Trace-cor}
		Let $q\geq 1$ and $\Gamma\subset \partial\Omega$.  There exists a constant $C>0$ such that for any {quasiuniform} triangulation $\mathcal{T}_h$, if $u\in H^{1}(\tau)$  for all $\tau \in \mathcal{T}_h$, then
		\begin{eqnarray}
			\|u\|^2_{0,\Gamma}\leq C\left(\left(\frac{h}{q}\right)\sum_{\tau \in\mathcal{T}_h}\|\nabla u\|^2_{0,\tau}+ \left(\frac{h}{q}\right)^{-1}\sum_{\tau \in\mathcal{T}_h}\|u\|^2_{0,\tau}\right).
		\end{eqnarray} 
		Moreover, if $u\in H^1(\Omega)$, then
		\begin{eqnarray}
			\|u\|^2_{0,\Gamma}\leq C\left(\left(\frac{h}{q}\right)\|\nabla u\|_0^2+ \left(\frac{h}{q}\right)^{-1}\| u\|^2_{0}\right).
		\end{eqnarray} 
	\end{corollary}
	\begin{proof}
		For any $\tau$ such that $\Gamma\cap\partial\tau\neq \emptyset$, we have $\|u\|_{0,\Gamma\cap\partial\tau}\leq \|u\|_{0,\partial\tau}$. Thus, by Theorem~\ref{Trace-discrete}, we have
		\begin{align*}
			\|u\|^2_{0,\Gamma}&= \sum_{\tau:\Gamma\cap\partial\tau\neq \emptyset}\|u\|^2_{0,\Gamma\cap\partial\tau}\leq \sum_{\tau:\Gamma\cap\partial\tau\neq \emptyset}\|u\|^2_{0,\partial\tau}\\
			&\leq C\sum_{\tau:\Gamma\cap\partial\tau\neq\emptyset}\left(\left(\frac{h}{q}\right)\|\nabla u\|^2_{0,\tau}+ \left(\frac{h}{q}\right)^{-1}\| u\|^2_{0,\tau}\right)\\
			&\leq C\left(\left(\frac{h}{q}\right)\sum_{\tau \in\mathcal{T}_h}\|\nabla u\|^2_{0,\tau}+ \left(\frac{h}{q}\right)^{-1}\sum_{\tau \in\mathcal{T}_h}\|u\|^2_{0,\tau}\right).
		\end{align*} 
		Moreover, if $u\in H^1(\Omega)$, then $\|u\|_0^2=\sum_{\tau\in\mathcal{T}_h}\|u\|^2_{0,\tau}$ and 
		$\|\nabla u\|_0^2=\sum_{\tau\in\mathcal{T}_h}\|\nabla u\|^2_{0,\tau}$.
\end{proof}
As is typical in the analysis of finite-element methods, our well-posedness results will rely on a Poincar\'e inequality.  In order to treat a more general set of boundary conditions, we state a more general form of the standard inequality.
\begin{lemma}\label{Lemma-poinc}(Poincar\'e Inequality~\cite{MR1974504,wloka_1987,MR0227584,MR2106270})
	If $u\in H^j(\Omega),\ j\in\{1,2\}$, then
	\begin{eqnarray*}
		\|u\|_{j-1}^2\leq C\left(|u|_j^2+\xi^2(u)\right),
	\end{eqnarray*}
	where $\xi$ is any seminorm on $H^j(\Omega)$ with the properties that
	\begin{itemize}
		\item There exists a constant $C>0$ such that, for all $u\in H^j(\Omega)$, we have 
		\begin{equation*}
			\xi(u)\leq C\|u\|_j.
		\end{equation*}
		\item If $p$ is a polynomial of degree less than $j$ (i.e., a constant function if $j=1$ or linear function if $j=2$), $\xi(p)=0$ if and only if $p=0$.
	\end{itemize} 
\end{lemma}
In particular, for $j=1$ and any functions $\psi_1$ and $\psi_2$ that are square integrable on $\Omega$ and $\partial \Omega$, respectively, with $\int_{\Omega}\psi_1\neq 0$, and $\int_{\partial \Omega}\psi_2\neq 0$, then the Poincar\'e inequality above holds for either seminorm $\xi_1(u)$ or $\xi_2(u)$ \cite{MR1974504}, defined as
\begin{eqnarray*}
	\xi_1(u)=\left|\int_{\Omega}\psi_1u\right|,\ \  \xi_2(u)=\left|\int_{\partial\Omega}\psi_2u\right|.
\end{eqnarray*}
Note that the choices of $\psi_1=1\ \text{on}\ \Omega, \ \text{and}\ \psi_2=1\ \text{on}\ \partial \Omega$ lead to the classical Poincar\'e inequalities that are most commonly used. For $j=2$, we will use $\xi(u)=\|u\|_0$, which satisfies the conditions above.  
%\spm{Need to fix notation here, so that we define this in $\mathbb{R}^d$ and not just $\mathbb{R}^2$}

A useful function in our analysis is the solution of $-\Delta S=1$ with homogeneous Dirichlet boundary conditions.  We next recall some properties of this function and its discrete approximation.
\begin{lemma}\label{Lemma-coercivity-poinc}
	The weak solution of $-\Delta S=1$, with homogeneous Dirichlet boundary conditions, has positive mean and bounded $H^1$ norm. In addition, the discrete solution, $S_h\in CG_1^{\partial\Omega}(\Omega,\mathcal{T}_h)$, of $\langle \nabla S_h,\nabla v \rangle = \langle 1,v\rangle$ for all $v\in CG_1^{\partial\Omega}(\Omega,\mathcal{T}_h)$ has the same properties. 
\end{lemma}
\begin{proof}
	The solutions, $S$ and $S_h$ , satisfy $\|\nabla S\|_0^2=\int_{\Omega}S$ and $\|\nabla S_h\|_0^2=\int_{\Omega}S_h$. Therefore, $S$ and $S_h$ have positive means. In addition, \revise{the basic energy estimate} that $ \|S\|_1\leq C\|1\|_0=C \,\mathrm{Area}(\Omega)$ and  $ \|S_h\|_1\leq C \,\mathrm{Area}(\Omega)$ also hold~\cite{MR3097958}.
\end{proof}

Our well-posedness proofs also rely on the standard Helmholtz decomposition in 2D \revise{for the $H(\text{div};\Omega)$ and $RT_{k}(\Omega,\mathcal{T}_h)$ spaces}, which we first state for the continuum case.  We use the standard definition of the curl of a scalar function in 2D, as $\nabla\times p = \begin{bmatrix} -p_y \\ p_x \end{bmatrix}$. 

\begin{lemma}\label{lemma-helm-cont}(2D Helmholtz decomposition \revise{for $H(\text{div};\Omega)$} ~\cite{MR2974162,MR851383,Mira})
	Let $\Omega \subset \mathbb{R}^2$, with $\partial\Omega=\Gamma_a\cup\Gamma_b$, where $\Gamma_a$ and $\Gamma_{b}$ are disjoint. 
	For $\vec{\alpha}\in H_{\Gamma_{a}}(\mathrm{div};\Omega)$, the following Helmholtz decomposition holds
	\begin{equation}
		\vec{\alpha}=\nabla\phi+\nabla\times p,
	\end{equation} 
	where $p\in H^1_{\Gamma_a}(\Omega)$, and $\phi\in H^1_{\Gamma_{b}}(\Omega)$ is the solution of $\int_{\Omega}\nabla \phi \cdot\nabla \chi=-\int_{\Omega}\nabla\cdot\vec{\alpha}\chi, \ \forall\chi\in H^1_{\Gamma_{b}}(\Omega)$. Furthermore, $p$ is a zero-mean function if $\Gamma_{a}=\emptyset$, and $\phi$ is a zero-mean function if $\Gamma_{b}=\emptyset$. This decomposition is orthogonal in the $L^2$ and $H(\mathrm{div})$ norms. %\revise{The $L_2(\Omega)$ space exhibits similar results, however, this is outside the purview of this work.}
\end{lemma}
\begin{remark}\label{rem2.2}
	If, in addition to the assumptions of Lemma \ref{lemma-helm-cont},  $\vec{\alpha}\in \left[H^{t+2}(\Omega)\right]^2\cap H_{\Gamma_{a}}(\mathrm{div};\Omega)$, $\nabla\cdot\vec{\alpha}\in H_0^{t+1}(\Omega)$, for $t\geq 0$, $\phi\in H^1_{\Gamma_{b}}(\Omega)$ is the solution of the mixed boundary value problem of the form $\int_{\Omega}\nabla \phi \cdot\nabla \chi=-\int_{\Omega}\nabla\cdot\vec{\alpha}\chi, \ \forall\chi\in H^1_{\Gamma_{b}}(\Omega)$, and the linear segments of $\Gamma_a$ and $\Gamma_{b}$ are ordered such that the singular solutions in \cite[Theorem 5.1.1.5]{MR3396210} do not arise, then $\phi\in H^{t+3}(\Omega)\cap H^1_{\Gamma_b}(\Omega)$, and $p$ is at least in $\varrho=\{p\in H^{t+2}(\Omega)\cap H^1_{\Gamma_{a}}\ | \ \nabla \times p\in[ H^{t+2}(\Omega)]^2\}$. 
\end{remark}
To see that this regularity can be achieved, we next state a regularity result for solution of the Poisson problem with mixed boundary conditions on the unit square.
\begin{lemma}
	Let $u$ be the solution of
	\begin{eqnarray*}
		-\Delta u&=&f, \ \ \ \ \ \text{in}\ \Omega,\\
		u&=&0, \ \ \ \ \ \text{on}\ \mathscr{D},\\
		\frac{\partial u}{\partial \vec{n}}&=&0, \  \ \ \ \ \text{on}\ \mathscr{N},
	\end{eqnarray*} 	
	where $f\in H_0^t(\Omega)$, $t>0$, and $\Omega=(0,1)^2$, $\partial\Omega=\Gamma_N\cup\Gamma_S\cup\Gamma_E\cup\Gamma_W$, where $\Gamma_N,\ \Gamma_S,\ \Gamma_E, \ \text{and}\ \Gamma_W$ are the North, South, East, and West faces of the square, respectively, and $\mathscr{D, \ N}$ are the faces on which we impose Dirichlet and Neumann boundary conditions, respectively. Then, if any of the following conditions hold, $u\in H^{t+2}(\Omega)$,
	\begin{itemize}
		\item $\mathscr{D}=\partial \Omega$ or $\mathscr{N}=\partial \Omega$ , and $t$ is even,
		\item $\mathscr{D}=\Gamma_E\cup\Gamma_W$  and  $\mathscr{N}=\Gamma_N\cup\Gamma_S$, and $t$ is odd, or
		\item  $\mathscr{D}=\Gamma_N\cup\Gamma_S$  and  $\mathscr{N}=\Gamma_E\cup\Gamma_W$, and $t$ is odd.
	\end{itemize}
\end{lemma}
\begin{proof}
	This result is a direct consequence of \cite[Theorem 5.1.1.5]{MR3396210}. Note that, in the notation of \cite{MR3396210}, we have $\Phi_j=0$ if edge $\Gamma_j\in \mathscr{N}$, and $\Phi_j=\frac{\pi}{2}$ if edge $\Gamma_j\in \mathscr{D}$.  Then, if $j$ and $j+1$ denote adjacent edges, we require that $\frac{\Phi_j-\Phi_{j+1}+(t+1)\frac{\pi}{2}}{\pi}$ is not an integer to achieve the stated regularity result, which is guaranteed for the cases given above.  In addition, for each pair of adjacent edges, $2\frac{\Phi_j-\Phi_{j+1}+m\pi}{\pi}$ is an integer for any integer $m$, which precludes any singular solutions. These conditions are sufficient to guarantee that $u\in H^{t+2}(\Omega)$.
\end{proof}
We also make use of the discrete analogue of Lemma~\ref{lemma-helm-cont}, stated next.
\begin{lemma}\cite{MR1401938,MR2974162}\label{lemma,helm}
	Under the same assumptions as Lemma~\ref{lemma-helm-cont}, the Helmholtz decomposition of $RT^{\Gamma_{a}}_{k+1}(\Omega,\mathcal{T}_h)$ is	
	\begin{equation}
		RT^{\Gamma_a}_{k+1}(\Omega,\mathcal{T}_h)=\bigg(\nabla^{\Gamma_{a}}_h DG_{k}(\Omega,\mathcal{T}_h)\bigg)\oplus\bigg(\nabla\times CG^{\Gamma_{a}}_{k+1}(\Omega,\mathcal{T}_h)\bigg),
	\end{equation} 
	where  $\nabla^{\Gamma_{a}}_h$ is the discrete gradient operator, $\nabla^{\Gamma_{a}}_h:DG_{k}(\Omega,\mathcal{T}_h)\rightarrow RT^{\Gamma_{a}}_{k+1}(\Omega,\mathcal{T}_h)$, such that 
	\begin{equation}\label{grad_h}
		\int_{\Omega}\nabla^{\Gamma_{a}}_h u\cdot\vec{v}=-\int_{\Omega}u\nabla\cdot\vec{v}, \ \ \forall \vec{v}\in RT^{\Gamma_{a}}_{k+1}(\Omega,\mathcal{T}_h).
	\end{equation} 
	This decomposition is orthogonal in the $L^2$ and $H(\mathrm{div})$ norms.	
\end{lemma}
Finally, we note that the Helmholtz decomposition allows us to define an alternative norm on $H(\mathrm{div};\Omega)$, which will be useful in the later analysis.
\begin{remark}
	Let $\vec{\alpha}_1, \vec{\alpha}_2 \in H(\mathrm{div};\Omega)$ with $\vec{\alpha}_1=\nabla\phi_1+\nabla\times p_1$, and $\vec{\alpha}_2=\nabla\phi_2+\nabla\times p_2$, computed as in Lemma~\ref{lemma-helm-cont}.  The following defines an inner product on $H(\mathrm{div};\Omega)$: 
	\begin{equation}\label{eq:Div_norm}
		(\vec{\alpha}_1,\vec{\alpha}_2)_{\mathrm{Div}}=q^{-4}\left(\int_{\Omega}p_1p_2+\int_{\Omega}\nabla\phi_1\cdot\nabla\phi_2+\int_{\Omega}\nabla\cdot\vec{\alpha}_1\nabla\cdot\vec{\alpha}_2\right),
	\end{equation}
	where $q$ is a positive constant (which will be taken as the $q$ in \eqref{eqn:our_problem}).
\end{remark}
%\revise{Finally, we state the following lemma, which is essential to prove the inf-sup condition of our mixed formulation. 
%	\begin{lemma}
%		Let $\partial\Omega = \Gamma_{1}\cup\Gamma_{2}$, then for all $p\in L^2(\Omega)$ with $p$ is of zero mean if  $\Gamma_{1}$ is empty, there exists $\vec{\psi} \in [H^2_{\Gamma_{1}}(\Omega)]^d$ such that 
%		\begin{eqnarray*}
%		\int_{\Omega}p\nabla\cdot \vec{\psi} \geq C_1\|p\|^2_0, \ and \ \|\vec{\psi}\|_1\leq C_2\|p\|_0.
%		\end{eqnarray*} 
%If, moreover, $p_h\in CG_{k}(\Omega, \mathcal{T}_h)$, then there exits $\vec{\psi}_h \in [CG_{k+2}(\Omega, \mathcal{T}_h)]^d$ such that 
%\begin{eqnarray*}
%	\int_{\Omega}p_h\nabla\cdot \vec{\psi}_h \geq C_1\|p_h\|^2_0, \ and \ \|\vec{\psi}_h\|_1\leq C_2\|p_h\|_0.
%\end{eqnarray*} 	
%\end{lemma}}
\section{Continuum Analysis}\label{sec:continuum}
%\spm{I feel like we jump a little to fast in what follows.  Shouldn't we assume the solution to \ref{12} is in $C^4(\Omega)$, then find the weak form?}
For convenience we restate the problem: find $u \in H^2(\Omega)$ such that
\begin{equation}\label{12}
	B\nabla\cdot\nabla\cdot\big(\nabla\nabla u+q^2\boldsymbol{T}u\big)+Bq^2\nabla\nabla u:\boldsymbol{T}+(Bq^4\boldsymbol{T}:\boldsymbol{T}+m)u=f,  
\end{equation}
Given a test function $ \phi\in H^2(\Omega)$, integration by parts gives
\begin{equation}\label{cc}
	\int_{\Omega}f\phi=a(u,\phi) + B\int_{\partial\Omega}\phi\nabla\cdot(\nabla\nabla u+q^2\boldsymbol{T}u)\cdot\vec{n}-B\int_{\partial\Omega}\nabla \phi \cdot (\nabla\nabla u+q^2\boldsymbol{T}u)\cdot\vec{n}
\end{equation}
where the bilinear form, $a$, is given by
\begin{eqnarray}
	a(u,\phi)&=& B\int_{\Omega}\nabla\nabla u:\nabla\nabla \phi+Bq^2\int_{\Omega}\boldsymbol{T}u:\nabla\nabla \phi+Bq^2\int_{\Omega}\nabla\nabla u:\boldsymbol{T}\phi\notag\\
	&+&\int_{\Omega}(Bq^4\boldsymbol{T}:\boldsymbol{T}+m)u\phi.\label{eq11}
\end{eqnarray}
From~\eqref{cc}, we identify that two boundary conditions are required on any segment of $\partial\Omega$, and that certain boundary conditions on $u$ arise naturally from the variational formulation. Consequently, we write $\partial \Omega=\Gamma_{0,2}\cup\Gamma_{0,1}\cup\Gamma_{3,2}\cup\Gamma_{3,1}$ with $\Gamma_{i,j}\cap\Gamma_{k,\ell} = \emptyset$ for $(i,j)\neq (k,\ell)$, and specify
\begin{alignat}{3}
	u&=0,\quad& (\nabla\nabla u+q^2\boldsymbol{T}u)\cdot\vec{n}&=\vec{0}, & \quad\text{on}&\ \Gamma_{0,2},\label{55}\\    
	u&=0,&\nabla u&=\vec{0}, &\text{on}&\ \Gamma_{0,1},\label{!!}\\
	\nabla\cdot(\nabla\nabla u+q^2\boldsymbol{T}u)\cdot\vec{n}&={0},& (\nabla\nabla u+q^2\boldsymbol{T}u)\cdot\vec{n}&=\vec{0},& \text{on}&\ \Gamma_{3,2},\label{!!!}\\
	\nabla\cdot(\nabla\nabla u+q^2\boldsymbol{T}u)\cdot\vec{n}&={0}, & \nabla u&=\vec{0}, & \text{on}&\ \Gamma_{3,1}\label{!!!!}.
\end{alignat}
The above decomposition of $\partial\Omega$ is into disjoint sets, with subscripts indicating the orders of the two boundary conditions imposed. We also use the notation $\Gamma_0 = \Gamma_{0,2}\cup\Gamma_{0,1}$ (that part of $\partial\Omega$ on which the zeroth-order boundary condition $u=0$ is imposed), and similarly we define $\Gamma_1 = \Gamma_{0,1} \cup \Gamma_{3,1}$, $\Gamma_2 = \Gamma_{0,2} \cup \Gamma_{3,2}$, and $\Gamma_3 = \Gamma_{3,2}\cup\Gamma_{3,1}$.
As typical, we consider homogeneous boundary conditions here, so that the boundary integrals in~\eqref{cc} vanish, but the results hold true for inhomogeneous boundary conditions, \revise{assuming that there exists a function $\delta \in H^4(\Omega)$ such that $u=\delta \ \text{on}\ \Gamma_{0}$, $\nabla u = \nabla \delta \ \text{on} \ \Gamma_1$, $(\nabla\nabla u+q^2\boldsymbol{T}u)\cdot\vec{n}=(\nabla\nabla \delta+q^2\boldsymbol{T}\delta)\cdot\vec{n} \ \text{on}\ \Gamma_2$, and $\nabla\cdot(\nabla\nabla u+q^2\boldsymbol{T}u)\cdot\vec{n}=\nabla\cdot(\nabla\nabla \delta+q^2\boldsymbol{T}\delta)\cdot\vec{n} \ \text{on} \ \Gamma_3$. The inhomogeneous problem can therefore be transformed into a homogeneous one in a manner similar to \cite[Section 1.5]{MR851383}}.  Consequently, we define
\[
\mathscr{V} = \left\{ u\in H^2(\Omega) : u = 0 \text{ on } \Gamma_{0} \text{ and } \nabla u = \vec{0} \text{ on } \Gamma_{1}\right\}.
\]
Because of the dependence on $q$ in the bilinear form, we analyze the
problem in a $q$-dependent norm on $\mathscr{V}$, given by
\begin{equation}\label{H^2-norm}
	\|u\|^2_{2,q}=q^{-4}\|\nabla\nabla u\|_0^2+q^{-4}\|\nabla u\|_0^2+\|u\|_0^2.
\end{equation}
\begin{assumption}
	In all results that follow, we assume $q\geq 1$ and that there exists an $\mathcal{O}(1)$ constant, $s$ \revise{that is independent of $q$}, such that $sq^{-4}\leq B\leq 1$, as this is the case of interest~\cite{xia2021structural}.
\end{assumption}

\begin{theorem}\label{thm3}
	Given $f\in L^2(\Omega)$, the variational problem to find $u\in \mathscr{V}$ such that 
	\begin{equation}\label{well-posed}
		a(u,\phi) = \int_\Omega f\phi \text{ for all } \phi \in \mathscr{V}
	\end{equation}  
	is well-posed.
\end{theorem}
\begin{proof}
	By the Lax--Milgram Theorem, this variational problem has a unique solution if the bilinear form $a$ is coercive and continuous on $\mathscr{V}$ in the $\|\cdot\|_{2,q}$ norm,
	and the associated linear form is continuous.  The assumption that $f\in L^2(\Omega)$ is sufficient to guarantee that the linear form is continuous.

        To prove continuity of $a$, we have
	\begin{eqnarray*}
		a(u,\phi)&=&  B\int_{\Omega}\nabla\nabla u:\nabla\nabla \phi+Bq^2\int_{\Omega}\nabla\nabla u:\boldsymbol{T}\phi+Bq^2\int_{\Omega}\nabla\nabla \phi:\boldsymbol{T}u\\
		&+&\int_{\Omega}(Bq^4\boldsymbol{T}:\boldsymbol{T}+m)u\phi\\
		& \leq& B\|\nabla\nabla u\|_{0}\|\nabla\nabla\phi\|_{0}+Bq^2\left(\|\nabla\nabla u\|_0\|\boldsymbol{T}\phi\|_0\right)+Bq^2\left(\|\nabla\nabla \phi\|_0\|\boldsymbol{T}u\|_0\right)\\
		&+&(Bq^4\mu_1+m)\|u\|_0\|\phi\|_0\\
		&\leq&Bq^4\|u\|_{2,q}\|\phi\|_{2,q}+2B\sqrt{\mu_1}q^4\|u\|_{2,q}\|\phi\|_{2,q}+(Bq^4\mu_1+m)\|u\|_{2,q}\|\phi\|_{2,q}.
	\end{eqnarray*} 
	This gives a $q$-dependent continuity bound, with
	$a(u,\phi)\leq CBq^4\|u\|_{2,q}\|\phi\|_{2,q}$, where $C=1+2\sqrt{\mu_1}+\mu_1+{m}/{s}$.
	
	To prove coercivity, we first observe that, for any $0 < C_1 < 1$,
	\begin{eqnarray}
		a(u,u)&=& B\|\nabla\nabla u\|_0^2+2Bq^2\int_{\Omega}\nabla\nabla u:\boldsymbol{T}u+\int_{\Omega}(Bq^4\boldsymbol{T}:\boldsymbol{T}+m)u^2\notag\\
		&\geq&B(1-C_1)\|\nabla\nabla u\|_0^2+Bq^4\left(1-\frac{1}{C_1}\right)\|\boldsymbol{T}u \|_0^2+m\|u\|_0^2 \notag\\
		&\geq&B\left(1-C_1\right)\|\nabla\nabla u\|_0^2+\left(m+Bq^4\mu_1\left(1-\frac{1}{C_1}\right)\right)\label{key2}\|u\|_0^2.
	\end{eqnarray}
	Note that since $0<C_1<1$, $1-\frac{1}{C_1} < 0$, so we get a lower bound for the expression on the second line by using the upper bound $\|\boldsymbol{T}u\|_0^2 \leq \mu_1 \|u\|_0^2$.
	Let $C_1=\frac{B\mu_1q^4}{B\mu_1q^4+C_2}$ for any constant $0<C_2<m$. Then, $1-C_1=\frac{C_2}{Bq^4\mu_1+C_2}$, and $1-\frac{1}{C_1}=-\frac{C_2}{Bq^4\mu_1}$, giving
	\begin{eqnarray}
		a(u,v)&\geq&B\left(\frac{C_2}{B\mu_1q^4+C_2}\right)\|\nabla\nabla u\|_0^2+\left(m-C_2\right)\|u\|_0^2.
	\end{eqnarray}
	Now, since $sq^{-4}\leq B\leq 1$, $B\left(\frac{C_2}{B\mu_1q^4+C_2}\right)=\left(\frac{C_2}{\mu_1q^4+\frac{C_2}{B}}\right)$ is an $\mathcal{O}(q^{-4})$ constant. That is, there exists a constant $C_3$ such that
	\begin{eqnarray}\label{18}
		a(u,u)\geq C_3\left(q^{-4}\|\nabla\nabla u\|_0^2+\|u\|_0^2\right).
	\end{eqnarray} 
	We then use Lemma \ref{Lemma-poinc}, with $j=2$, and $\xi(u)=\|u\|_0$. That is, there exists a constant $C_4$ such that $\|\nabla u\|_0^2\leq C_4\left(\|\nabla\nabla u\|_0^2+\|u\|_0^2\right)$, giving
	\begin{eqnarray*}
		a(u,u)&\geq& C_3q^{-4}\left(\|\nabla\nabla u\|_0^2+q^4\|u\|_0^2\right)\\
		&\geq& \frac{C_3q^{-4}}{2}\left(\|\nabla\nabla u\|_0^2+\|u\|_0^2\right)+\frac{C_3}{2}\left(q^{-4}\|\nabla\nabla u\|_0^2+\|u\|_0^2\right)\\
		&\geq&\left(\frac{C_3}{2C_4}q^{-4}\|\nabla u\|_0^2+\frac{C_3}{2}\left(q^{-4}\|\nabla\nabla u\|_0^2+\|u\|_0^2\right)\right),
	\end{eqnarray*} 
	As a result, $a(u,u)\geq
	\left(\min\{\frac{C_3}{2C_4},\frac{C_3}{2}\}\right)\|u\|_{2,q}^2$,
	giving an $\mathcal{O}(1)$ coercivity constant.

\end{proof}
%\spm{Need to be clearer in this second case about scaling of coercivity constant.  Is it still $q^{-4}$ overall?  This isn't clear to me.}
%\begin{corollary}\label{cor:smooth}
%	Given $f\in L^2(\Omega)$ and $\Omega$ be a polygonal domain whose interior angles, $\omega$ are either $\omega\leq \frac{\pi}{3}$ or $\omega=\frac{\pi}{2}$. Then, the solution, $u$, of Problem \eqref{eq:mod_biharmonic} belongs to $H^4(\Omega)$.
%\end{corollary}
%\begin{proof}
%	This is a direct result of \cite[Appendix A]{MR2142191}.
%\end{proof}
The mixed finite-element formulation presented below relies on a reformulation of \eqref{12} to a lower-order system. We introduce
$\vec{v}=\nabla u$ and $\vec{\alpha}=B\nabla\cdot(\nabla \vec{v} +q^2\boldsymbol{T}u)$. Then, Equation~\eqref{12} is equivalent to the system of equations
\begin{eqnarray}
	\nabla\cdot\vec{\alpha}+Bq^2\nabla\vec{v}:\boldsymbol{T}+(Bq^4\boldsymbol{T}:\boldsymbol{T}+m)u&=&f\label{1},\\
	\vec{\alpha}-B\Delta\vec{v}-Bq^2\nabla\cdot (\boldsymbol{T}u)\label{2}&=&0,\\
	\left(\vec{v}-\nabla u\right)&=&0\label{3}.
\end{eqnarray}
To convert this system to weak form, we
multiply \eqref{1}, \eqref{2}, and \eqref{3} by the test functions $\phi, \vec{\psi}, \text{and} \ \vec{\beta}$, respectively, and integrate by parts. 
This yields the weak form of finding $(u,\vec{v},\vec{\alpha}) \in L^2(\Omega)\times V\times H_{\Gamma_{3}}(\mathrm{div};\Omega)$ such that 
\begin{align}
	\int_{\Omega}\nabla\cdot\vec{\alpha}\phi+Bq^2\nabla\vec{v}:\boldsymbol{T}\phi+(Bq^4\boldsymbol{T}:\boldsymbol{T}+m)u\phi&=\int_{\Omega}f\phi,\label{5}\\
	\int_{\Omega}\vec{\alpha}\cdot\vec{\psi}+B\nabla\vec{v}:\nabla\vec{\psi}+Bq^2\boldsymbol{T}u:\nabla\vec{\psi}-B\int_{\revise{\Gamma_{2}}}\vec{\psi}\cdot(\nabla\vec{v}+q^2\boldsymbol{T}u)\cdot\vec{n}&=0\label{6},\\
	\int_{\Omega}\vec{\beta}\cdot\vec{v}+\int_{\Omega}u\nabla\cdot\vec{\beta}-\int_{\revise{\Gamma_{0}}}u\vec{\beta}\cdot\vec{n}&=0,\label{7}
\end{align}
%% Note that the BCs \eqref{55}-\eqref{!!!!} are now the same as the following 
%% \begin{eqnarray}
%% u&=&0,\ \  (\nabla\nabla u+q^2Tu)\cdot\vec{n}=\vec{0},\ \ \text{on}\ \Gamma_{0,2},\label{555}\\    
%% u&=&0,\ \ \vec{v}=\vec{0},\ \ \ \ \ \ \ \ \ \ \ \ \ \ \ \  \ \ \text{on}\ \Gamma_{0,1}, \ \label{!!5}\\
%% \vec{\alpha}\cdot\vec{n}&=&0, \ \ (\nabla\nabla u+q^2Tu)\cdot\vec{n}=\vec{0},\  \ \ \text{on}\ \Gamma_{3,2},\label{!!!5}\\
%% \vec{\alpha}\cdot\vec{n}&=&0, \ \ \vec{v}=\vec{0},\ \ \ \ \ \ \ \ \ \ \ \ \ \ \ \ \ \ \ \text{on}\ \Gamma_{3,1}\label{!!!!5}.
%% \end{eqnarray}	
$\forall (\phi,\vec{\psi}, \vec{\beta})\in L^2(\Omega)\times V\times \revise{H_{\Gamma_{3}}(\text{div};\Omega)}$, where $V=\left\{\vec{v}\in[H^{1}_{\Gamma_{1}}(\Omega)]^d,\ \vec{v}\times\vec{n}=0\ \text{on}\ \Gamma_{0,2}  \right\}$. Note that since $u=0$ on $\Gamma_{0,2}$, we are free to require the boundary condition $\vec{v}\times \vec{n}=0$ on $\Gamma_{0,2}$, which will be needed below.  Equivalently, the system \eqref{5}-\eqref{7} can be written as a saddle-point system, to find $(u,\vec{v},\vec{\alpha}) \in L^2(\Omega)\times V\times H_{\Gamma_{3}}(\mathrm{div};\Omega)$ such that  
\begin{eqnarray}
	\mathcal{A}\left((u,\vec{v}), (\phi,\vec{\psi})\right)+b\left(\vec{\alpha},(\phi,\vec{\psi})\right)&=&F(\phi), \ \ \forall (\phi, \vec{\psi})\in L^2(\Omega)\times V\label{keyy1}\\
	b\left(\vec{\beta}, (u,\vec{v})\right)&=&0, \ \ \forall \vec{\beta}\in H_{\Gamma_{3}}(\mathrm{div};\Omega)\label{keyy2}
\end{eqnarray}
where
\begin{eqnarray}
	\mathcal{A}\left((u,\vec{v}),(\phi,\vec{\psi})\right)&=&B\int_{\Omega}\nabla\vec{v}:\nabla\vec{\psi}+Bq^2\int_{\Omega}\nabla\vec{v}:\boldsymbol{T}\phi+Bq^2\int_{\Omega}\nabla\vec{\psi}:\boldsymbol{T}u\notag\\
	&+&\int_{\Omega}(Bq^4\boldsymbol{T}:\boldsymbol{T}+m)u\phi,\label{keyr}\\
	b\left(\vec{\alpha},(u,\vec{v})\right)&=&\int_{\Omega}\vec{\alpha}\cdot\vec{v}+\int_{\Omega}u\nabla\cdot\vec{\alpha},\label{key6}\\
	F(\phi)&=&\int_{\Omega}f\phi.\label{RHS_func}
\end{eqnarray}
Here, we see that while $\alpha$ is defined in terms of $u$ above, it serves as a Lagrange multiplier in the saddle-point form, weakly enforcing that $\vec{v} = \nabla u$. 
\revise{\begin{remark}
	Proving the inf-sup condition for Problem \eqref{keyy1}-\eqref{keyy2} when $\partial \Omega = \Gamma_{3,1}$ requires that for all $p\in H^1_0(\Omega)$, there exists $\vec{\psi}\in [H_0^1(\Omega)]^2$ such that 
	\begin{eqnarray*}
		\int_{\Omega}p\nabla\cdot \vec{\psi} \geq C_1\|p\|^2_0, \ and \ \|\vec{\psi}\|_1\leq C_2\|p\|_0.
	\end{eqnarray*} 
	However, to our knowledge this requires additional assumptions on $p$ and $\vec{\psi}$, such as when $p$ has zero mean or $\vec{\psi}$ is free on a portion of $\partial\Omega$. Therefore, we exclude the case when $\partial \Omega= \Gamma_{3,1}$ from Theorem~\ref{thm:continuum_mixed}.  
\end{remark}}
In the next theorem, we prove well-posedness of \eqref{keyy1}-\eqref{keyy2} in the two-dimensional case.  Here, we make use of a similarly weighted product norm for $(u,\vec{v})$, defined as
\begin{equation}\label{eq:wtd_product_norm}
	\|(u,\vec{v})\|_{0,q,1}^2 = \|u\|_0^2 + q^{-4}\|\vec{v}\|_1^2.
\end{equation}

\begin{theorem}\label{thm:continuum_mixed}
	Assume that $\Omega \subset \mathbb{R}^2$ with $\partial\Omega \neq \Gamma_{3,1}$. Given $f\in L^2(\Omega)$, \eqref{keyy1}-\eqref{keyy2} is well-posed using the product norm defined in~\eqref{eq:wtd_product_norm} for $(u,\vec{v})$ and the $H_\mathrm{Div}$ norm induced by~\eqref{eq:Div_norm} for the Helmholtz decomposition, $\vec{\alpha}=\nabla\phi+\nabla\times p$, given in Lemma~\ref{lemma-helm-cont}. 
	%If, in addition, the solution, $u$, of Problem \eqref{well-posed} is smooth enough then,
	%\begin{equation}
	%\vec{v}=\nabla  u, \ \ \ \text{and} \ \ \ \vec{\alpha}=\nabla\cdot\left(\nabla\nabla u+q^2Tu\right).
	%\end{equation}  
\end{theorem}
\begin{proof}
	By Brezzi's theory, proving well-posedness requires establishing the coercivity and continuity of $\mathcal{A}$, an inf-sup condition on $b$, and the continuity of $b$. As above, $F(\phi)$ is clearly continuous.
	
	We first consider the two continuity bounds.  The continuity bound for $\mathcal{A}$ gives an $\mathcal{O}(Bq^4)$ continuity constant, established similarly to the proof in Theorem \ref{thm3}. The Helmholtz decomposition in Lemma~\ref{lemma-helm-cont} and integration by parts are needed to show continuity of $b$. Let $\vec{\alpha}\in H_{\Gamma_{3}}(\mathrm{div};\Omega)$, writing $\vec{\alpha}=\nabla\phi+\nabla\times p$, where $\phi\in H^1_{\Gamma_{0}}(\Omega)$ and $p\in H^1_{ \Gamma_{3}}(\Omega)$. Then,   
	\begin{eqnarray*}
		b\left(\vec{\alpha},(u,\vec{v})\right)&=&\int_{\Omega}\vec{\alpha}\cdot\vec{v}+\int_{\Omega}u\nabla\cdot\vec{\alpha}=\int_{\Omega}\left(\nabla \phi+\nabla\times p\right)\cdot\vec{v}+\int_{\Omega}u\nabla\cdot\vec{\alpha} \\
		&=&\int_{\Omega}\nabla\phi\cdot\vec{v}+\int_{\Omega}p\nabla\times \vec{v}+\int_{\Omega}u\nabla\cdot\vec{\alpha}\\
		&\leq&q^4\left(q^{-2}\|\nabla\phi\|_0\right)\left(q^{-2}\|\vec{v}\|_0\right)+q^4\left(q^{-2}\|p\|_0\right)\left(q^{-2}\|\nabla\times\vec{v}\|_0\right)\\
		&+&q^2\|u\|_0\left(q^{-2}\|\nabla\cdot\vec{\alpha}\|_0\right)\\
		&\leq& q^4\left(1+\sqrt{2}+q^{-2}\right)\|(u,\vec{v})\|_{0,q,1}\|\vec{\alpha}\|_{\mathrm{Div}},
	\end{eqnarray*}
	where we have used the fact that $\|\nabla\times \vec{v}\|_0\leq \sqrt{2}\|\nabla\vec{v}\|_0$. We note that the ``extra'' boundary condition imposed on $\vec{v}\in V$ is needed here to ensure the boundary integral from integration-by-parts is identically zero.
	
	We now show that the bilinear form, $\mathcal{A}$, is coercive on  $$\Lambda=\{(u,\vec{v})\in L^2(\Omega)\times V\ : \ b\left(\vec{\alpha},(u,\vec{v}) \right)=0, \ \forall \vec{\alpha}\in H_{\Gamma_{3}}(\mathrm{div};\Omega) \}.$$ 
	Similarly to the proof of Inequality~\eqref{18}, there exists $C_1>0$ such that
	\begin{eqnarray}\label{eq29}
		\mathcal{A}\big((u ,\vec{v} ), (u, \vec{v})\big)\geq C_1\left(q^{-4}\|\nabla\vec{v}\|_0^2+\|u\|_0^2\right) \quad \forall\ (u, \vec{v}) \in \Lambda.
	\end{eqnarray} 
	If $(u,\vec{v}) \in \Lambda$, then $b(\beta_i, (u,\vec{v})) = 0$ for the particular choices $\beta_1 = [S, 0]^\top$ and $\beta_2 = [0, S]^\top$, where $S$ is the function defined in Lemma~\ref{Lemma-coercivity-poinc}. In other words, $\int_{\Omega}\vec{\beta}_i\cdot \vec{v}+\int_{\Omega}u\nabla\cdot\vec{\beta}_i=0$ for $i=1,2$, which gives
	\[\left|\int_{\Omega}Sv_1\right|=\left|\int_{\Omega}S_xu\right|\leq \|S_x\|_0\|u\|_0\leq C_2\|u\|_0,\]
	and		
	\[ \left|\int_{\Omega}Sv_2\right|=\left|\int_{\Omega}S_yu\right|\leq \|S_y\|_0\|u\|_0\leq C_2\|u\|_0.\]

	As $S$ has a positive mean, we can apply the Poincar\'e inequality in the form given in Lemma~\ref{Lemma-poinc}, which leads to the fact
	\begin{eqnarray}\label{eq32}
		\|\vec{v}\|^2_0\leq C_3\left(\|u\|_0^2+\|\nabla\vec{v}\|_0^2\right). 
	\end{eqnarray} 
	Combining this with Inequality \eqref{eq29} then gives
	\begin{eqnarray*}
		\mathcal{A}\left((u,\vec{v}), (u,\vec{v})\right)\geq C_4\|(u,\vec{v})\|_{0,q,1}.
	\end{eqnarray*}
	That is, coercivity of $\mathcal{A}$ holds with an $\mathcal{O}(1)$ constant.
	
	Next, we prove the required inf-sup condition, of the form
	\begin{eqnarray*}
		I=\sup_{(u,\vec{v})\in L^2\times V}\frac{\int_{\Omega}\vec{\alpha}\cdot\vec{v}+\int_{\Omega}u\nabla\cdot\vec{\alpha}}{\sqrt{\|u\|_0^2+q^{-4}\|\vec{v}\|_1^2
		}}\geq Cq^{2}\|\vec{\alpha} \|_{\mathrm{Div}},\ \ \forall \vec{\alpha}\in H_{\Gamma_{3}}(\mathrm{div};\Omega).\label{infsup_continuum}
	\end{eqnarray*}
	In our mixed formulation, the boundary condition on $u$  is weakly enforced. Therefore, given $\vec{\alpha} = \nabla\eta + \nabla\times p$, where $p\in H_{\Gamma_{3}}^1(\Omega)$ and $\eta\in H^1_{\Gamma_{0}}$, we may choose $u=C_1\left(\nabla\cdot\vec{\alpha}-\eta\right)$, for a positive constant, $C_1$, to be specified later. This choice of $u$ might not be zero on $\Gamma_{0}$. Note that $$\|u\|_0=C_1\|\nabla\cdot\vec{\alpha}-\eta\|_0\leq C_1\left(\|\nabla\cdot \vec{\alpha}\|_0+\|\eta\|_0\right), \ \text{and}\ \int_{\Omega}u\nabla\cdot\vec{\alpha}=C_1\left(\|\nabla\cdot\vec{\alpha}\|_0^2+\|\nabla\eta\|_0^2\right).$$
	If $\Gamma_{3,1}$ is a proper subset of $\partial \Omega$, then the inf-sup condition of \cite[Lemma 2.7]{MR3660776} establishes that for all $p\in H^1_{\Gamma_{3}}(\Omega)$, there exists $\vec{\psi}\in [H^1_{\Gamma_{0,1}\cup\Gamma_{3,2}}(\Omega)]^2$ such that $\int_{\Omega}p\nabla\cdot\vec{\psi}\geq\|p\|_0^2$, and $\|\vec{\psi}\|^2_1\leq C_2\|\vec{p}\|^2_0$. Note that 
	\begin{itemize}
		\item If $\partial\Omega=\Gamma_{0}$, then $p$ is a zero-mean function, and such a $\vec{\psi}$ exists by the standard inf-sup condition for the Stokes problem with Dirichlet boundary conditions.
		\item Otherwise (so long as $\partial\Omega \neq \Gamma_{3,1}$), such a $\vec{\psi}\in V$ exists with $\vec{\psi}=\vec{0}$ on $\Gamma_{0}\cup\Gamma_{3,1}$ and $\vec{\psi}\times\vec{n}=0$ on $\Gamma_{3,2}$, following \cite[Lemma 2.7]{MR3660776}.
	\end{itemize}
	To establish the inf-sup condition needed here, we then choose $\vec{v} =[\psi_2, -\psi_1]^T$ which, by construction, belongs to $V$, giving $\nabla \cdot\vec{\psi} =\nabla\times \vec{v} $ and $\|\vec{\psi} \|_1^2=\|\vec{v}\|_1^2$.  With this,
	\begin{align*}
		I&\geq\sup_{(u,\vec{v})\in L^2\times V}\frac{\int_{\Omega}\nabla\eta\cdot\vec{v}+\int_{\Omega}p\nabla\times \vec{v}+\int_{\Omega}u\nabla\cdot\vec{\alpha}}{\sqrt{\|u\|_0^2+q^{-4}\|\vec{v}\|_1^2}}\\
		&\geq \sup_{(u,\vec{v})\in L^2\times V}\frac{\int_{\Omega}\nabla\eta\cdot\vec{v}+\int_{\Omega}p\nabla\times \vec{v}+\int_{\Omega}u\nabla\cdot\vec{\alpha}}{\sqrt{\|u\|_0^2+\|\vec{v}\|_1^2}}\\
		&\geq\frac{\|p\|_0^2+C_1\|\nabla\cdot\vec{\alpha}\|_0^2+C_1\|\nabla \eta\|_0^2-\frac{C_3}{2}\|\nabla\eta\|_0^2-\frac{1}{2C_3}\|\vec{v}\|_0^2}{\sqrt{C_2\|p\|_0^2+2C_1^2\|\eta\|_0^2+2C_1^2\|\nabla\cdot\vec{\alpha}\|_0^2}}\\&\geq\frac{\left(1-\frac{C_2}{2C_3}\right)\|p\|_0^2+C_1\|\nabla\cdot\vec{\alpha}\|_0^2+\left(C_1-\frac{C_3}{2}\right)\|\nabla \eta\|_0^2}{\sqrt{C_2\|p\|_0^2+2C_1^2\|\eta\|_0^2+2C_1^2\|\nabla\cdot\vec{\alpha}\|_0^2}},
	\end{align*} 
	where we use the facts that $\int_{\Omega}p\nabla\times\vec{v}\geq \|p\|_0^2$, $\|\vec{v}\|_1^2\leq C_2\|p\|_0^2$, and $\left|\int_{\Omega}\nabla\eta\cdot\vec{v}\right|\leq \frac{C_3}{2}\|\nabla\eta\|_0^2+\frac{1}{2C_3}\|\vec{v}\|^2_0$. Choose $C_1>\frac{C_3}{2}>\frac{C_2}{4}$. Note that $\|\eta\|_0^2\leq C_4\|\nabla\eta\|_0^2$ for some positive $C_4$ by the Poincar\'{e} inequality. Thus,
	\begin{eqnarray*}
		I
		&\geq&\frac{\left(1-\frac{C_2}{2C_3}\right)\|p\|_0^2+C_1\|\nabla\cdot\vec{\alpha}\|_0^2+\left(C_1-\frac{C_3}{2}\right)\|\nabla \eta\|_0^2}{\sqrt{C_2\|p\|_0^2+2C_1^2C_4\|\nabla\eta\|_0^2+2C_1^2\|\nabla\cdot\vec{\alpha}\|_0^2}}\\
		&\geq& C\left(\|p\|_0+\|\nabla \eta\|_0+\|\nabla\cdot\vec{\alpha}\|_0\right)\geq Cq^{2}\|\vec{\alpha}\|_{\mathrm{Div}}.
	\end{eqnarray*} 
	As a result, the inf-sup condition holds true and the three-field formulation is well-posed.
	% Finally, if the solution given that the choices $\vec{\beta}= Q\left(\vec{v}-\nabla u\right)$ and $\left(\phi,\vec{\psi}\right)=\left(0,Q\left(\nabla\cdot\nabla\nabla u + q^2Tu\right)\right)$ in Equations \eqref{keyy2} and \eqref{keyy1} respectively lead the fact that $\vec{v}=\nabla u$ and $\vec{\alpha}=\nabla\cdot\left(\nabla\nabla u+q^2Tu\right),$ where $Q$ is a function in $H^1_0({\Omega})$ that is positive over $\Omega$.
	%\spm{We say this holds for ``smooth enough'' $u$ in the statement of the theorem, so need to quantify what that means here.}
\end{proof}

%\spm{If we need them for later, it's important to note that the continuity constant for $a$ is $\mathcal{O}(q^4)$, a coercivity constant of $\mathcal{O}(q^{-4})$, and an inf-sup constant that is $\mathcal{O}(1)$.}
\begin{remark}
	The dependence on $q$ of the continuity, coercivity, and inf-sup constants in the proofs above can lead to pessimistic error bounds for the finite-element methods developed below.  While it is tempting to try and prove convergence using other weighted $H^2$ norms or $L^2\times H^1$ product norms, we are unaware of a simple weighting of the terms in such a norm that leads to $\mathcal{O}(1)$ continuity and coercivity constants in the weighted norms.  We note that a common use case of these results will be when $Bq^4$ is an $\mathcal{O}(1)$ constant, in which case it is only the inf-sup constant and continuity bounds for $b$ that are suboptimal.
\end{remark}

%\spm{Seems like we need an assumption on the regularity of $T$ in order to state this}
\begin{corollary}\label{Unique}
	Suppose $\Omega\subset \mathbb{R}^2$ has $\partial\Omega=\left\{\cup_{i=1}^{N_1}\Gamma^{i}\right\}\cup\left\{\cup_{i=1}^{N_2}\bar{\Gamma}^{i}\right\}$, for $\Gamma^{i}=(x,a_ix+b_i)$, and $\bar{\Gamma}^{i}=(c_iy+d_i,y)$ where $N_1$ and $N_2$ are positive integers. Let $u$ and $\left(\bar{u},\vec{v},\vec{\alpha}\right)$ be the solutions of Problems \eqref{well-posed} and \eqref{keyy1}-\eqref{keyy2}, respectively. Further, assume $u$ and $\bar{u}$ are in $H^t(\Omega)$, $t\geq 4$, and $\boldsymbol{T}\in\boldsymbol{C}^{t-2}(\Omega)$, where $\boldsymbol{C}^m(\Omega)$ is the space of $2\times 2$ tensors with each component in $C^m(\Omega)$. Then $u$ and $\left(\bar{u},\vec{v},\vec{\alpha}\right)$ are equivalent in the sense that $u=\bar{u}, \ \vec{v}=\nabla u$, and $\vec{\alpha}=\nabla\cdot\left(\nabla\nabla u+q^2\boldsymbol{T}u\right)$. 
\end{corollary}
\begin{proof}
	First, let $u\in H^t(\Omega)$, $t\geq 4$, be the solution of Problem \eqref{well-posed}. By direct calculation, $(u,\vec{v},\vec{\alpha})$ is a solution to Problem \eqref{keyy1}-\eqref{keyy2}, which is unique by Theorem \ref{thm:continuum_mixed}. Note that $u\in H^t(\Omega)$ is sufficient to guarantee that $\vec{v}\in \left(H^{1}(\Omega)\right)^2$ and $\vec{\alpha} \in H(\mathrm{div};\Omega)$.
	
	Conversely, let $\left(\bar{u},\vec{v},\vec{\alpha}\right)$ be the solution of \eqref{keyy1}-\eqref{keyy2} with $\bar{u}\in H^t(\Omega)$. Let
	\[
	D(x,y)= \Pi_{i=1}^{N_1}\left(y-a_ix-b_i\right)^2\Pi_{i=1}^{N_2}\left(x-c_iy-d_i\right)^2,
	\]
	so that $D \in C^{\infty}(\Omega)\cap H^1_0(\Omega)$ is positive in the interior of $\Omega$. Choosing $(\phi,\vec{\psi})=\left(0,DB\nabla\cdot\left(\nabla\nabla \bar{u}+q^2\boldsymbol{T}\bar{u}\right)\right)$ and $\vec{\beta}=D\left(\vec{v}-\nabla \bar{u}\right)$ in \eqref{keyy1}--\eqref{keyy2} and integrating by parts implies that $\vec{v}=\nabla \bar{u}$ and $\vec{\alpha}=B\nabla\cdot\left(\nabla\nabla \bar{u}+q^2\boldsymbol{T}\bar{u}\right)$:
	\begin{eqnarray*}
		\int_{\Omega}\vec{\beta}\cdot\left(\vec{v}-\nabla\bar{u}\right)=\int_{\Omega}D\left(\vec{v}-\nabla\bar{u}\right)\cdot\left(\vec{v}-\nabla\bar{u}\right)=0.
	\end{eqnarray*}	
	Note that $D$ is sufficiently smooth so that
	$\vec{\beta}=D\left(\vec{v}-\nabla\bar{u}\right)\in
	H_0(\mathrm{div};\Omega)\subset
	H_{\Gamma_{3}}(\mathrm{div};\Omega)$. As
	$D\left(\vec{v}-\nabla u\right)\cdot\left(\vec{v}-\nabla u\right)$ is
	non-negative over $\Omega$, this implies that $\vec{v}=\nabla \bar{u}$. Similarly, one can prove that $\vec{\alpha} = B\nabla\cdot\left(\nabla\nabla \bar{u}+q^2\boldsymbol{T}\bar{u}\right)$.
	As above, the regularity of $\bar{u}$ is necessary to ensure that $\vec{v}$ and $\vec{\alpha}$ (as well as $\vec{\psi}$ and $\vec{\beta}$ defined above) have the regularity to satisfy these equations.  With this value for $\vec{\alpha}$, taking $\vec{\psi}=\vec{0}$ and $\phi\in \mathscr{V} \subset L^2(\Omega)$ in \eqref{keyy1} leads to the fact that $\bar{u}$ is a solution of Problem \eqref{well-posed}. Thus, $\bar{u}=u$ by the uniqueness of the solution of Problem \eqref{well-posed}. 
\end{proof}
%\spm{Do we need to motivate the construction of $Q$ here?  Also, can we rename $Q$, to avoid possible ambiguity with $\boldsymbol{Q}$?}
%\spm{The bounding of $I$ here is a bit tricky.  It makes sense to me when I check it, but I wonder if we need to guide the reader a bit more...}

%\spm{I removed the assumption that $\nabla T : \nabla T \leq \mu_2$ in the theorems above, since we didn't use it, but left the general assumption in the start of the section.  We use this bound below...}

\section{Discrete Analysis}\label{sec:discrete}

We now consider three different discretizations of the smectic density equation.  First, in~\Cref{ssec:conforming}, we consider an $H^2$-conforming discretization based on Argyris elements for $\Omega \subset \mathbb{R}^2$. The method offers optimal convergence bounds, but its analogue for $\Omega \subset \mathbb{R}^3$ is very difficult to implement. 
Thus, in~\Cref{ssec:C0IP}, we consider a $C^0$-interior-penalty method, which allows for the use of continuous Lagrange elements of any order $k\geq 2$, in both two and three dimensions. Finally, in~\Cref{ssec:mixed}, we develop a discretization of the mixed formulation \eqref{keyy1}-\eqref{keyy2} along the lines proposed in~\cite{farrell2021new}, that also offers some anticipated advantages, especially in the construction of preconditioners (to be considered in future work).
\subsection{Conforming Methods}\label{ssec:conforming}
We first consider the case of $\Omega \subset\mathbb{R}^2$ with full Neumann boundary conditions, where $\partial\Omega=\Gamma_{3,2}$. While several choices of conforming elements are possible, we focus on Argyris elements, $\text{ARG}_5(\Omega, \mathcal{T}_h)$, which arise from choosing a basis for the 21 degrees of freedom for a fifth-order polynomial space, $CG_5(\Omega, T)$, on each triangle, $T$, in such a way as to ensure that the resulting space is $H^2$-conforming \cite{kirby2012common}. The weak form is to find $u_h\in \text{ARG}_5(\Omega, \mathcal{T}_h)$ such that \cite{MR0520174}

\begin{eqnarray}\label{eq30}
	a(u_h,\phi_h)=F(\phi_h), \ \ \forall \phi_h\in \text{ARG}_5(\Omega, \mathcal{T}_h),
\end{eqnarray} 
where $a$ is defined in \eqref{eq11} and $F(\phi_h)$ is defined in \eqref{RHS_func}.
\begin{corollary}\label{cor3}
	Let $f\in L^2(\Omega)$, and let $\{\mathcal{T}_h\}$ be a family of quasiuniform meshes of $\Omega$.  Problem \eqref{eq30} is well-posed for $\partial\Omega=\Gamma_{3,2}$. 
	Moreover, if $hq\leq 1$ and $u\in H^t(\Omega)$ for $3 \leq t \leq 6$ is the solution of Problem \eqref{eq11}, then 
	\begin{equation}\label{eq31}
		\|u-u_h\|_{2,q}\leq CBq^2h^{t-2}|u|_t.
	\end{equation}
\end{corollary}
\begin{proof}
	For $\partial\Omega=\Gamma_{3,2}$, the bilinear form $a(u_h,\phi_h)$ is symmetric, continuous and coercive and the linear form $F$ is continuous, as shown above. Since $u_h, \phi_h\in \text{ARG}_5(\Omega, \mathcal{T}_h)\subset H^2(\Omega)$, this is a conforming discretization and is well-posed following Theorem \ref{thm3} and the Lax--Milgram theorem. 
	Finally, C\'ea's lemma and standard bounds on the Argyris interpolation operator \cite{MR2322235} lead to the estimate in \eqref{eq31}, as
	\begin{eqnarray*}
		\|u-u_h\|_{2,q}&\leq& CBq^4\inf_{v_h\in\text{ARG}_{5}(\Omega,\mathcal{T}_h)}\left(q^{-2}\|\nabla\nabla(u-v_h)\|_0+q^{-2}\|\nabla(u-v_h)\|_0+\|u-v_h\|_0\right) \\
		&\leq&CBq^4\left(q^{-2}h^{t-2} + q^{-2}h^{t-1} + h^t\right)|u|_t\leq CBq^2h^{t-2}|u|_t.
	\end{eqnarray*}
\end{proof}
%\spm{I think you're being mildly pessimistic here - the real bound from the first inequality is something like
%  \[
%\|u-u_h\|_{2,q} \leq Cq^4\left(q^{-2}h^{t-2} + q^{-2}h^{t-1} + h^t\right)|u|_t.
%\]
%If I assume that $qh<1$, doesn't this give me an overall bound of $q^2h^{t-2}$?  Can we state this better?}
\begin{remark}
	We are naturally interested in how the error estimate above depends on $q$.  From the coercivity and continuity constants of Theorem \ref{thm3}, which scale as $\mathcal{O}(1)$ and $\mathcal{O}(Bq^4)$, respectively, we see that the quasioptimality constant scales like $\mathcal{O}(Bq^4)$.  When $Bq^4 = \mathcal{O}(1)$, as can be the case (see, e.g., \cite{pevnyi2014modeling}), this gives an error bound that $\|u-u_h\|_{2,q} \leq (q^{-2}h^{t-2})|u|_t$, which is optimal in $h$.  For larger values of $B$, we retain optimality in $h$, but see some degradation in $q$, as might be expected.
	Moreover, for the case of expected solutions to \eqref{eq:PSS} that behave like $e^{iq\vec{\nu}\cdot\vec{x}}$ (showing similar behaviour to the observed solutions of the generalized models in \cite{xia2021structural,pevnyi2014modeling}), we have $|u|_t \sim\mathcal{O}(q^t)$.  Considering the case of a strong solution with $u\in H^6(\Omega)$, this gives an error estimate that scales like $\mathcal{O}(Bq^{8})$, but with an $L^2$ error estimate of $h^{4}$.  Again, the value of $B$ strongly influences the impact of this scaling: when $Bq^4 = \mathcal{O}(1)$, then this necessitates choosing a mesh, $\mathcal{T}_h$, such that $hq < 1$, which is not an unreasonable requirement when $q$ is, itself, an $\mathcal{O}(1)$ constant.  If, however, $B = \mathcal{O}(1)$, the requirement on $\mathcal{T}_h$ is stricter, needing $h^4q^{8} < 1$ in order to guarantee convergence in the large $q$ limit.  While we are not interested in prohibitively large values of $q$ (as in \cite{pevnyi2014modeling, xia2021structural}, we consider $q \sim 40$), this recalls standard results in the literature on numerical approximation of solutions to the Helmholtz equation and the \textit{pollution effect} that leads to similar restrictions \cite{BAYLISS1985396, IHLENBURG19959}.
\end{remark}

\begin{remark}  The above result naturally extends to domains $\Omega\subset\mathbb{R}^3$ with three-dimensional $H^2(\Omega)$ conforming elements \cite{zhang2009}.
\end{remark}
Strongly enforcing essential boundary conditions using Argyris elements is well-known to be difficult~\cite{articl}, although extensions of Corollary \ref{cor3} would hold if we could do so.  Instead, if $\Gamma_{0,2}\cup\Gamma_{0,1}\cup\Gamma_{3,1} \neq \emptyset$, we enforce the essential boundary conditions weakly using Nitsche-type penalty methods. Then, the weak form is to find $u_h\in \text{ARG}_5(\Omega, \mathcal{T}_h)$ such that 
\begin{eqnarray}\label{32}
	A_h(u_h,\phi_h)=F(\phi_h), \ \ \forall \phi_h\in \text{ARG}_5(\Omega, \mathcal{T}_h),
\end{eqnarray} 
where
\begin{align*}
	A_h(u_h,\phi_h)=&a(u_h,\phi_h) +B \int_{\Gamma_{0}}\phi_h\nabla\cdot(\nabla\nabla u_h+q^2\boldsymbol{T}u_h)\cdot\vec{n}\\
	&- B\int_{\Gamma_{0}}u_h\nabla\cdot(\nabla\nabla \phi_h+q^2\boldsymbol{T}\phi_h)\cdot\vec{n}
	+\frac{1}{qh^3}\int_{\Gamma_{0}}u_h\phi_h\\
	&-B\int_{\Gamma_{1}}\nabla \phi_h \cdot (\nabla\nabla u_h+q^2\boldsymbol{T}u_h)\cdot\vec{n}+B\int_{\Gamma_{1}}\nabla u_h \cdot (\nabla\nabla \phi_h+q^2\boldsymbol{T}\phi_h)\cdot\vec{n}\\
	&+\frac{1}{q^3h}\int_{\Gamma_{1}}\nabla u_h\cdot \nabla \phi_h.
\end{align*} 
We prove coercivity and continuity of the bilinear form, $A_h$, in the strengthened $H^2(\Omega)$ norm, $\tnorm{\cdot}{2,q,h}$, defined as 
\begin{eqnarray}\label{33}
	\tnorm{u_h}{2,q,h}^2&=&\|u_h\|_{2,q}^2+\frac{1}{qh^3}\|u_h\|^2_{0,\Gamma_{0}}+ \frac{h^3}{q^7} \|\nabla\cdot\left(\nabla\nabla u_h+q^2\boldsymbol{T}u_h\right)\cdot\vec{n}\|^2_{0,\Gamma_{0}}\notag\\
	&&+\frac{1}{q^3h}\|\nabla u_h\|_{0,\Gamma_{1}}^2+\frac{h}{q^5}\|\left(\nabla\nabla u_h+q^2\boldsymbol{T}u_h\right)\cdot\vec{n}\|^2_{0,\Gamma_{1}}.
\end{eqnarray}
As shown below, the weights in the norm $\tnorm{\cdot}{2,q,h}$ allow us to prove optimal-in-$q$ error estimates for solutions in the space $H^t(\Omega), \ t\geq 4$.  Note that the choice of the Nitsche formulation in~\eqref{32} results in non-symmetric linear systems to be solved; while we do not focus on effective linear solvers here, this asymmetry may be seen as a downside of this approach.  However, we note that using a symmetric Nitsche formulation led to suboptimal error bounds in the analogous results to those that follow.
\begin{theorem}\label{Thm4}
	Let $f\in L^2(\Omega)$, and let $\{\mathcal{T}_h\}$ be a family of quasiuniform meshes of $\Omega$. Let $\boldsymbol{T}$ be given s.t. $|\boldsymbol{T}|^2=\boldsymbol{T}:\boldsymbol{T}\leq\mu_1$ and $|\nabla \boldsymbol{T}|^2=(\nabla \boldsymbol{T}):(\nabla \boldsymbol{T})\leq \mu_2$ pointwise on $\bar{\Omega}$. Then, there exist constants $C_1$ and $C_2$ such that for any $u_h,\phi_h\in ARG_5(\Omega, \mathcal{T}_h)$,
	\begin{align*}
		|A_h(u_h,\phi_h)| & \leq C_1Bq^4 \tnorm{u_h}{2,q,h}\tnorm{\phi_h}{2,q,h}, \\
		A_h(u_h,u_h) & \geq C_2 \tnorm{u_h}{2,q,h}^2.
	\end{align*}
	Moreover, Problem \eqref{32} is well-posed over $ARG_5(\Omega,\mathcal{T}_h)$.  
\end{theorem}
\begin{proof}
	The continuity of $A_h(u_h,\phi_h)$ and $F(\phi_h)$ for $u_h,\phi_h\in \text{ARG}_5(\Omega,\mathcal{T}_h)$ follow directly from the Cauchy--Schwarz inequality applied termwise, making use of the leeway offered by the suboptimal continuity of $a(u,\phi)$ with respect to the $\|\cdot\|_{2,q}$ norm.  For example, we use the bound
	\begin{align*}
		B\int_{\Gamma_{0}}\phi_h\nabla&\cdot(\nabla\nabla u_h+q^2\boldsymbol{T}u_h)\cdot\vec{n} \leq\\ &B\left(\left(\frac{q^3}{h^3}\right)^{1/2}\|\phi_h\|_{0,\Gamma_{0}}\right)\left(\left(\frac{h^3}{q^3}\right)^{1/2}\|\nabla\cdot(\nabla\nabla u_h+q^2\boldsymbol{T}u_h)\cdot\vec{n}\|_{0,\Gamma_{0}}\right) \\
		&= Bq^4\left(\left(\frac{1}{qh^3}\right)^{1/2}\|\phi_h\|_{0,\Gamma_{0}}\right)\left(\left(\frac{h^3}{q^7}\right)^{1/2}\|\nabla\cdot(\nabla\nabla u_h+q^2\boldsymbol{T}u_h)\cdot\vec{n}\|_{0,\Gamma_{0}}\right) \\
		&\leq Bq^4 \tnorm{\phi_h}{2,q,h}\tnorm{u_h}{2,q,h},
	\end{align*}
	with similar bounds for the other terms in $A_h(u_h,\phi_h)$ (recalling that $1 \leq \frac{1}{s}Bq^4$ for $\mathcal{O}(1)$ constant, $s$).  Therefore, we focus on the proof of coercivity.
	
	%\spm{I feel like we're abusing $C_3$ below, and that it isn't the same constant from one line to the next.  Is that right, or is there something else going on to make this all the same $C_3$?  Can we use a generic $C$ until we arrive at the last constant?  (The same for $C_4$, really...)}
	Trace theorems \cite[Section 2.1.3]{MR2431403}, and standard inverse estimates \cite[Theorem 4.5.11]{MR2373954} can be used to prove the following inequalities that are needed to show coercivity of $A_h$. For $\phi_h\in \text{ARG}_{5}(\Omega, \mathcal{T}_h)$, we have
	\begin{align}
		\|\nabla\cdot\nabla\nabla \phi_h\cdot\vec{n}\|^2_{0,\Gamma_{0,2}\cup\Gamma_{3,1}}&\leq \frac{C}{h}\|\nabla\cdot\nabla\nabla \phi_h\|_0^2,\label{34}\\
		\|\nabla\cdot\nabla\nabla \phi_h\|_0^2&\leq C\|\nabla\nabla\nabla \phi_h\|_0^2\leq \frac{C}{h^2}\|\nabla\nabla \phi_h\|_0^2,\\
		\|\nabla\cdot(\nabla\nabla \phi_h+q^2\boldsymbol{T}\phi_h)\cdot\vec{n}\|^2_{0,\Gamma_{0}}&\leq C\left(\frac{1}{h^3}\|\nabla\nabla \phi_h\|_0^2+q^4\left(\frac{\mu_2}{h}\|\phi_h\|_0^2+\frac{\mu_1}{h}\|\nabla\phi_h\|_0^2\right)\right)\notag\\
		&\leq C\left(\frac{1}{h^3}\|\nabla\nabla \phi_h\|_0^2+q^4\left(\frac{\mu_2}{h}\|\phi_h\|_0^2+\frac{\mu_1}{h^3}\|\phi_h\|_0^2\right)\right)\notag\\
		&\leq \frac{C_3q^4}{h^3}\|\phi_h\|^2_{2,q}
		.\label{Ineq37}
	\end{align}
	In addition, 
	\begin{equation}
		\|(\nabla\nabla\phi_h+q^2\boldsymbol{T}\phi_h)\cdot\vec{n}\|_{0,\Gamma_{1}}^2 \leq C\left(\frac{1}{h}\|\nabla\nabla \phi_h\|_0^2+\frac{q^4\mu_1}{h}\|\phi_h\|_0^2\right)\leq \frac{C_4q^4}{h}\|\phi_h\|^2_{2,q}\label{35-},
	\end{equation}
	where the constants, $C_3$ and $C_4$, will be used below.
	From Theorem \ref{thm3}, we have that $a(u_h,u_h)\geq C_5 \|u_h\|_{2,q}^2$, for $C_5>0$.  Using this, we have
	\begin{align*}
		A_h(u_h,u_h) &\geq C_5\|u_h\|_{2,q}^2
		+\frac{1}{qh^3}\|u_h\|^2_{0,\Gamma_{0}}+\frac{1}{q^3h}\|\nabla u_h\|^2_{\Gamma_{1}}\\
		& \geq \frac{C_5}{3}\left(\|u_h\|_{2,q}^2+\frac{h^3}{C_3q^4}\|\nabla\cdot(\nabla\nabla u_h+q^2\boldsymbol{T}u_h)\cdot\vec{n}\|^2_{0,\Gamma_{0}}\right)\\
		&+\frac{C_5h}{3C_4q^4}\|(\nabla\nabla u_h+q^2\boldsymbol{T}u_h)\cdot\vec{n}\|^2_{0,\Gamma_{1}} +\frac{1}{qh^3}\|u_h\|^2_{0,\Gamma_{0}}+\frac{1}{q^3h}\|\nabla u_h\|^2_{\Gamma_{1}}\\
		& \geq\frac{C_5}{3} \|u_h\|_{2,q}^2+ \frac{C_5h^3}{3C_3q^7}\|\nabla\cdot(\nabla\nabla u_h+q^2\boldsymbol{T}u_h)\cdot\vec{n}\|^2_{0,\Gamma_{0}}\\
		&+\frac{C_5 h}{3C_4q^5}\|(\nabla\nabla u_h+q^2\boldsymbol{T}u_h)\cdot\vec{n}\|^2_{0,\Gamma_{1}} +\frac{1}{qh^3}\|u_h\|^2_{0,\Gamma_{0}}+\frac{1}{q^3h}\|\nabla u_h\|^2_{\Gamma_{1}}.
	\end{align*}
	That is, there exists a constant $C_2=\min\{\frac{C_5}{3}, \frac{C_5}{3C_3}, \frac{C_5}{3C_4}, 1\}$ such that 	$A_h(u_h,u_h)  \geq C_2 \tnorm{u_h}{2,q,h}^2$. 	
\end{proof}
\begin{remark}[Galerkin orthogonality]\label{galerkin}
	Let $u\in H^{t}(\Omega),\ t\geq 4$ be the solution of \eqref{12}. If $u_h\in \text{ARG}_5(\Omega,\mathcal{T}_h)$ is the solution of \eqref{32}, then $A_h(u-u_h, \phi_h)=0$, $\forall \phi_h\in \text{ARG}_5(\Omega,\mathcal{T}_h)$.
\end{remark}
We carry out standard error analysis using the Galerkin orthogonality property, following the same approach as used for Poisson's equation in \cite{MR1479170}.
\begin{lemma}\label{Lemma4}
	Let the assumptions of Theorem \ref{Thm4} hold.  Given functions $v\in H^t(\Omega), \ t\geq 4$ and $v_h\in \text{ARG}_5(\Omega,\mathcal{T}_h)$, there exists a positive constant, $C$, such that   
	\begin{align}
		\tnorm{v-v_h}{2,q,h}&\leq \frac{C}{h^2}\left(\|v-v_h\|_0+\frac{h}{q}\|v-v_h\|_1+\frac{h^2}{q^2}\|v-v_h\|_2+\frac{h^3}{q^3}\sum_{\tau\in \mathcal{T}_h}\|v-v_h\|_{3,\tau}\right.\notag\\
		&\phantom{xxxxxxxxxxxxxxxxxxxxxxxxxxxxxxxxxxxxx}\left.+\frac{h^4}{q^4}\sum_{\tau\in \mathcal{T}_h}\|v-v_h\|_{4,\tau}\right).
	\end{align}
\end{lemma}
\begin{proof}
	%The following inequalities hold true for $v_h\in \text{ARG}_{5}(\Omega,\tau_h)\subset DG_{5}(\Omega,\tau_h)$ and $\nabla v_h\in DG_{4}(\Omega,\tau_h)$\cite{farrell2021new,MR2431403,MR2373954}
	%\begin{eqnarray}
	%\|v_h\|^2_{0,\Gamma_{0}}\leq \frac{C}{h}\|v_h\|^2_0,\label{37}\\
	%\|\nabla v_h\|^2_{0,\Gamma_{1}}\leq \frac{C}{h}\|\nabla v_h\|_0,\\ 
	%\|\nabla v_h\|_0\leq \frac{C}{h}\|v_h\|_0,\\
	%\|\nabla \nabla v_h\|_0\leq \frac{C}{h^2}\|v_h\|_0\label{40},
	%\end{eqnarray}
	%The result is a direct consequence of \eqref{34}-\eqref{35-} and \eqref{37}-\eqref{40}. 
	Define $r = v-v_h$, and note that $ r\in H^4(\tau), \forall \tau\in\mathcal{T}_h$.
	Applying Corollary \ref{Trace-cor} to the boundary integrals in \eqref{33} yields
	\begin{align}
		\frac{h^3}{q^7}\|\nabla\cdot(\nabla\nabla r+q^2\boldsymbol{T}r)\cdot\vec{n}\|^2_{0,\partial\tau\cap\left(\Gamma_{0}\right)} &\leq C\frac{h^3}{q^7}\|\nabla\nabla\nabla r\|_{0,\partial\tau\cap\left(\Gamma_{0}\right)}^2+ \mu_1\frac{h^3}{q^3}\|\nabla r\|_{0,\partial\tau\cap\left(\Gamma_{0}\right)}^2\notag\\&+  \mu_2\frac{h^3}{q^3}\|r\|_{0,\partial\tau\cap\left(\Gamma_{0}\right)}^2\notag\\
		&\leq C\frac{h^3}{q^7}\left(\frac{h}{q} \|\nabla \nabla\nabla\nabla r\|_{0,\tau}^2+\frac{q}{h }\|\nabla\nabla\nabla r\|_{0,\tau}^2\right)\notag\\
		&+C\frac{h^3}{q^3}\left(\frac{h}{q} \|\nabla \nabla r\|_{0,\tau}^2+\frac{q}{h }\|\nabla r\|_{0,\tau}^2+\frac{q}{h }\|r\|_{0,\tau}^2\right)\notag\\
		&\leq C\left(\frac{h^4}{q^8} \|\nabla \nabla\nabla\nabla r\|_{0,\tau}^2+\frac{h^2}{q^6}\|\nabla\nabla\nabla r\|_{0,\tau}^2\right)\notag\\
		&+ C\left(\frac{h^4}{q^4} \|\nabla \nabla r\|_{0,\tau}^2+\frac{h^2}{q^2}\|\nabla r\|_{0,\tau}^2+\frac{h^2}{q^2}\|r\|_{0,\tau}^2\right)\label{37}
	\end{align}
	\begin{align}
		\frac{h}{q^5}\|(\nabla\nabla r+q^2\boldsymbol{T}r)\cdot\vec{n}\|_{0,\partial\tau\cap\left(\Gamma_{1}\right)}^2 &\leq C\frac{h}{q^5}\|\nabla\nabla r\|_{0,\partial\tau\cap\left(\Gamma_{1}\right)}^2+C\mu_1\frac{h}{q}\|r\|_{0,\partial\tau\cap\left(\Gamma_{1}\right)}^2\notag\\
		&\leq C\left(\frac{h^2}{q^6}\|\nabla\nabla\nabla r\|_{0,\tau}^2+\frac{1}{q^4}\|\nabla\nabla r\|_{0,\tau}^2\right)\notag\\
		&+C\left(\frac{h^2}{q^2} \|\nabla r\|_{0,\tau}^2+\|r\|_{0,\tau}^2\right)\label{trace-Hessian},
	\end{align}
	\begin{align}
		\frac{1}{q^3h}\|\nabla r\|^2_{0,\partial\tau\cap\left(\Gamma_{1}\right)}&\leq {C}\left( \frac{1}{q^4}\|\nabla\nabla r\|_{0,\tau}^2+\frac{1}{q^2h^2}\|\nabla r\|_{0,\tau}^2\right)\label{trace-grad},\\
	\end{align}
	and
	\begin{align}
		\frac{1}{qh^3}\|r\|^2_{0,\partial\tau\cap(\Gamma_{0})}&\leq C\left(\frac{1}{q^2h^2}\|\nabla r\|_{0,\tau}^2+\frac{1}{h^4}\|r\|_{0,\tau}^2\right)\label{40}.
	\end{align}
	Summing inequalities \eqref{37}-\eqref{40} over $\tau\in\mathcal{T}_h$ and then combining these with the fact that $r\in H^2(\Omega)$ leads to
	\begin{align}
		\tnorm{v-v_h}{2,q,h}^2&\leq \frac{C}{h^4}\left(\|v-v_h\|_0^2+\frac{h^2}{q^2}\|v-v_h\|_1^2+\frac{h^4}{q^4}\|v-v_h\|^2_2+\frac{h^6}{q^6}\sum_{\tau\in \mathcal{T}_h}\|v-v_h\|_{3,\tau}^2\right.\notag\\
		&\phantom{xxxxxxxxxxxxxxxxxxxxxxxxxxxxxxxxx}\left.+\frac{h^8}{q^8}\sum_{\tau\in \mathcal{T}_h}\|v-v_h\|_{4,\tau}^2\right).
	\end{align}
	Taking the square root of both sides and using the fact that $\sqrt{\sum_{i=1}^{4}x^2_i}\leq \sum_{i=1}^{4}x_i, \ \forall x_i>0$ completes the proof.
\end{proof}
%\begin{lemma}\label{Remark4.4}
%Let the assumptions of Lemma \ref{Lemma4} hold. Then for $v_h\in \text{ARG}_5(\Omega,\tau_h)$ we have the bound.
%\begin{eqnarray}
%\tnorm{v_h}{2,h}\leq \frac{C}{h^2}\left(\|v_h\|_0+h\|\nabla v_h\|_0+h^2\|v_h\|_2+h^3\sum_{\tau\in \tau_h}\|v_h\|_{3,\tau}+h^4\sum_{\tau\in \tau_h}\|v_h\|_{4,\tau}\right)
%\end{eqnarray}
%\end{lemma}
%\begin{proof}
%This is a direct consequence of Inequalities \eqref{34}-\eqref{35-} and \eqref{trace-grad}-\eqref{40}.
%\begin{eqnarray}
%\tnorm{v_h}{2,h}\leq \frac{C}{h^2}\left(\|v_h\|_0+h\|\nabla v_h\|_0+h^2\|v_h\|_2\right)\leq \frac{C}{h^2}\left(\|v_h\|_0+h\|\nabla v_h\|_0+h^2\|v_h\|_2+h^3\sum_{\tau\in \tau_h}\|v_h\|_{3,\tau}+h^4\sum_{\tau\in \tau_h}\|v_h\|_{4,\tau}\right)
%\end{eqnarray}
%\end{proof}
%\spm{Need to track assumptions from earlier results in statement of this theorem}
%\spm{I want to see powers of $q$ in this result.}
\begin{theorem}\label{thm}
	Let the assumptions of Lemma \ref{Lemma4} hold, let $u_h\in \text{ARG}_5(\Omega,\mathcal{T}_h)$ be the solution of \eqref{32}, and let $u\in H^t(\Omega), \ 4\leq t\leq 6$ be the solution of \eqref{12}. Then, 	
	\begin{equation}
		\tnorm{u-u_h}{2,q,h}\leq CBq^{4}h^{t-2}|u|_{t}.
	\end{equation}
\end{theorem}
\begin{proof}
	For any $v_h\in \text{ARG}_{5}(\Omega,\mathcal{T}_h)$, by the triangle inequality, we have
	\begin{eqnarray}
		\tnorm{u-u_h}{2,q,h}\leq \tnorm{u-v_h}{2,q,h}+ \tnorm{u_h-v_h}{2,q,h}.
	\end{eqnarray}
	By the coercivity and continuity of $A_h$, and Remark \ref{galerkin}, 
	\begin{eqnarray}
		\tnorm{u_h-v_h}{2,q,h}^2&\leq& C A_h(u_h-v_h, u_h-v_h)=CA(u-v_h, u_h-v_h)\notag \\
		&\leq&CBq^4\tnorm{u-v_h}{2,q,h}\tnorm{u_h-v_h}{2,q,h}.
	\end{eqnarray}
	Therefore, $\tnorm{u_h-v_h}{2,q,h}\leq CBq^4\tnorm{u-v_h}{2,q,h}$ and 
	\begin{eqnarray}
		\tnorm{u-u_h}{2,q,h}\leq CBq^4 \left(\inf_{v_h\in\text{ARG}_{5}(\Omega,\mathcal{T}_h)}\tnorm{u-v_h}{2,q,h}\right).
	\end{eqnarray}
	Applying Lemma \ref{Lemma4} and existing bounds on the Argyris interpolation operator \cite{MR2322235} leads to the bound 
	\begin{align*}
		\tnorm{u-u_h}{2,q,h}\leq& \frac{CBq^{4}}{h^2}\inf_{v_h\in\text{ARG}_{5}(\Omega,\mathcal{T}_h)}\left(\|u-v_h\|_0 +\frac{h}{q}|u-v_h|_1 +\frac{h^2}{q^2}|u-v_h|_2 \right.\\ & \phantom{xxxxxxxxxxxxxxxxxxxxxxx}\left.+\frac{h^3}{q^3}\sum_{\tau\in \mathcal{T}_h}|u-v_h|_{3,\tau}+\frac{h^4}{q^4}\sum_{\tau\in \mathcal{T}_h}|u-v_h|_{4,\tau}\right)\\
		\leq& CBq^{4}h^{t-2}|u|_t, \ \ \text{for} \ 4\leq t\leq 6.
	\end{align*} 
\end{proof}
%\spm{In Corollary~\ref{cor3}, we use a seminorm for the Argyris interpolation bounds, but here we use the full norm?  Shouldn't this be semi-norms above and below?  And can we use the same notation ($s$ vs $t$) in all results?}
\begin{remark}
	Comparing this with the bound in~\Cref{cor3}, we see a slight degradation in the power of $q$, but no degradation in $h$.  Thus, if $u\in H^6(\Omega)$ and behaves like $e^{iq\vec{n}\cdot\vec{x}}$, we now seek a mesh with $h^4q^6 < 1$ when $Bq^4 = \mathcal{O}(1)$, which is still reasonable for $q \approx 40$.  Numerical results show, however, that even this estimate is pessimistic, and that a reasonable error tolerance can generally be achieved when $h = \mathcal{O}(1)$, i.e., independent of $q$.
\end{remark}
%% \begin{remark}
%% Given the solution $u=\sin(q \vec{n}\cdot\vec{x})$, and $u_h$ is the solution of Problem \eqref{32}, we have the following error estimates using Theorem \eqref{thm}, 
%% \begin{eqnarray}
%% \|u-u_h\|_{2,q,h}&\leq& CBq^{10}h^{4},\label{Error}
%% %\|u-u_h\|_{0}&\leq& C\theta_1(q)\left(Bq^{10}+B^2q^{27/2}\right)h^{9/2},
%% \end{eqnarray}   
%% Given that one is interested in $\|u-u_h\|_{2,q,h}\leq 10^{-8}$, Inequality \eqref{Error} suggests that $h\leq \frac{10^{-2}}{\sqrt[4]{CB}q^{5/2}}$. Note that these are pessimistic estimates as our numerical experiments showed that $h\leq \Xi$ is enough to guarantee such an estimate. 
%% \ajh{The value of $\Xi$ will be given later once we are happy with the theory and the numerical experiments.}    
\begin{remark}
	The discrete solution, $u_h$, may fail to exactly satisfy the essential boundary conditions, with $u_h\neq0 \ \text{on}\ \Gamma_{0}$ and/or $\nabla u_h\neq0\ \text{on} \ \Gamma_{1}$. Nevertheless, the error in these terms on the boundary also converges to zero, since we have the bounds
	\begin{eqnarray}
		\|u_h\|_{0,\Gamma_{0}}=\|u-u_h\|_{0,\Gamma_{0}}\leq (qh^3)^{1/2}\tnorm{u-u_h}{2,q,h}\leq CBq^{9/2}h^{t-1/2}|u|_t,\\ 
		\|\nabla u_h\|_{0,\Gamma_{1}}=\|\nabla (u-u_h)\|_{0,\Gamma_{1}}\leq (q^3h)^{1/2}\tnorm{u-u_h}{2,q,h}\leq CBq^{11/2}h^{t-3/2}|u|_t.
	\end{eqnarray} 
\end{remark}
\begin{remark}
	Preliminary results, not reported here, showed that the symmetric version of Nitsche's method resulted in the same dependence on $h$ as above, but with worse dependence on $q$.  Notably, with the symmetric Nitsche approach, we could also recover optimal-in-$h$ convergence for the error in $u$ measured in the $L^2$-norm, but with a dramatic increase in the power of $q$ in the approximation results.  Here, we can prove only a slight improvement in the $L^2$-error estimate for $u$ using arguments similar to \cite[Proposition 5.3]{doi:10.1137/10081784X}, but such estimates have little value, since they again trade worse dependence on $q$ for better dependence on $h$.  Whether such results can be improved (e.g., using a nonsymmetric penalty-free version of Nitsche's method, as in~\cite{doi:10.1137/10081784X}) is left for future work.
\end{remark}
\subsection{C0IP methods}\label{ssec:C0IP}
We next apply a C0IP method for the primal formulation \eqref{12}, aiming to approximate the solution with a $H^1(\Omega)$-conforming function and to weakly enforce $H^2(\Omega)$-conformity. Such an interior penalty method for the biharmonic operator, with either homogeneous clamped boundary conditions or Cahn--Hilliard type boundary conditions with a vanishing corner degree of freedom (to guarantee uniqueness of the solution), was presented in \cite{MR3051409}. As in that approach, we use Nitsche-type penalty methods to implement essential boundary conditions on the gradient, but strongly impose essential boundary conditions on the solution. The Nitsche term is, consequently, added to $\Gamma_{1}$. In this case, the weak form is to find $u_h\in {CG}_{k}^{\Gamma_{0}}(\Omega,\mathcal{T}_h),\ k\geq 2$, such that 
\begin{eqnarray}\label{C0IP}
	\tilde{a}_h(u_h,\phi_h)=\int_{\Omega}f\phi_h,\ \ \forall \phi_h\in {CG}_{k}^{\Gamma_{0}}(\Omega,\mathcal{T}_h),
\end{eqnarray}  
where 
\begin{align*}
	\tilde{a}(u_h,\phi_h) &= \hat{a}(u_h,\phi_h)- B\sum_{e\in\epsilon_h\setminus\Gamma_{2}}\int_{e}\left\{\!\!\!\left\{\vec{n}\cdot\bigg(\nabla\nabla u_h+q^2\boldsymbol{T}u_h\bigg)\cdot \vec{n} \right\}\!\!\!\right\}\left[\!\!\left[\frac{\partial \phi_h}{\partial n}\right]\!\!\right]\\
	+&B  \sum_{e\in\epsilon_h\setminus\Gamma_{2}}\int_{e}\left\{\!\!\!\left\{\vec{n}\cdot\bigg(\nabla\nabla \phi_h+q^2\boldsymbol{T}\phi_h\bigg)\cdot \vec{n}\right\}\!\!\!\right\}\left[\!\!\left[\frac{\partial u_h}{\partial n}\right]\!\!\right]\\
	+&\frac{1}{q^3h}\sum_{e\in\epsilon_h\setminus\Gamma_{2}}\int_{e}\left[\!\!\left[\frac{\partial u_h}{\partial n}\right]\!\!\right]\left[\!\!\left[\frac{\partial \phi_h}{\partial n}\right]\!\!\right],
\end{align*}
and
\begin{align*}
	\hat{a}(u_h,\phi_h) &= B\sum_{\tau\in\mathcal{T}_h}\left(\int_{\tau}\nabla\nabla u_h:\nabla\nabla \phi_h+q^2\int_{\tau}\nabla\nabla u_h:\boldsymbol{T}\phi_h+q^2\int_{\tau}\nabla\nabla \phi_h:\boldsymbol{T}u_h\right)\\
	&+\int_{\Omega}(Bq^4\boldsymbol{T}:\boldsymbol{T}+m)u_h\phi_h.
	%\hat{a}(u_h,\phi_h) &= B\sum_{\tau\in\tau_h}\left(\int_{\tau}\nabla\nabla u_h:\nabla\nabla \phi_h+Bq^2\int_{\tau}\nabla\nabla u_h:\boldsymbol{T}\phi_h+Bq^2\int_{\tau}\nabla\nabla \phi_h:\boldsymbol{T}u_h\right)+\int_{\Omega}Bq^4(\boldsymbol{T}:\boldsymbol{T}+m)u_h\phi_h.
\end{align*}
Here, $\left[\!\left[\cdot\right]\!\right], \left\{\!\!\left\{\cdot \right\}\!\!\right\} $ denote the standard jump and average functions defined in \cite{MR3873982,MR2142191, MR3051409,MR345432}, and $\mathcal{T}_h$ and $\epsilon_h$ are the sets of cells and edges (including the boundary) in the mesh, respectively. We define the following norm on
$CG_{k}^{\Gamma_{0}}(\Omega,\mathcal{T}_h)$, 
\begin{align}
	\tnorm{u_h}{h}^2&= q^{-4}\left( \sum_{\tau\in \mathcal{T}_h}|u_h|_{2,\tau}^2+\|\nabla u_h\|_0^2\right)+\|u_h\|_0^2\\
	&+\frac{h}{q^5}\sum_{e\in\epsilon_h\setminus\Gamma_{2}}\int_{e}\left\{\!\!\!\left\{\vec{n}\cdot\bigg(\nabla\nabla u_h+q^2\boldsymbol{T}u_h\bigg)\cdot \vec{n}\right\}\!\!\!\right\}^2+\frac{1}{q^3h}\sum_{e\in\epsilon_h\setminus\Gamma_{2}}\int_{e}\left[\!\!\left[\frac{\partial u_h}{\partial n}\right]\!\!\right]^2.\label{C0-norm}
\end{align}
The following inequalities are useful in proving the well-posedness of \eqref{C0IP}.
\begin{lemma}\label{Lemma3}
	Let $\{\mathcal{T}_h\}$ be a family of quasiuniform meshes of $\Omega$ and $\boldsymbol{T}:\boldsymbol{T}\leq \mu_1$ pointwise on $\bar{\Omega}$.
	\begin{itemize}
		\item The $H^2$-discrete Poincar\'e inequality is that there exists $C_1>0$ such that
		\begin{equation}\label{355}
			\| \nabla u_h\|_0^2\leq C_1 \bigg(\|u_h\|_0^2+\sum_{\tau\in \mathcal{T}_h}|u_h|_{2,\tau}^2+\frac{1}{h}\sum_{e\in\epsilon_h\setminus\partial\Omega}\int_{e}\left[\!\!\left[\frac{\partial u_h}{\partial n}\right]\!\!\right]^2\bigg), \ \ \forall \phi_h\in CG_{k}(\Omega,\mathcal{T}_h).
		\end{equation}
		%		In addition, 
		%		\begin{equation}\label{35}
		%		q^2\| u_h\|_1^2\leq C \bigg(q^4\|u_h\|_0^2+ \sum_{\tau\in \tau_h}|u_h|_{2,\tau}^2+\frac{q^4}{h}\sum_{e\in\epsilon_h\setminus\Gamma_{0}}\int_{e}\left[\!\!\left[\frac{\partial u_h}{\partial n}\right]\!\!\right]^2\bigg), \ \ \forall \phi_h\in CG^{\Gamma_{0}}_{k}(\Omega,\tau_h).
		%		\end{equation}
		\item There exists $C_2>0$ such that %\cite[Inequality 3.20]{MR3051409}
		%  \spm{This inequality isn't exactly what's given in [18, (3.20)].  I
		%    think you need to include some details to fill in the gap.}
		\begin{eqnarray}\label{eq55}
			\sum_{e\in\epsilon_h}\int_{e}\left\{\!\!\!\left\{\vec{n}\cdot\bigg(\nabla\nabla
			\phi_h+q^2\boldsymbol{T}\phi_h\bigg)\cdot
			\vec{n}\right\}\!\!\!\right\}^2\leq \frac{C_2}{h}\left(\bigg(\sum_{\tau\in \mathcal{T}_h}|\phi_h|^2_{2,\tau}\bigg)+q^4\|\phi_h\|^2_0\right). \label{36}
		\end{eqnarray}
	\end{itemize}
\end{lemma}
\begin{proof}
	From \cite[Example 5.4]{MR2106270}, we have
	\begin{eqnarray}
		\|u_h\|_1^2\leq C\left(\sum_{\tau\in \mathcal{T}_h}|u_h|^2_{2,\tau}+\frac{1}{h}\sum_{e\in\epsilon_h\setminus\partial\Omega}\int_{e}\left[\!\!\left[\frac{\partial u_h}{\partial n}\right]\!\!\right]^2+[\xi(u_h)]^2\right),
	\end{eqnarray}
	where $\xi$ is a seminorm that satisfies Equations (1.2), (1.3), (2.15), and (3.3) in  \cite{MR2106270}.
	As $\xi(u)=\|u\|_0$ satisfies these properties, Inequality \eqref{355} holds. While (1.2) and (1.3) can be trivially proved, (2.15) is shown in~\cite[Corollary 2.2]{MR2106270}, and (3.3) follows directly from~\cite[Inequality (3.2)]{MR2106270}.
	To prove Inequality \eqref{36}, we first use \cite[Inequality (3.20)]{MR3051409} to bound the term containing $\vec{n}\cdot\left(\nabla\nabla \phi_h\right)\cdot\vec{n}$.  To bound the remaining term, we
	apply standard inverse trace inequalities to get
	\begin{eqnarray}
		\sum_{e\in\epsilon_h}\int_{e}\left\{\!\!\!\left\{\vec{n}\cdot\bigg(q^2\boldsymbol{T}\phi_h\bigg)\cdot \vec{n}\right\}\!\!\!\right\}^2\leq  \frac{Cq^4}{h}\|\phi_h\|^2_0.\label{58}
	\end{eqnarray}
	Adding this to the right-hand side of \cite[Inequality (3.20)]{MR3051409} completes the proof.
\end{proof}
%\spm{We need to be careful with where we need assumptions on the domain, and what assumptions we need.  I've removed the polygonal assumption here, since it doesn't appear necessary \textit{here}.  I think, instead, we need it only in the error estimates, since that's when we need to compare $u$ and $u_h$.}
These inequalities are enough to establish coercivity (and, thus, well-posedness) of the discrete problem in \eqref{C0IP}.
\begin{theorem}\label{Theroem7}
	Let $\{\mathcal{T}_h\}$ be a family of quasiuniform meshes of $\Omega$, $f\in L^2(\Omega)$, and $\boldsymbol{T}:\boldsymbol{T}\leq \mu_1$ pointwise on $\bar{\Omega}$.  Then, Problem \eqref{C0IP} is well-posed.
\end{theorem}
\begin{proof}
	The bilinear form $\tilde{a}$ defined in \eqref{C0IP} is continuous and coercive in the norm defined in~\eqref{C0-norm}. Proving continuity is straightforward using the Cauchy--Schwarz inequality, yielding a continuity constant that is $\mathcal{O}(Bq^4)$, as in the conforming case.
	Coercivity of $\tilde{a}$ can be proven by combining the inequalities of Lemma \ref{Lemma3} and Cauchy--Schwarz inequality. By direct substitution, we have    
	\begin{align*}
		\tilde{a}(u_h,u_h)&=B\sum_{\tau \in\mathcal{T}_h}\left(\int_{\tau}|u_h|_{2,\tau}^2+2q^2\int_{\tau}\nabla\nabla u_h:\boldsymbol{T}u_h\right)+\int_{\Omega}(Bq^4\boldsymbol{T}:\boldsymbol{T}+m)u_h^2\\
		&+\frac{1}{q^3h}\sum_{e\in\epsilon_h\setminus\Gamma_{2}}\int_{e}\left[\!\!\left[\frac{\partial u_h}{\partial n}\right]\!\!\right]^2.\notag
	\end{align*}        
	As in the proof of Theorem~\ref{thm3}, we can show that there exists $C_3>0$ such that
	\begin{align*}
		B\sum_{\tau \in\mathcal{T}_h}\left(\int_{\tau}|u_h|_{2,\tau}^2+2q^2\int_{\tau}\nabla\nabla u_h:\boldsymbol{T}u_h\right)&+\int_{\Omega}(Bq^4\boldsymbol{T}:\boldsymbol{T}+m)u_h^2\\
		&\geq C_3\left(q^{-4}\sum_{\tau\in \mathcal{T}_h}|u_h|_{2,\tau}^2+\|u_h\|_0^2 \right).
	\end{align*}
	%\spm{Using $C_2$ and $C_3$ below makes me think you're referring to the constants from above, but they're not named as such.}
	With this, we have
	\begin{align*}
		\tilde{a}(u_h,u_h) &\geq C_3\left(q^{-4}\sum_{\tau\in \mathcal{T}_h}|u_h|_{2,\tau}^2+\|u_h\|_0^2 \right)+\frac{1}{q^3h}\sum_{e\in\epsilon_h\setminus\Gamma_{2}}\int_{e}\left[\!\!\left[\frac{\partial u_h}{\partial n}\right]\!\!\right]^2\\
		&\geq \frac{C_3}{3}\left(q^{-4}\sum_{\tau\in \mathcal{T}_h}|u_h|_{2,\tau}^2+\|u_h\|_0^2\right)+\frac{C_3h}{3C_2q^4}\sum_{e\in\epsilon_h\setminus \Gamma_{2}}\int_{e}\left\{\!\!\!\left\{\vec{n}\cdot\bigg(\nabla\nabla
		\phi_h+q^2\boldsymbol{T}\phi_h\bigg)\cdot
		\vec{n}\right\}\!\!\!\right\}^2 \\
		&+\frac{\min\{C_3,1\}}{3}\left(q^{-4}\sum_{\tau\in \mathcal{T}_h}|u_h|_{2,\tau}^2+\|u_h\|_0^2+\frac{1}{q^3h}\sum_{e\in\epsilon_h\setminus\Gamma_{2}}\int_{e}\left[\!\!\left[\frac{\partial u_h}{\partial n}\right]\!\!\right]^2\right)\\
		&+\frac{2}{3q^3h}\sum_{e\in\epsilon_h\setminus\Gamma_{2}}\int_{e}\left[\!\!\left[\frac{\partial u_h}{\partial n}\right]\!\!\right]^2\\
		&\geq \frac{\min\{C_3,1\}}{3}\left(q^{-4}\sum_{\tau\in \mathcal{T}_h}|u_h|_{2,\tau}^2+\frac{q^{-4}}{C_1}\|\nabla u_h\|_0^2+\|u_h\|_0^2\right)\\&+\frac{C_3h}{3C_2q^5}\sum_{e\in\epsilon_h\setminus \Gamma_{2}}\int_{e}\left\{\!\!\!\left\{\vec{n}\cdot\bigg(\nabla\nabla
		\phi_h+q^2\boldsymbol{T}\phi_h\bigg)\cdot
		\vec{n}\right\}\!\!\!\right\}^2
		+\frac{2}{3q^3h}\sum_{e\in\epsilon_h\setminus\Gamma_{2}}\int_{e}\left[\!\!\left[\frac{\partial u_h}{\partial n}\right]\!\!\right]^2,
	\end{align*}
	where $C_1$ and $C_2$ are defined in Lemma \eqref{Lemma3}. Thus, the coercivity constant is $\mathcal{O}(1)$.
\end{proof}
\begin{remark}
	Note that the bilinear form $\tilde{a}$ is also continuous with respect to the mesh-dependent norm for functions in $H^t(\Omega)\cap H^1_{\Gamma_{0}}(\Omega), \ t\geq 4$, where the jump terms vanish over interior edges. That is, there exists a positive constant, $C$, such that 
	\begin{equation}
		\tilde{a}(u,\phi)\leq CBq^{4} \tnorm{u}{h}\tnorm{\phi}{h}, \ \forall u,\phi\in H^t(\Omega)\cap H^1_{\Gamma_{0}}(\Omega).
	\end{equation} 
\end{remark}
\begin{lemma}\label{Lemma-c0ip}
	Let $\{\mathcal{T}_h\}$ be a family of quasiuniform meshes of $\Omega$, $v\in H^t(\Omega)\cap H^1_{\Gamma_{0}}(\Omega), \ t\geq 4$, $\boldsymbol{T}:\boldsymbol{T}\leq \mu_1$ pointwise on $\bar{\Omega}$, and $v_h\in CG^{\Gamma_{0}}_k(\Omega,\mathcal{T}_h)$.  Then,
	\begin{align}\label{65}
		\tnorm{v-v_h}{h}^2&\leq \frac{C}{h^4}\left(\|v-v_h\|_0^2+\frac{h^2}{q^2}| v-v_h|_1^2+\frac{h^4}{q^4}\sum_{\tau\in \mathcal{T}_h}|v-v_h|_{2,\tau}+\frac{h^6}{q^6}\sum_{\tau\in \mathcal{T}_h}|v-v_h|_{3,\tau}^2\right.\notag\\&\phantom{xxxxxxxxxxxxxxxxxxxxxxxxxxxxxxxxx}\left.+\frac{h^8}{q^8}\sum_{\tau\in\mathcal{T}_h}|v-v_h|_{4,\tau}^2\right)
	\end{align}
\end{lemma}
\begin{proof}
	First note that $\left[\!\!\left[\frac{\partial v}{\partial n}\right]\!\!\right]=0$ on the interior edges of $\mathcal{T}_h$.
	For $v\in H^t(\Omega)\cap H^1_{\Gamma_{0}}(\Omega), \ t\geq 4$, we have that $ r=v-v_h\in H^4(\tau),\ \forall \tau\in\mathcal{T}_h$.  We apply Corollary~\ref{Trace-cor} to the boundary integrals in  \eqref{C0-norm}, yielding		
	\begin{equation}
		\frac{1}{q^3h}\sum_{e\in\epsilon_h\setminus\Gamma_{2}}\int_{e}\left[\!\!\left[\frac{\partial r}{\partial n}\right]\!\!\right]^2\leq C\sum_{\tau\in\mathcal{T}_h}\left(\frac{1}{q^2h^2}\|\nabla r\|_{0,\tau}^2+\frac{1}{q^4}|r|^2_{2,\tau}\right)\label{66}
	\end{equation}
	and
	\begin{align}
		\frac{h}{q^5}\sum_{e\in\epsilon_h\setminus\Gamma_{2}}\int_{e}\left\{\!\!\!\left\{\vec{n}\cdot\bigg(\nabla\nabla r+q^2\boldsymbol{T}r\bigg)\cdot \vec{n}\right\}\!\!\!\right\}^2 &\leq C\left(\sum_{\tau\in\mathcal{T}_h}\left(\frac{1}{q^4}|r|_{2,\tau}^2+\frac{h^2}{q^6}|r|_{3,\tau}\right)\right.\notag\\
		&\phantom{xxxxxxxxx}\left.+\|r\|_0^2+\frac{h^2}{q^2}|\nabla r|_1^2\right). \label{eq:fixme}  
	\end{align} 
	As a result, Inequality \eqref{65} holds.
\end{proof}
\begin{remark}[Galerkin orthogonality]\label{galerkin-C0IP}
	Let $u\in H^{t}(\Omega),\ t\geq 4$ be the solution of \eqref{12}. If $u_h\in CG^{\Gamma_{0}}_k(\Omega,\mathcal{T}_h)$ is the solution of \eqref{32}, then $A_h(u-u_h, \phi_h)=0$, $\forall \phi_h\in CG_{k}^{\Gamma_{0}}(\Omega,\mathcal{T}_h)$.
\end{remark}
\begin{theorem}\label{Thm7}
	Assume that the solution of \eqref{12} satisfies $u\in H^t(\Omega)$, for $t\geq 4$. If $u_h\in  {CG}_{k}^{\Gamma_{0}}(\Omega,\mathcal{T}_h)$ is the solution of \eqref{C0IP}, then
	\begin{eqnarray*}
		\tnorm{u-u_h}{h}\leq CBq^{4} h^{\min\{t,k+1\}-2}|u|_t.
	\end{eqnarray*} 	
\end{theorem}
\begin{proof}
	We use the orthogonality property to prove the error estimates as in Theorem \ref{thm}. Given $u\in H^t(\Omega), \ t\geq 4$, then we have the standard quasi-optimality result,
	\begin{eqnarray}
		\tnorm{u-u_h}{h}\leq CB q^{4}\inf_{v_h\in CG_k^{\Gamma_{0}}(\Omega,\mathcal{T}_h)}\tnorm{u-v_h}{h}.
	\end{eqnarray}  
	Now, we use Lemma \ref{Lemma-c0ip} and the standard Lagrange interpolation error estimates \cite{MR0520174,MR2373954} to yield
	\begin{align*}
		\tnorm{u-u_h}{h}&\leq \frac{CBq^{4}}{h^2}\left(\|u-v_h\|_0+\frac{h}{q}\|\nabla (u-v_h)\|_0+\frac{h^2}{q^2}\sum_{\tau\in \mathcal{T}_h}|u-v_h|_{2,\tau}\right.\\
		&\phantom{xxxxxxxxxxxxxxxxxxxxxx}\left.+\frac{h^3}{q^3}\sum_{\tau\in \mathcal{T}_h}|u-v_h|_{3,\tau}+\frac{h^4}{q^4}\sum_{\tau\in \mathcal{T}_h}|u-v_h|_{4,\tau}\right)\\
		&\leq \frac{CBq^{4}}{h^2}h^{\min\{t,k+1\}}|u|_t =CBq^4h^{\min\{t,k+1\}-2}|u|_t.
	\end{align*}
\end{proof}

\subsection{Mixed finite elements}\label{ssec:mixed}
We now consider a mixed finite-element discretization of the systems reformulation given in~\eqref{keyy1}-\eqref{keyy2}. We consider a conforming discretization, with $u_h\in DG_k(\Omega,\mathcal{T}_h) \subset L^2(\Omega)$ and $\vec{\alpha}_h\in RT^{\Gamma_{3}}_{k+1}(\Omega,\mathcal{T}_h) \subset H_{\Gamma_{3}}(\mathrm{div};\Omega)$.  In order to prove the required inf-sup condition on $b(\vec{\alpha}_h,(u_h,\vec{v}_h))$, the choice of space for $\vec{v}_h$ is based on generalized Taylor--Hood elements, writing  $\vec{v}_h\in V_h$, where $V_h=\left\{\vec{\psi}_h\ \middle| \ \ \vec{\psi}_h\in [CG_{k+2}(\Omega,\mathcal{T}_h)]^2\cap V \right\}$, with (as before) $V=\left\{\vec{v}\in[H^{1}_{\Gamma_{1}}(\Omega)]^2\middle|\ \vec{v}\times\vec{n}=0\ \text{on}\ \Gamma_{0,2}  \right\}$.  As in the proof of Theorem~\ref{thm:continuum_mixed}, the Helmholtz decomposition of $\vec{\alpha}_h \in RT^{\Gamma_{3}}_{k+1}(\Omega,\mathcal{T}_h)$ will be used to establish the inf-sup condition, now in its discrete form, with  $p_h \in CG^{\Gamma_{3}}_{k+1}(\Omega,\mathcal{T}_h)$, but only for domains $\Omega\subset\mathbb{R}^2$.   

\begin{theorem}\label{Thm3}
	Let the assumptions of Theorem~\ref{thm:continuum_mixed} be satisfied, and let $\mathcal{T}_h$ be a quasiuniform family of triangular meshes of $\Omega$. Let the bilinear forms $\mathcal{A}$ and $b$ and linear form $F$ be defined as in \eqref{keyr}-\eqref{RHS_func}. For sufficiently small $h$, the discrete saddle-point problem of finding $(u_h,\vec{v}_h,\vec{\alpha}_h) \in  DG_k(\Omega,\mathcal{T}_h)\times V_h\times RT^{\Gamma_{3}}_{k+1}(\Omega,\mathcal{T}_h)$ such that
	\begin{eqnarray}
		\mathcal{A}\big((u_h,\vec{v}_h), (\phi_h,\vec{\psi}_h))+b(\vec{\alpha}_h,(\phi_h,\vec{\psi}_h))&=&F(\phi_h), \ \ \forall (\phi_h, \vec{\psi}_h)\in\notag DG_k(\Omega,\mathcal{T}_h)\times V_h\\
		b\big(\vec{\beta}_h, (u_h,\vec{v}_h)\big)&=&0, \ \ \forall \vec{\beta}\in RT^{\Gamma_{3}}_{k+1}(\Omega,\mathcal{T}_h)\label{keyu}
	\end{eqnarray}
	is well-posed for $k\geq 1$.
\end{theorem}
\begin{proof}
	We follow the standard theory (see, e.g., \cite{MR3097958}), requiring continuity of $\mathcal{A}$ and $b$, coercivity of $\mathcal{A}$, and an inf-sup condition on $b$.  Because we consider a conforming discretization, continuity of both $\mathcal{A}$ and $b$ follow directly as in Theorem~\ref{thm:continuum_mixed}, with the same constants in the norms used there, where we use the discrete Helmholtz decomposition of $\vec{\alpha}_h=\nabla\times p_h+\nabla_h^{\Gamma_{3}}\phi_h$ defined in Lemma \ref{lemma,helm} to prove continuity of $b$.
	Similarly, we consider coercivity of the bilinear form $\mathcal{A}\left((u_h,\vec{v}_h), (\phi_h, \vec{\psi}_h)\right)$ on the set 
	\begin{equation*}
		\Lambda_h=\left\{(u_h,\vec{v}_h)\in DG_{k}(\Omega,\mathcal{T}_h)\times V_h \,\middle| \, b\left(\vec{\alpha}_h,(u_h,\vec{v}_h)\right)=0, \ \forall \vec{\alpha}_h\in RT_{k+1}^{\Gamma_{3}}(\Omega,\mathcal{T}_h)\right\}.
	\end{equation*}
	This follows again as in the continuum case, using 
	$\vec{\beta}_1=\begin{bmatrix}
	S_h\\0
	\end{bmatrix}$ and $\vec{\beta}_2=\begin{bmatrix}
	0\\S_h
	\end{bmatrix}$, with $S_h\in CG_1^{\partial\Omega}(\Omega,\mathcal{T}_h)$ as defined in Lemma~\ref{Lemma-coercivity-poinc} in place of the continuum analogues in Theorem~\ref{thm:continuum_mixed}.

	Finally, we establish the discrete inf-sup condition, that
	\begin{eqnarray}\label{key8}
		I=\sup_{(u_h,\vec{v}_h)\in DG_{k}(\Omega,\mathcal{T}_h)\times V_h}\frac{\int_{\Omega}\vec{\alpha_h}\cdot\vec{v}_h+\int_{\Omega}u_h\nabla\cdot\vec{\alpha}_h}{\|(u_h,\vec{v}_h)\|_{0,q,1}}\geq Cq^2\|\vec{\alpha}_h\|_{\mathrm{Div}},
	\end{eqnarray}
	for some constant, $C$.
	By Lemma~\ref{lemma,helm}, for any $\vec{\alpha}_h \in RT_{k+1}^{\Gamma_{3}}(\Omega,\mathcal{T}_h)$, there exists $p_h\in CG^{\Gamma_{3}}_{k+1}(\Omega,\mathcal{T}_h)$ and $\eta_h\in DG_{k}(\Omega,\mathcal{T}_h)$ such that 
	\begin{equation}\label{key7}
		\vec{\alpha}_h=\nabla\times p_h+\nabla^{\Gamma_{3}}_h \eta_h.
	\end{equation}
	This gives the equivalent form to~\eqref{key8} of
	\begin{align*}
		I=\sup_{(u_h,\vec{v}_h)\in DG_{k}\times V_h}\frac{\int_{\Omega}\left(\nabla^{\Gamma_{3}}_h\eta_h+\nabla\times p_h\right)\cdot\vec{v}_h+\int_{\Omega}u_h\nabla\cdot\vec{\alpha}_h}{\sqrt{\|u_h\|_0^2+q^{-4}\|\vec{v}_h\|_1^2
		}}\geq &Cq^2\|\vec{\alpha}_h \|_{\mathrm{Div}},\notag\\ 
		& \forall \vec{\alpha}_h\in RT_{k+1}^{\Gamma_{3}}(\Omega,\mathcal{T}_h).\label{infsup_discrete}
	\end{align*}
	We show this by choosing $u_h=C_1\left(\nabla\cdot\vec{\alpha}_h-\eta_h\right)$. Let $h$ be sufficiently small so that the inf-sup condition of \cite[Lemma 3.5]{MR3660776} holds.  Then, for all $p_h\in CG^{\Gamma_{3}}_{k+1}(\Omega,\mathcal{T}_h)$, there exists a vector $\vec{\psi}_h\in V_h$ such that $\int_{\Omega}p_h\nabla\cdot\vec{\psi}_h\geq\|p_h\|_0^2$, and $\|\vec{\psi}_h\|^2_1\leq C_2\|\vec{p}_h\|^2_0$. To establish the inf-sup condition needed here, we choose $\vec{v}_h =[\psi_{2,h}, -\psi_{1,h}]^T$ which also belongs to $V_h$, giving $\nabla \cdot\vec{\psi}_h =\nabla\times\vec{v}_h $ and $\|\vec{\psi}_h \|_1^2=\|\vec{v}_h\|_1^2$.  The remainder of the proof follows identically as in the continuum case.
\end{proof} 
\begin{remark}\label{rem:3d_mixed}

		In the three-dimensional case, the Helmholtz decomposition of $\vec{\alpha}_h = \nabla_h \phi_h +\nabla\times \vec{p}_h $, gives $\phi_h\in DG_{k}(\Omega,\mathcal{T}_h)$ and $\vec{p}_h\in \mathrm{N}^1_{k+1}(\Omega,\mathcal{T}_h)$, where the $\mathrm{N}^1_{k+1}(\Omega,\mathcal{T}_h)$ is the N\'ed\'elec element of the first kind of order $k+1$~\cite{arnold2000multigrid}.  Thus, we can no longer leverage existing inf-sup results based on Taylor-Hood elements to extend Theorem~\ref{Thm3} to $\Omega\subset \mathbb{R}^3$.  While, numerically, it appears that enriching the space for $\vec{v}_h$, either adding bubble degrees of freedom or directly using a higher-order space, leads to a stable scheme, we leave analyzing this extension for future work.  
\end{remark}	
To measure the error estimates that arise from our three-field mixed formulation, we define the approximation errors,
\begin{eqnarray}
	E_{\left(u,\vec{v}\right)}&:=&\inf_{(\phi_h,\vec{\psi}_h)\in DG_{k}(\Omega,\mathcal{T}_h)\times V}\|(u,\vec{v})-(\phi_h,\vec{\psi}_h)\|_{0,q,1},\\
	E_{\vec{\alpha}}&:=&\inf_{\vec{\beta}_h\in RT_{k+1}(\Omega,\mathcal{T}_h)}\|\vec{\alpha}-\vec{\beta}_h\|_{\mathrm{Div}}.
\end{eqnarray}

%\spm{Please check the below: I used Thm 5.2.2 from BBF and got a different result than you gave.\\  Looking at this again, we might get a slight improvement in these results if we take $\alpha = B^{1/2}\nabla\cdot(\nabla\nabla u + q^2Tu)$.  We should check this...}
\begin{corollary}\label{cor:error}
	Let the assumptions of Theorem \ref{Thm3} be satisfied. Assume that $u\in H^{k+5}$ and $\boldsymbol{T}\in \boldsymbol{C}^{k+2}(\Omega)$, for $k\geq1$, $(u,\vec{v},\vec{\alpha})$ is the unique solution of Problem \eqref{keyy1}-\eqref{keyy2}, and $(u_h,\vec{v}_h,\vec{\alpha}_h)$ is the solution of Problem \eqref{keyu}. Then,  
	\begin{eqnarray}
		\|(u,\vec{v})-(u_h,\vec{v}_h)\|_{0,q,1}& \leq&C_1\left(\left(Bq^4+B^{1/2}q^4\right)E_{(u,\vec{v})} +q^4 E_{\vec{\alpha}}\right),\label{79}
		\\
		\|\vec{\alpha}-\vec{\alpha}_h\|_{\textrm{\normalfont Div}} & \leq&C_2\left(\left(B^{3/2}q^4+Bq^4\right) E_{(u,\vec{v})}+B^{1/2}q^4E_{\vec{\alpha}}\right), \label{80}
	\end{eqnarray}
	where $C_1,C_2$ are positive constants independent of $h$, $B$ and $q$.
\end{corollary}
\begin{proof}
	The standard error estimate, for example in \cite[Theorem 5.2.2]{MR3097958}, leads to \eqref{79} and \eqref{80}. Note that $\mathcal{A}$ and $b$ in \eqref{keyy1}-\eqref{keyy2} are continuous with $\mathcal{O}(Bq^4)$ and $\mathcal{O}(q^4)$ continuity constants, respectively, the coercivity constant is $\mathcal{O}(1)$, and the inf-sup constant is $\mathcal{O}(q^{2})$. 
\end{proof}
%\spm{Add the correct statement about regularity of $T$ for the below to hold.\\ I changed the multiplier on $|\vec{v}|_{k+3}$ below to $q^{-4}$ from $q^{-2}$, since I think we're working in the squared norm here - isn't that right?}
In the next corollary, we bound the approximation errors $E_{(u,\vec{v})} \ \text{and}\ E_{\vec{\alpha}}$ when $u\in H^{k+5}(\Omega)$ and $\boldsymbol{T}\in \boldsymbol{C}^{k+2}(\Omega)$. Note that, in this case, $\vec{v}=\nabla u\in \left[H^{k+4}(\Omega)\right]^2$ and $\vec{\alpha}=\nabla\cdot\left(\nabla\nabla u+q^2\boldsymbol{T}u\right)\in \left[H^{k+2}(\Omega)\right]^2$ by Corollary~\ref{Unique}. 
\begin{corollary}\label{Cor7}
	Let the assumptions of Corollary~\ref{cor:error} be satisfied and write $\vec{\alpha}=\nabla\phi+\nabla\times p$.  If, furthermore, $\phi\in H^{k+3}(\Omega)$ and $p\in \varrho$ (as defined in Remark~\ref{rem2.2}), then

	\begin{eqnarray}
		E_{\left(u,\vec{v}\right)}\leq Ch^{k+1}\bigg(|u|_{k+1}^2 +\frac{1}{q^4} |\vec{v}|_{k+3}^2\bigg)^{1/2}\label{E76},\\
		E_{\vec{\alpha}}\leq \frac{C}{q^2}h^{k+1}\left(|\nabla \phi|^2_{k+1}+|\Delta\phi|_{k+1}^2+|p|^2_{k+2}\right)^{1/2}.\label{E77}
	\end{eqnarray} 		
\end{corollary}
\begin{proof}
	Inequality \eqref{E76} holds using the classical continuous/discontinuous Lagrange interpolants~\cite{MR3097958}. To prove Inequality \eqref{E77}, we use the fact that
	$\vec{\alpha}\in [H^{k+2}(\Omega)]^2\cap H_{\Gamma_{3}}(\mathrm{div};\Omega)$ and, therefore, the functions $p$ and $\phi$ are in $H^{k+2}(\Omega)\cap H_{\Gamma_{0}}^1(\Omega)$ and  $H^{k+3}(\Omega)\cap H_{\Gamma_{3}}^1(\Omega)$ respectively if the conditions of Remark~\ref{rem2.2} are satisfied.  We bound $\|\vec{\alpha} - \vec{\beta}_h\|_{\mathrm{Div}}$ by writing $\vec{\beta}_h=\nabla_h^{\Gamma_{3}}\phi_h+\nabla\times p_h$ and noting that
	\[
	\|\vec{\alpha} - \vec{\beta}_h\|_{\mathrm{Div}}^2 = q^{-4}\left(\|p-p_h\|_0^2 + \|\nabla\phi- \nabla_h^{\Gamma_{3}}\phi_h\|_{\mathrm{div}}^2\right).
	\]
	We choose $\left(\phi_h, \vec{\zeta}_h\right)\in DG_{k}(\Omega,\mathcal{T}_h)\times RT_{k+1}^{\Gamma_{3}}(\Omega,\mathcal{T}_h)$ to be the solution of the mixed Poisson problem,
	\begin{eqnarray*}
		\int_{\Omega}\gamma_h\nabla\cdot\vec{\zeta}_h=\int_{\Omega}\Delta\phi \gamma_h,\ \ \forall \gamma_h\in DG_{k}(\Omega,\mathcal{T}_h),\\
		\int_{\Omega}\vec{\zeta}_h\cdot\vec{\Upsilon}_h+\phi_h\nabla\cdot\vec{\Upsilon}_h=0,\ \ \forall \vec{\Upsilon}_h\in RT^{\Gamma_{a}}_{k+1}(\Omega,\mathcal{T}_h).
	\end{eqnarray*}  
	The standard error estimate for $\vec{\zeta}_h$ is that
	\begin{eqnarray*}
		\|\vec{\zeta}_h-\nabla \phi\|_{\mathrm{div}}\leq Ch^{k+1}\left(|\nabla\phi|_{k+1}+|\Delta\phi|_{k+1}\right);
	\end{eqnarray*}
	however, $\vec{\zeta}_h=\nabla_h^{\Gamma_{3}}\phi_h$, giving
	\begin{equation}
		\|\nabla \phi-\nabla_h^{\Gamma_{3}}\phi_h\|_{\mathrm{div}}\leq Ch^{k+1}\left(|\nabla \phi|_{k+1}+|\Delta \phi|_{k+1}\right).\label{75}
	\end{equation}
	Choosing $p_h$ to be the interpolant of $p$ in $CG_{k+1}^{\Gamma_{0}}(\Omega,\mathcal{T}_h)$ gives
	\begin{equation}\label{key76}
		\|p-p_h\|_0\leq Ch^{k+2}|p|_{k+2}.
	\end{equation}
	Adding Inequalities \eqref{75} and \eqref{key76} leads to \eqref{E77}.
\end{proof}
While we can always compute the discrete Helmholtz decomposition of $\vec{\alpha}_h$, it is not always possible to compute the corresponding continuum Helmholtz decomposition of $\vec{\alpha}$, which would be needed to verify the above results by computing $\|\vec{\alpha}-\vec{\alpha}_h\|_{\mathrm{Div}}$. Therefore, we use the $H(\mathrm{div})$ norm in practice. We next show that the approximation error of $\vec{\alpha}$ in the $H(\mathrm{div})$ norm can be bounded by that in the strengthened norm.
\begin{corollary}\label{cor:first_error}
	Let the assumptions of Corollary \ref{Cor7} be satisfied. Then, 
	\begin{eqnarray}
		q^{-2}\|\vec{\alpha}-\vec{\alpha}_h\|_{\textrm{0}} & \leq& C\left(h^{k+1}|p|_{k+2}+\frac{1}{h}\|\vec{\alpha}-\vec{\alpha}_h\|_{\mathrm{Div}}\right),\label{Cor4-2}
	\end{eqnarray}
	where $C$ is a positive constant independent of $h$.
\end{corollary}
\begin{proof}
	Write $\vec{\alpha}_h=\nabla_h^{\Gamma_{3}}\phi_h+\nabla\times p_h$. Then we have 
	\begin{eqnarray}
		\|\vec{\alpha}-\vec{\alpha}_h\|_0&\leq& C\left(\|\nabla\phi-\nabla_h^{\Gamma_{3}}\phi_h\|_0+\|\nabla\times p-\nabla\times p_h\|_0\right)\notag\\
		&\leq&C\left(\|\vec{\alpha}-\vec{\alpha}_h\|_{\textrm{Div}}+\|\nabla\times p-\nabla\times p_h\|_0\right).\label{83}
		%&\leq&  C\left(\|\nabla\phi-\nabla_h^{\Gamma_{3}}\phi_h\|_0^2+\frac{1}{h^2}\|p-p_h\|_0^2\right)\\
		%&\leq&\frac{Cq^4}{h^2}\|\vec{\alpha}-\vec{\alpha}_h\|_{\mathrm{Div}}^2.
	\end{eqnarray} 
	Thus, we only need to bound $\|\nabla\times p-\nabla\times p_h\|_0$.
	To do this, we note that
	\begin{eqnarray*}
		\|\nabla\times p-\nabla\times p_h\|_0&\leq& \|\nabla\times p-\nabla\times z_h\|_0+\|\nabla\times z_h-\nabla\times p_h\|,
	\end{eqnarray*}
	where $z_h$ is the interpolant of $p$ in $CG_{k+1}(\Omega,\mathcal{T}_h)$, for which $\|\nabla \times p-\nabla \times z_h\|_0\leq Ch^{k+1}|p|_{k+2}$. Also, using standard arguments, we have that $\|\nabla \times z_h-\nabla \times p_h\|_0\leq \frac{C}{h}\|z_h-p_h\|_{0}$. Thus, 
	\begin{eqnarray}
		\|\nabla\times p-\nabla\times p_h\|_0&\leq& C\left(h^{k+1}|p|_{k+2}+\frac{1}{h}\| z_h-p_h\|_0\right)\notag\\
		&\leq& C\left(h^{k+1}|p|_{k+2}+\frac{1}{h}\| z_h-p\|_0+\frac{1}{h}\|p-p_h\|_0\right)\notag\\
		&\leq&C\left(h^{k+1}|p|_{k+2}+\frac{1}{h}\|p-p_h\|_0\right)\notag\\
		&\leq&C\left(h^{k+1}|p|_{k+2}+\frac{1}{h}\|\vec{\alpha}-\vec{\alpha}_h\|_{\text{Div}}\right)\label{84}.
	\end{eqnarray}
	Combining Inequalities \eqref{83} and \eqref{84} leads to \eqref{Cor4-2}.
\end{proof}	
%\spm{Need a summary remark with error estimates that depend explicitly on $B$, $q$, $h$, and $|u|_t$ when $u$ is sufficiently smooth.}
\begin{remark}\label{rem4.9}
	Let the assumptions of Corollary~\ref{Cor7} be satisfied, and let $W=\left(Bq^2+B^{1/2}q^2\right)$, $Z_1= \bigg(|u|_{k+1}^2 +\frac{1}{q^4} |\vec{v}|_{k+3}^2\bigg)^{1/2}$, and $Z_2=\left(|\nabla \phi|^2_{k+1}+|\Delta\phi|_{k+1}^2+|p|^2_{k+2}\right)^{1/2}$. Then, 
	\begin{align*}
		\|(u,\vec{v})-(u_h,\vec{v}_h)\|_{0,q,1} \leq&C_1q^2h^{k+1}\left(WZ_1+ Z_2\right),
		\\
		\frac{1}{q^2}\|\vec{\alpha}-\vec{\alpha}_h\|_{\textrm{0}} \leq \frac{C}{h}E_{\vec{\alpha}}& \leq C_2B^{1/2}q^2h^k\left(WZ_1+Z_2\right),\\
		\frac{1}{q^2}\|\nabla\cdot\vec{\alpha}-\nabla\cdot\vec{\alpha}_h\|_{\textrm{0}} \leq E_{\vec{\alpha}}& \leq C_3B^{1/2}q^2h^{k+1}\left(WZ_1+Z_2\right). \label{69}
	\end{align*} 
\end{remark}

\subsection{Comparing the three discretizations}

\revise{
The error bounds in~\Cref{cor3,thm,Thm7,rem4.9} specify the convergence of the different methods.  Here, we note some comparative advantages and disadvantages between the three approaches.
}
% Corollary 3 for Argyris with \Gamma_{3,2} BCs:(Bq^2) h^{t-2}|u|_t for 3 \leq t \leq 6 in wtd H^2-norm
% Thm 5 for Argyris w/ Nitsche: (Bq^4) h^{t-2}|u|_t for 4 \leq t \leq 6
% Theorem 7 for C0IP: (Bq^4)h^{min(t,k+1)-2}|u|_t in weighted broken H^2 norm for t \geq 4
% Remark 16 for mixed (Bq^4 + B^{1/2}q^4) h^{k+1} |u|_{k+1} (+ other terms on higher derivatives)
% Cor 4: u in H^{k+5}, k \geq 1

\revise{
  The conforming Argyris discretization requires the weakest regularity assumptions on $u$ when $\partial\Omega = \Gamma_{3,2}$ and Nitsche boundary conditions can be avoided, needing only $u\in H^3(\Omega)$.  For other boundary conditions, the minimal regularity assumptions for Argyris and C0IP are the same, requiring $u\in H^4(\Omega)$.  The mixed formulation has the strictest regularity assumptions, requiring $u\in H^6(\Omega)$ (and smoothness of the tensor $T$).  The Argyris discretization with $\partial\Omega = \Gamma_{3,2}$ also has the advantage of the best dependence on the problem parameters, with a bound depending on $Bq^2$, in comparison to $Bq^4$ for Argyris with general boundary conditions and C0IP, and $(B+B^{1/2})q^4$ for the mixed formulation.
}

\revise{
  The main disadvantage of the Argyris formulation is the requirement for high polynomial order: we require degree at least 5 in two dimensions (noting we state results here specific to $\text{ARG}_{5}(\Omega,\mathcal{T}_h)$) and degree at least 9 in three dimensions.  Using ninth-order elements in 3D is likely to be cost prohibitive. In addition, their implementation is difficult, and they are not currently implemented in the Firedrake package used in the numerical results in Section~\ref{sec:numerics}, so we cannot confirm this expectation.
  }

\revise{
  It is interesting to note that the error bounds for Argyris with general boundary conditions and the C0IP approach are identical when we consider the C0IP approach with $k=5$ and $4\leq t \leq 6$, to match the Argyris method as best we can, with the exception of the norms in which we measure the error (which are as similar as possible).  This points to a strong advantage of the C0IP approach, as it can be used for any polynomial order $k\geq 2$, providing much greater flexibility than the Argyris approach.  An important, but unresolved, question is whether effective linear solvers exist for this discretization, noting that the interior penalty terms tend to lead to conditioning issues in the resulting linear systems.
  }

\revise{
Finally, we note that the mixed approach seems to offer both advantages and disadvantages.  In comparison to C0IP at equal order, we see that the convergence bound for the mixed approach is better than that for C0IP by a factor of $h^2$, assuming sufficient regularity.  For example, if $u\in H^{10}(\Omega)$ and we use fifth-order polynomials, the convergence bounds for Argyris and C0IP scale like $h^4$, while those for the mixed method scale like $h^6$, albeit with a worse constant and in a weaker norm.  Numerical results that follow, however, show that for the case where $B = q^{-4}$, the mixed approach at order 3 offers better $L^2$ approximation of $u$ than the C0IP approach at equal order.  This is, of course, achieved at a cost of substantially more degrees of freedom in the mixed approach.  Preliminary experiments (not reported here) suggest that the development of effective preconditioners for the mixed approach is substantially simpler than for the other discretizations, but this needs to be investigated in future work.}

\section{Numerical experiments}\label{sec:numerics}

To verify the analyses of the three finite-element discretizations, we next present numerical experiments to measure convergence rates. The experiments were done using the finite-element package Firedrake~\cite{rathgeber2017firedrake}, which offers close integration with
PETSc for the linear solvers \cite{balay2018petsc,
	kirby2018solver}. All numerical experiments were run on a workstation with dual 8-core Intel Xeon 1.7 GHz CPUs and 384 GB of RAM.  While the development of efficient linear solvers for these discretizations is an important task, except when noted below, we consider only solution using the sparse direct solvers, PaStiX~\cite{HENON2002301} and MUMPS~\cite{amestoy2001}, which worked best for experiments in 2D and 3D, respectively. The discretization codes and major components of Firedrake used in these experiments are archived on Zenodo~\cite{zenodo/Firedrake-20220721.0}.
We use the method of manufactured solutions to estimate convergence rates, where we fix forcing terms and boundary data for the PDE to exactly match those for a known solution, $u$. 
\subsection{2D experiments}
In all experiments, we consider uniform triangular meshes of the unit square in two dimensions, generated by uniformly meshing the unit square into square elements with edge length $h = 1/N$, and then cutting each square into two triangles, from bottom left to top right.  For the tests below, we write the boundary of the unit square as $\partial\Omega=\Gamma_N\cup\Gamma_S\cup\Gamma_E\cup\Gamma_W$, denoting the North, South, East, and West edges of the square, and fix $\Gamma_{0,2}=\Gamma_S, \Gamma_{0,1}=\Gamma_N, \Gamma_{3,2}=\Gamma_E,$ and $\Gamma_{3,1}=\Gamma_W$. For mesh size $h$, we define $u_h$ to be the finite-element solution on the mesh and the approximation error to be $E_h=u-u_h$. We can measure $E_h$ in several ways, such as the absolute $L^2$-error, which we denote by ${Abs}_e(u_h, h)=\|E_h\|_0$. Similar definitions are used, as needed, for other quantities, such as the weighted $H^2$-error in $u_h$ as given in \eqref{33} and \eqref{C0-norm}, the weighted $L^2\times H^1$-error in $(u_h,\vec{v}_h)$ as given in \Cref{eq:wtd_product_norm}, and the weighted $L^2$-norm and $H(\mathrm{div})$-seminorm errors in $\vec{\alpha}_h$.
Table~\ref{DoFs} records the number of degrees of freedom in the system matrix for the conforming, C0IP, and mixed discretizations with $u\in \text{ARG}_{5}(\Omega,\mathcal{T}_h)$,	$u\in CG_{k}(\Omega,\mathcal{T}_h))$, $k=2,3,4$, and $\left(u,\vec{v},\vec{\alpha}\right)\in DG_{k}(\Omega,\mathcal{T}_h)\times V_{k+2}(\Omega,\mathcal{T}_h)\times RT_{k+1}(\Omega,\mathcal{T}_h)$, $k=1,2,3$, respectively, on the different grids of a unit square domain. We note that the theoretical results in Section~\ref{sec:discrete} establish convergence rates of $\mathcal{O}(h^{k-1})$ in the weighted $H^2$ norm and $\mathcal{O}(h^{k+1})$ convergence rates in the weighted $L^2\times H^1$ product norm for the C0IP and mixed discretizations, respectively.  Comparing, for example, C0IP with $k=3$ to the mixed method with $k=1$, we see that the mixed method requires about four times the number of degrees of freedom to achieve the same convergence rate.
\begin{table}
	\centering
	\caption{Numbers of degrees of freedom in the system matrix for the conforming , C0IP, and mixed discretizations with $u\in \text{ARG}_{5}(\Omega,\mathcal{T}_h)$,	$u\in CG_{k}(\Omega,\mathcal{T}_h))$, $k=2,3,4$, and $\left(u,\vec{v},\vec{\alpha}\right)\in DG_{k}(\Omega,\mathcal{T}_h)\times V_{k+2}(\Omega,\mathcal{T}_h)\times RT_{k+1}(\Omega,\mathcal{T}_h)$, $k=1,2,3$, respectively on a unit square domain. \label{DoFs}}
	\begin{tabular}{ c| c  |c c c |c c c}  
		\toprule
		\multicolumn{1}{c|}
		{${h^{-1}}$} & \multicolumn{1}{c}{Conforming}& \multicolumn{3}{|c|}{C0IP}&\multicolumn{3}{c}{mixed} \\ \midrule
		&Argyris& $k=2$ & $k=3$ &$k=4$ &$k=1$& $k=2$&$k=3$\\ \midrule
		$2^6$&37,766&16,641&37,249&66,049&140,290&267,650&43,5970 \\ 
		$2^7$&149,254&66,049&148,225&263,169&559,106&1,067,778 & 1,740,290\\ 
		$2^8$&593,414&263,169&591,361&1,050,625&2,232,322&4,265,474 & 6,953,986 \\ 	
		$2^9$&2,366,470&1,050,625&2,362,369&4,198,401&8,921,090 & 17,050,626&27,801602\\ 
		\bottomrule
	\end{tabular}
\end{table}

We first consider an exact solution given by $u=\sin\left(q\vec{\nu}\cdot \left[x,y\right]\right)$, with $q=40$, $\boldsymbol{T}=\vec{\nu}\otimes\vec{\nu}$, $\vec{\nu}=\left[\frac{3}{5},\frac{4}{5}\right]$, and $m=10$. We plot $\log({Abs}_e(\cdot))$ against $\log_2(1/h)$, so that slopes of the data plotted correspond to the experimental convergence rates. We approximate the slope, $\mathfrak{s}$, of each line using last two points.
	Figure~\ref{tab:ex1_bigB} presents results for $B=1$.  Here, and in the subsequent figures, we use solid markers on lines corresponding to data for Argyris and C0IP results, and open markers on results for the mixed formulations.  In all cases, we use solid green lines for the case of $u_h \in \text{ARG}_5(\Omega,\mathcal{T}_h)$, and dense dotted brown lines for $k=1$, dotted blue lines for $k=2$, dashed red lines for $k=3$, and dash-dotted orange lines for $k=4$ when we present data either for the C0IP method with $u_h\in CG_{k}(\Omega,\mathcal{T}_h)$ (which we consider for $k=2,3,4$) or the mixed method with $(u_h,\vec{v}_h,\vec{\alpha}_h)\in DG_{k}(\Omega,\mathcal{T}_h)\times V_{k+2}\times RT_{k+1}(\Omega,\mathcal{T}_h)$ (which we consider for $k=1,2,3$).

As shown in Figure~\ref{tab:ex1_bigB} (left), the Argyris method results in the best approximation on each grid (as expected, since it is the highest-order discretization) and we see an overall convergence rate in the weighted $H^2$ norm of $\mathcal{O}(h^4)$, as expected from the theoretical analysis in Section~\ref{ssec:conforming}.  Beyond the scope of those results, we observe slightly improved, $\mathcal{O}(h^5)$, convergence for the Argyris method in the $L^2$ norm.  For the C0IP method, we see a lack of convergence in the weighted $H^2$ norm for the case of $k=2$; this seems contrary to the theoretical analysis in Section~\ref{ssec:C0IP}, but we note that $Bq^4h$ is quite large for the values of $q$ and $h$ considered here, and the predicted convergence is seen in the next example when we consider a (much) smaller value for $B$.  Moreover, we observe the expected $\mathcal{O}(h^{k-1})$ convergence for $k=3$ and $4$. For even values of $k$, we see improved convergence in the $L^2$ norm in comparison with the weighted $H^2$ norm, giving a $\mathcal{O}(h^{k})$ for $k=2$ and $4$, but no improvement in the convergence rate in the $L^2$ norm for $k=3$.  Understanding these gaps in performance is an interesting question for future work.

Finally, for the mixed method, in Figure~\ref{tab:ex1_bigB} (right) we observe $\mathcal{O}(h^{k+1})$ convergence in both the $L^2$ norm of $u$ and weighted $H^1$ norm of $\vec{v}$, consistent with the analysis of this formulation in Section~\ref{ssec:mixed}.  Furthermore, we observe that the $L^2$ norm of $\vec{\alpha}$ converges like $\mathcal{O}(h^k)$ as predicted in Remark~\ref{rem4.9}, while we observe one order higher convergence for the $H(\text{div})$ seminorm error in $\vec{\alpha}$.  An important observation here is that the added cost of the mixed method over the C0IP approach appears to pay off, with improved convergence rates consistent with the analysis (and smaller errors overall); thus, while the mixed formulation has more degrees of freedom at the same order as C0IP, we gain something in the quality of our approximation for that price.  (Indeed, the mixed method with $k=3$ attains roughly the same weighted-norm convergence rate and errors as the fifth-order Argyris discretization.)

\begin{figure}
	\begin{subfigure}{.5\textwidth}
		\centering		
		\begin{tikzpicture}[scale=.55]
		\begin{semilogyaxis}[
		title=Errors for Argyris and C0IP,
		xlabel={$\log_2(1/h)$},
		xtick={6,7,8,9},
		ylabel={$\log\left(Abs_e(\cdot,h)\right)$},
		ymin=1e-14, ymax =3,
		legend pos=outer north east
		]
		\addplot[
		mark=triangle*,mark options={scale=2,solid}, color=seabornblue,dotted,thick
		]
		coordinates {
			(6,0.41369057711826757)
			(7,0.17354517085795468)
			(8,0.12352119994082526)
			(9,0.11607779971346532)
		};
		\addplot[
		mark=*,mark options={scale=1,solid},color=seabornblue,dotted,thick
		]
		coordinates {
			(6,0.25028722984593377)
			(7,0.08809160429998084)
			(8,0.024382598978190836)
			(9,0.006241568783881783)
		};
		\addplot[
		mark=triangle*,mark options={scale=2,solid},thick, color=seabornred,dashed
		]
		coordinates {
			(6,0.07087513760285062)
			(7,0.019175626219412826)
			(8,0.00490215933880805)
			(9,0.001233344457842171)
		};
		\addplot[
		mark=*,mark options={scale=1,solid},thick, color=seabornred,dashed
		]
		coordinates {
			(6,0.012779416424464157)
			(7,0.003118301048895839)
			(8,0.0007896305525275524)
			(9,0.00019904557926734206)
		};
		\addplot[
		mark=triangle*,mark options={scale=2,solid},thick, color=orange,dash pattern={on 7pt off 2pt on 1pt off 3pt}
		]
		coordinates {
			(6,0.0007463839415142511)
			(7,9.798189951248248e-05)
			(8,1.2418152920139371e-05)
			(9,1.5582994793484387e-06)
		};
		
		\addplot[
		mark=*,mark options={scale=1,solid},color=orange,dash pattern={on 7pt off 2pt on 1pt off 3pt},thick
		]
		coordinates {
			(6,2.1015285174195684e-05)
			(7, 1.2993118638683074e-06)
			(8, 8.250882870970825e-08)
			(9,5.22098775046618e-09)
		};
		
		\addplot[
		mark=triangle*,mark options={scale=2,solid},color=seaborngreen,thick
		]
		coordinates {
			(6,5.958162520111303e-06)
			(7,3.6298068767668254e-07)
			(8, 2.247062536071873e-08)
			(9,1.3988975204332626e-09)
		};
		
		\addplot[
		mark=*,mark options={scale=1,solid},color=seaborngreen,thick
		]
		coordinates {
			(6, 5.0377428450523854e-08)
			(7, 1.284866114809077e-09)
			(8, 3.93020063545978e-11)
			(9, 1.2981092585199718e-12)
		};
		
		\legend{ $\mathfrak{s}=-0.09$, $\mathfrak{s}=-1.97$, $\mathfrak{s}=-1.99$, $\mathfrak{s}=-1.99$, $\mathfrak{s}=-2.99$, $\mathfrak{s}=-3.98$, $\mathfrak{s}=-4.00$, $\mathfrak{s}=-4.92$}
		\end{semilogyaxis}
		\end{tikzpicture}		
	\end{subfigure}%
	\begin{subfigure}{.5\textwidth}
		\centering		
		\begin{tikzpicture}[scale=.55]
		\begin{semilogyaxis}[
		title=Errors for mixed formulation,
		xlabel={$\log_2(1/h)$},
		xtick={6,7,8,9},
		ylabel={$\log\left(Abs_e(\cdot,h)\right)$},
		ymin=1e-14, ymax =3,
		legend pos=outer north east
		]
		\addplot[
		mark=triangle,mark options={scale=2,solid}, color=brown,densely dotted,thick 
		]
		coordinates {
			(6, 3.37383522e-03)
			(7, 8.36600300e-04)
			(8, 2.08896979e-04)
			(9, 5.22167430e-05)
		};
		\addplot[
		mark=o,mark options={scale=1,solid}, color=brown,densely dotted,thick 
		]
		coordinates {
			(6, 4.14865739e-03)
			(7, 1.03640229e-03)
			(8, 2.59176781e-04)
			(9, 6.48010109e-05)
		};
		\addplot[
		mark=square,mark options={scale=2,solid},color=brown,densely dotted,thick 
		]
		coordinates {
			(6, 0.49847745)
			(7, 0.25935446)
			(8, 0.13111167)
			(9, 0.06576732)
		};
		
		\addplot[
		mark=star,mark options={scale=2,solid},color=brown,densely dotted,thick 
		]
		coordinates {
			(6, 2.45085973e+00)
			(7, 6.23040403e-01)
			(8, 1.56580373e-01)
			(9, 3.92177311e-02)
		};
		
		\addplot[
		mark=triangle,mark options={scale=2,solid}, color=seabornblue,dotted,thick 
		]
		coordinates {			
			(6, 1.38824397e-04)
			(7, 1.74560247e-05)
			(8, 2.18634468e-06)
			(9, 2.73495932e-07)
		};
		\addplot[
		mark=o,mark options={scale=1,solid},color=seabornblue,dotted,thick 
		]
		coordinates {			
			(6, 1.55448580e-04)
			(7, 1.94601676e-05)
			(8, 2.43343209e-06)
			(9, 3.04207493e-07)
		};
		
		\addplot[
		mark=square,mark options={scale=2,solid}, color=seabornblue,dotted,thick 
		]
		coordinates {			
			(6, 9.49636236e-03)
			(7, 2.10762806e-03)
			(8, 5.04155536e-04)
			(9, 1.24001902e-04)
		};
		\addplot[
		mark=star,mark options={scale=2,solid}, color=seabornblue,dotted,thick 
		]
		coordinates {			
			(6, 6.61873872e-02)
			(7, 8.16686855e-03)
			(8, 1.01724394e-03)
			(9, 1.27034742e-04)
		};
		
		\addplot[
		mark=triangle,mark options={scale=2,solid},color=seabornred, dashed,thick 
		]
		coordinates {			
			(6, 3.77361570e-06)
			(7, 2.35750580e-07)
			(8, 1.47362462e-08)
			(9, 9.25500675e-10)
		};	
		\addplot[
		mark=o,mark options={scale=1,solid},color=seabornred, dashed,thick 
		]
		coordinates {			
			(6, 1.54780720e-06)
			(7, 9.67020912e-08)
			(8, 6.04331375e-09)
			(9, 3.77698072e-10)
			
		};
		\addplot[
		mark=square,mark options={scale=2,solid},color=seabornred, dashed,thick 
		]
		coordinates {			
			(6, 1.09301751e-03)
			(7, 1.40031685e-04)
			(8, 1.76169413e-05)
			(9, 2.32546704e-06)
		};		
		
		\addplot[	
		mark=star,mark options={scale=2,solid},color=seabornred, dashed,thick ]	
		coordinates {			
			(6, 2.49205238e-03)
			(7, 1.57684594e-04)
			(8, 9.89496435e-06)
			(9, 6.22465607e-07)
			
		};		
		%		\addplot[dotted, domain=6:9]{45/(2^x)^3};
		\legend{$\mathfrak{s}=-2.00 $, $\mathfrak{s}=-2.00 $, $\mathfrak{s}=-1.00 $, $\mathfrak{s}=-2.00 $,$\mathfrak{s}=-3.00 $,$\mathfrak{s}= -3.00$,$\mathfrak{s}=-2.02$,$\mathfrak{s}=-3.00 $,$\mathfrak{s}=-3.99 $,$\mathfrak{s}=-4.00 $, $\mathfrak{s}=-2.92$, $\mathfrak{s}=-3.99$}		
		\end{semilogyaxis}
		\end{tikzpicture}
		
	\end{subfigure}%
	\caption{Absolute approximation errors and rate of convergence for $B=1$.  At left, results for $u_h \in \text{ARG}_5(\Omega,\mathcal{T}_h)$ (solid green lines),  and $u_h\in CG_{k}(\Omega,\mathcal{T}_h)$, $k={2,3,4}$ with the C0IP formulation, where dotted blue, dashed red, and dash-dotted orange lines present results for $k=2,3,4$, respectively.  Discs denote errors in the $L^2$ norm, while triangles denote errors in the appropriately weighted $H^2$ norm.  At right, results for $(u_h,\vec{v}_h,\vec{\alpha})_h\in DG_k(\Omega,\mathcal{T}_h)\times V_{k+2}(\Omega,\mathcal{T}_h)\times RT_{k+1}(\Omega,\mathcal{T}_h)$, with dense dotted brown, dotted blue, and dashed red lines presenting results for $k=1,2,3$, respectively. Here, discs and triangles denote the $L^2$ and weighted $H^1$ errors for $u_h$ and $\vec{v}_h$, respectively, while squares and stars denote the weighted $L^2(\Omega)$ and $H(\mathrm{div})$-seminorm errors for $\vec{\alpha}_h$.}\label{tab:ex1_bigB}
\end{figure}
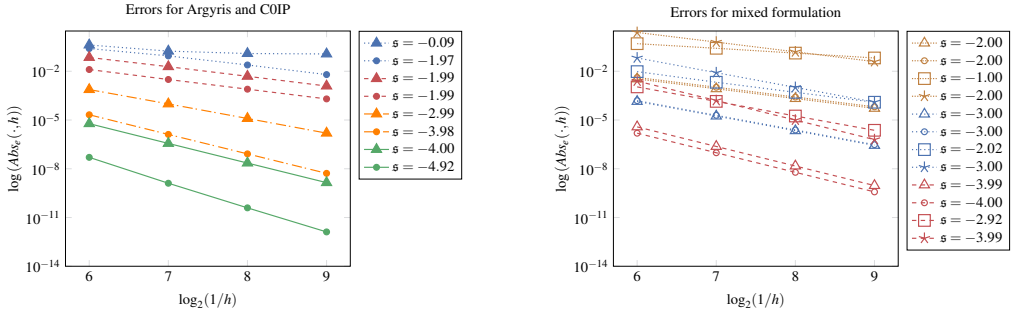

	In Figure~\ref{tab:ex1_smallB}, we consider the same example as above, but with $B=q^{-4}$, to test sensitivity to this parameter.  We first note that, in all but one case, the weighted $H^2$ norm (or $L^2\times H^1$ norm) errors decay at the same rates as in the $B=1$ example above, consistent with the theoretical error bounds.  The one exception is the case of $k=2$ for the C0IP discretization, where we go from no convergence with $B=1$ to convergence like $\mathcal{O}(h)$ for $B=q^{-4}$.  Considering the other norms, we see a slight improvement in the $L^2$ error for the Argyris case (from $\mathcal{O}(h^5)$ to almost $\mathcal{O}(h^6)$), but no other significant changes.  In particular, while our error estimates in all cases depend on quantities like $Bq^4$, we see only small changes in these errors, despite the large change in this value.  The one notable change is for the mixed method, where the errors in $\vec{\alpha}$ with $B=q^{-4}$ are seen to be equal to those with $B=1$ multiplied by a factor of $q^{-4}$. In these experiments, for the mixed formulation with $k = 3$ and $h = 1/512$, the direct solver failed to factorize the linear system in the available memory; therefore, a monolithic multigrid solver similar to the one described in \cite{farrell2021new} was used instead.
\begin{figure}
	\begin{subfigure}{.5\textwidth}
		\centering
		\begin{tikzpicture}[scale=.55]
		\begin{semilogyaxis}[
		title=Errors for Argyris and C0IP,
		xlabel={$\log_2(1/h)$},
		xtick={6,7,8,9},
		ylabel={$\log\left(Abs_e(\cdot,h)\right)$},
		ymin=1e-14, ymax =3,
		legend pos=outer north east
		]

		\addplot[
		mark=triangle*,mark options={scale=2,solid},color=seabornblue,dotted, thick
		]
		coordinates {
			(6, 0.11821146503986717)
			(7, 0.05922234665156641)
			(8, 0.029627065330476616)
			(9, 0.014814545919371206)
		};
		\addplot[
		mark=*,mark options={scale=1,solid},color=seabornblue,dotted,thick
		]
		coordinates {
			(6, 0.002214252776136992)
			(7, 0.0005697623883984998)
			(8, 0.00014428218733346707)
			(9, 3.6222895715883504e-05)
		};
		\addplot[
		mark=triangle*,mark options={scale=2,solid},color=seabornred,dashed,thick
		]
		coordinates {
			(6, 0.006439763881308099)
			(7, 0.0016215661347244923)
			(8, 0.0004061043740297412)
			(9, 0.00010156542633566701)
		};
		\addplot[
		mark=*,mark options={scale=1,solid},color=seabornred,dashed,thick
		]
		coordinates {
			(6, 2.8395352371512826e-05)
			(7, 5.363837696725863e-06)
			(8, 1.307464374623749e-06)
			(9, 3.2651698546213695e-07)
		};
		\addplot[
		mark=triangle*,mark options={scale=2,solid},thick,color=orange,dash pattern={on 7pt off 2pt on 1pt off 3pt}
		]
		coordinates {
			(6, 0.00020524156460125146)
			(7, 2.564766568744513e-05)
			(8, 3.206168263746667e-06)
			(9, 4.0080353040597843e-07)
		};
		
		\addplot[
		mark=*,mark options={scale=1,solid},thick, color=orange,dash pattern={on 7pt off 2pt on 1pt off 3pt}
		]
		coordinates {
			(6, 5.327595040017765e-07)
			(7, 1.884869680238103e-08)
			(8, 8.277853440947071e-10)
			(9, 4.4668925120356555e-11)
		};
		\addplot[
		mark=triangle*,mark options={scale=2,solid},color=seaborngreen,thick
		]
		coordinates {
			(6,5.835456470469256e-06)
			(7,3.591355121845937e-07)
			(8,2.235036083097049e-08)
			(9,1.3951388967370348e-09)
		};
		
		\addplot[
		mark=*,mark options={scale=1,solid},color=seaborngreen,thick
		]
		coordinates {
			(6,1.3587795161494706e-08)
			(7,2.014188653230311e-10 )
			(8, 3.236979218490259e-12)
			(9,6.038172819898544e-14)
		};
		\legend{ $\mathfrak{s}=-1.00$, $\mathfrak{s}=-1.99$, $\mathfrak{s}=-2.00$, $\mathfrak{s}=-2.00$, $\mathfrak{s}=-3.00$, $\mathfrak{s}=-4.21$, $\mathfrak{s}=-4.00$, $\mathfrak{s}=-5.74$}
		\end{semilogyaxis}
		
		\end{tikzpicture}
	\end{subfigure}
	\begin{subfigure}{.5\textwidth}
		\centering
		\begin{tikzpicture}[scale=.55]
		\begin{semilogyaxis}[
		title=Errors for mixed formulation,
		xlabel={$\log_2(1/h)$},
		xtick={6,7,8,9},
		ylabel={$\log\left(Abs_e(\cdot,h)\right)$},
		ymin=1e-14, ymax =3,
		legend pos=outer north east
		]
		\addplot[
		mark=triangle,mark options={scale=2,solid},thick,color=brown, densely dotted
		]
		coordinates {
			(6, 3.36005693e-03) 
			(7, 8.36368398e-04)
			(8, 2.08892947e-04)
			(9, 5.22166508e-05)
		};
		\addplot[
		mark=o,mark options={scale=1,solid},thick,color=brown, densely dotted 
		]
		coordinates {
			(6, 4.13719192e-03)
			(7, 1.03621520e-03)
			(8, 2.59173831e-04)
			(9, 6.48009620e-05)
		};
		\addplot[
		mark=square,mark options={scale=2,solid},thick,color=brown, densely dotted
		]
		coordinates {
			(6, 1.94826920e-07)
			(7, 1.01318086e-07)
			(8, 5.12163756e-08)
			(9, 2.56904775e-08)
		};
		
		\addplot[
		mark=star,mark options={scale=2,solid} ,thick,color=brown, densely dotted
		]
		coordinates {
			(6, 9.33106099e-07)
			(7, 2.41751775e-07)
			(8, 6.10616775e-08)
			(9, 1.53130331e-08)
		};

		\addplot[
		mark=triangle,mark options={scale=2,solid},thick,color=seabornblue,dotted 
		]
		coordinates {			
			(6, 1.38817894e-04)
			(7, 1.74559503e-05)
			(8, 2.18634324e-06)
			(9, 2.73495892e-07)
		};
		\addplot[
		mark=o,mark options={scale=1,solid},thick,color=seabornblue,dotted 
		]
		coordinates {			
			(6, 1.55448310e-04)
			(7, 1.94601616e-05)
			(8, 2.43343188e-06)
			(9, 3.04207486e-07)
		};
		
		\addplot[
		mark=square,mark options={scale=2,solid},thick,color=seabornblue,dotted 
		]
		coordinates {			
			(6, 3.711846052142909e-09)
			(7, 8.234896638160779e-10)
			(8, 1.9695014537825178e-10)
			(9, 4.843914878095911e-11)
		};
		\addplot[
		mark=star,mark options={scale=2,solid},thick,color=seabornblue,dotted 
		]
		coordinates {			
			(6, 2.5755677394805225e-08)
			(7, 3.1867494637091296e-09)
			(8, 3.9725365526380433e-10)
			(9, 4.9619637926264484e-11)
		};

		\addplot[
		mark=triangle,mark options={scale=2,solid},color=seabornred, dashed,thick 
		]
		coordinates {			
			(6, 3.77356680e-06)
			(7, 2.35750326e-07)
			(8, 1.47362441e-08)
			(9, 9.2550084e-10)
		};	
		\addplot[
		mark=o,mark options={scale=1,solid},color=seabornred, dashed,thick 
		]
		coordinates {			
			(6, 1.54780346e-06)
			(7, 9.67020361e-08)
			(8, 6.04331288e-09)
			(9, 3.77697994e-10)
		};
		\addplot[
		mark=square,mark options={scale=2,solid},color=seabornred, dashed,thick 
		]
		coordinates {			
			(6, 4.26966835e-10)
			(7, 5.47001195e-11)
			(8, 6.88162459e-12)
			(9, 9.08411871e-13)
		};		
		
		\addplot[	mark=star,mark options={scale=2,solid},color=seabornred, dashed,thick]	
		coordinates {			
			(6, 9.65204129e-10)
			(7, 6.14614219e-11)
			(8, 3.86310687e-12)
			(9, 2.43493969e-13)
		};		
		
		\legend{$\mathfrak{s}=-2.00 $, $\mathfrak{s}=-2.00$, $\mathfrak{s}=-1.00$, $\mathfrak{s}=-2.00$, $\mathfrak{s}=-3.00$, $\mathfrak{s}=-3.00$, $\mathfrak{s}=-2.02$, $\mathfrak{s}=-3.00$, $\mathfrak{s}=-3.99$, $\mathfrak{s}=-4.00$, $\mathfrak{s}=-2.92$, $\mathfrak{s}=-3.99$ }	
		\end{semilogyaxis}
		\end{tikzpicture}
	\end{subfigure}
	\caption{Absolute approximation errors and rate of convergence for $B=q^{-4}$.  At left, results for $u_h \in \text{ARG}_5(\Omega,\mathcal{T}_h)$ (solid green lines),  and $u_h\in CG_{k}(\Omega,\mathcal{T}_h)$, $k={2,3,4}$ with the C0IP formulation, where dotted blue, dashed red, and dash-dotted orange lines present results for $k=2,3,4$, respectively.  Discs denote errors in the $L^2$ norm, while triangles denote errors in the appropriately weighted $H^2$ norm.  At right, results for $(u_h,\vec{v}_h,\vec{\alpha}_h) \in DG_k(\Omega,\mathcal{T}_h)\times V_{k+2}(\Omega,\mathcal{T}_h)\times RT_{k+1}(\Omega,\mathcal{T}_h)$, with dense dotted brown, dotted blue, and dashed red lines presenting results for $k=1,2,3$, respectively. Here, discs and triangles denote the $L^2$ and weighted $H^1$ errors for $u_h$ and $\vec{v}_h$, respectively, while squares and stars denote the weighted $L^2(\Omega)$ and $H(\mathrm{div})$-seminorm errors for $\vec{\alpha}_h$.}\label{tab:ex1_smallB}
\end{figure}
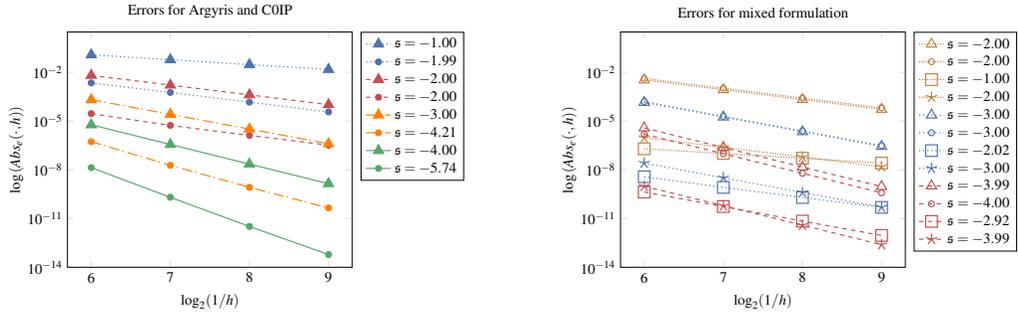

An important question is whether the dependence on $q$ in the theoretical results above is due to inefficient proof techniques, or is an actual dependence that is seen in the finite-element results.  In the next experiments, we fix the values of $m= 10, \ \text{and}\ h=2^{-7}$, while we vary both the manufactured solution and the values of $B$ and $q$.  Figure~\ref{tab:exq1} again considers $B=1$, considering the same solution as above, $u=\sin(q(\frac{3x}{5}+\frac{4y}{5}))$ (left), and a $q$-independent solution, $u_1=100\sin(2 \pi x+3\pi y)\left(xy(1-x)(1-y)\right)^3$ (right). In ~\ref{tab:exq1} (left), we see weighted $H^2$-norm convergence scaling like $\mathcal{O}(q^4)$ for the Argyris results, but like $\mathcal{O}(q^3)$ (or slightly better) for the C0IP and mixed methods.  In contrast, in the $L^2$ norm, the Argyris results scale worse than $\mathcal{O}(q^6)$, the C0IP results scale like $\mathcal{O}(q^4)$, while the mixed method still scales like $\mathcal{O}(q^3)$.  In all cases, these already outperform the theoretical bounds, since the $H^t$-seminorm of $u$ scales like $q^t$.  In contrast, in Figure~\ref{tab:exq1} (right), we see either no growth (for the mixed method) or dependence like negative powers of $q$ (for the Argyris and C0IP methods).  This strongly suggests that it may be possible to improve the dependence on $q$ in the results above, which we leave as a question for future research.  Figure~\ref{tab:exq2} shows similar results for the case of $B=q^{-4}$, highlighting that the dependence on $Bq^4$ (and similar terms) is more likely an artifact of our analysis than an inherent limit on these methods.

\begin{figure}
	\begin{subfigure}{.5\textwidth}
		\centering
		\begin{tikzpicture}[scale=.55]
		\begin{semilogyaxis}[
		title=Errors for $q$-dependent solution,
		xlabel={$\log_2(q)$},
		ymin=1e-14, ymax=3,
		ylabel={$\log\left(Abs_e(\cdot,\frac{1}{2^7})\right)$},
		legend pos=outer north east
		]

		\addplot[
		mark=*,mark options={scale=1,solid},color=seaborngreen,thick
		]
		coordinates {
			(3, 3.8305987158105327e-13)
			(4, 1.180804806660733e-11)
			(5, 3.9752329415491665e-10)
			(6, 1.4817442736280198e-08)
			(7, 1.266563353447447e-06)
		};
		
		\addplot[
		mark=triangle*,mark options={scale=2,solid},color=seaborngreen,thick
		]
		coordinates {
			(3, 5.796596744395899e-10)
			(4, 9.211779090037042e-09)
			(5, 1.480828839554257e-07)
			(6, 2.4070984779249563e-06)
			(7, 4.036131984468305e-05)
		};

		\addplot[
		mark=*,mark options={scale=1,solid},color=orange,dash pattern={on 7pt off 2pt on 1pt off 3pt},thick
		]
		coordinates {
			(3,4.298901006803504e-05)
			(4,0.00022086780269651565)
			(5,0.0015828449854843793)
			(6,0.012890844036280949)
			(7,0.2004452030056078)
		};
		
		\addplot[
		mark=triangle*,mark options={scale=2,solid},color=orange,dash pattern={on 7pt off 2pt on 1pt off 3pt},thick
		]
		coordinates {
			(3, 0.0003590393573491073)
			(4, 0.0019802497484291887)
			(5, 0.011035949675891473)
			(6, 0.06038362230703038)
			(7, 0.35374822539596373)
		};
		
		\addplot[
		mark=o,mark options={scale=1,solid},color=seabornblue,dotted,thick
		]
		coordinates {
			(3, 1.55739433e-07)
			(4, 1.24650671e-06)
			(5, 9.96923232e-06)
			(6, 7.96682582e-05)
			(7, 6.36195597e-04)
		};
		\addplot[
		mark=triangle,mark options={scale=2,solid},color=seabornblue,dotted,thick
		]
		coordinates {
			(3, 2.09292183e-07)
			(4, 1.67503035e-06)
			(5, 1.33940282e-05)
			(6, 1.06960566e-04)
			(7, 8.52009201e-04)
		};

		\legend{$\mathfrak{s}=6.42$,$\mathfrak{s}=4.07$,$\mathfrak{s}=3.96$,$\mathfrak{s}=2.55$, $\mathfrak{s}=3.00$,$\mathfrak{s}=2.99$}
		\end{semilogyaxis}
		\end{tikzpicture}

	\end{subfigure}%
	\begin{subfigure}{.5\textwidth}
		\centering
		\begin{tikzpicture}[scale=.55]
		\begin{semilogyaxis}[
		title=Errors for $q$-independent solution,
		xlabel={$\log_2(q)$},
		ymin=1e-14, ymax=3,
		ylabel={$\log\left(Abs_e(\cdot,\frac{1}{2^7})\right)$},
		legend pos=outer north east
		]

		\addplot[
		mark=*,mark options={scale=1,solid},color=seaborngreen,thick
		]
		coordinates {
			(3, 1.8449940569259869e-09)
			(4, 5.891653475231165e-10)
			(5, 1.7549728799693162e-10)
			(6, 5.051799642172539e-11)
			(7, 1.382384046619963e-11)
		};
		
		\addplot[
		mark=triangle*,mark options={scale=2,solid},color=seaborngreen,thick
		]
		coordinates {
			(3, 2.49197859287577e-07)
			(4, 6.215028512809969e-08)
			(5, 1.560934589440694e-08)
			(6, 3.93842403750654e-09)
			(7, 9.994646348972032e-10)
		};
		
		\addplot[
		mark=*,mark options={scale=1,solid},color=orange,dash pattern={on 7pt off 2pt on 1pt off 3pt},thick
		]
		coordinates {
			(3, 5.158993967565061e-06)
			(4, 4.463478770067244e-06)
			(5, 4.6848811325751127e-07)
			(6, 1.007337565937616e-07)
			(7, 2.4683104100446077e-08)
		};
		
		\addplot[
		mark=triangle*,mark options={scale=2,solid},color=orange,dash pattern={on 7pt off 2pt on 1pt off 3pt},thick
		]
		coordinates {
			(3, 4.7505308733790106e-05)
			(4, 1.6907813974380376e-05)
			(5, 5.611340680872402e-06)
			(6, 1.9648500465547696e-06)
			(7, 6.912221289688313e-07)
		};	
		
		\addplot[
		mark=o,mark options={scale=1,solid},color=seabornblue,dotted,thick
		]
		coordinates {
			(3, 1.05054270e-08)
			(4, 1.05054361e-08)
			(5, 1.05054632e-08)
			(6, 1.05054844e-08)
			(7, 1.05055848e-08)
		};
		
		\addplot[
		mark=triangle,mark options={scale=2,solid},color=seabornblue,dotted,thick
		]
		coordinates {
			(3, 1.36715505e-08)
			(4, 1.36616621e-08)
			(5, 1.36609368e-08)
			(6, 1.36602222e-08)
			(7, 1.36577360e-08)
		};

		\legend{$\mathfrak{s}= -1.87$,$\mathfrak{s}=-1.98 $,$\mathfrak{s}= -2.03$,$\mathfrak{s}=- 1.51$, $\mathfrak{s}= 0.00$,$\mathfrak{s}= 0.00$}
		\end{semilogyaxis}
		\end{tikzpicture}
		
	\end{subfigure}
	
	\caption{The $L^2$ (discs) and weighted $H^2$ (triangles) absolute approximation errors at $1/h=2^7$ and different values of $q$, for the conforming method, with $u_h\in \text{ARG}_5(\Omega,\mathcal{T}_h)$ (solid green lines), the C0IP method with $u_h\in CG_3(\Omega,\mathcal{T}_h)$ (dash-dotted orange lines), and the mixed method with $(u_h,\vec{v}_h)\in DG_{2}(\Omega,\mathcal{T}_h)\times V_{4}(\Omega,\mathcal{T}_h)$ (dotted blue lines) for $B=1$. Left: $u=\sin(q\left(\frac{3x}{5}+\frac{4y}{5}\right))$. Right: $u=100\sin\left(2\pi x+3\pi y\right)\left(xy(1-x)(1-y)\right)^3$.}\label{tab:exq1}
\end{figure}
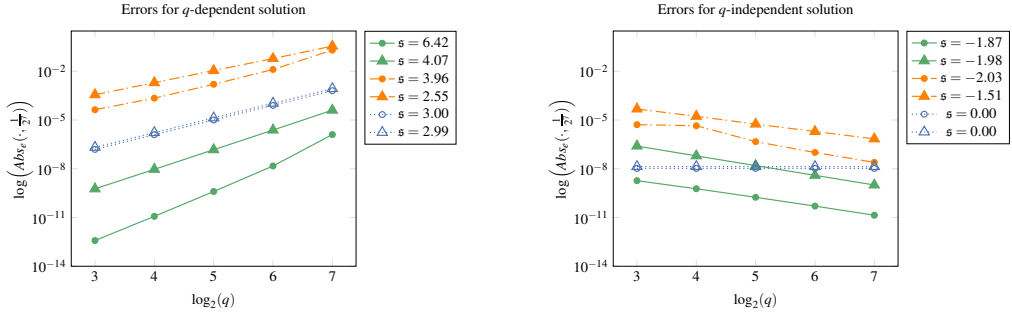

\begin{figure}
	\begin{subfigure}{.5\textwidth}
		\centering
		\begin{tikzpicture}[scale=.55]
		\begin{semilogyaxis}[
		title=Errors for $q$-dependent solution,
		xlabel={$\log_2(q)$},
		ymin=1e-14, ymax=3,
		ylabel={$\log\left(Abs_e(\cdot,\frac{1}{2^7})\right)$},
		legend pos=outer north east
		]

		\addplot[
		mark=*,mark options={scale=1,solid},color=seaborngreen,thick				]
		coordinates {
			(3, 5.937243450779367e-14)
			(4, 1.2724324213571643e-12)
			(5, 5.377057121024738e-11)
			(6, 3.4374973550783754e-09)
			(7, 2.525941560846035e-07)
		};
		
		\addplot[
		mark=triangle*,mark options={scale=2,solid},color=seaborngreen,thick
		]
		coordinates {
			(3, 5.77386141104243e-10)
			(4, 9.158894900701534e-09)
			(5, 1.467770762756348e-07)
			(6, 2.3716265794443516e-06)
			(7, 3.938195663721402e-05)
		};

		\addplot[
		mark=*,mark options={scale=1,solid},color=orange,dash pattern={on 7pt off 2pt on 1pt off 3pt},thick
		]
		coordinates {
			(3,1.3417006010305702e-06)
			(4, 2.5466741753334956e-06)
			(5, 4.526163019324708e-06)
			(6, 1.1271962230548402e-05)
			(7, 0.00016065768083667965)
		};
		
		\addplot[
		mark=triangle*,mark options={scale=2,solid},color=orange,dash pattern={on 7pt off 2pt on 1pt off 3pt},thick
		]
		coordinates {
			(3, 8.105529990336901e-05)
			(4, 0.0002869656089125352)
			(5, 0.001053994568342978 )
			(6, 0.004101223233536515 )
			(7, 0.01687386645446957)
		};

		\addplot[
		mark=o,mark options={scale=1,solid},color=seabornblue,dotted,thick
		]
		coordinates {
			(3, 1.55739431e-07)
			(4, 1.24650663e-06)
			(5, 9.96923014e-06)
			(6, 7.96681073e-05)
			(7, 6.34875661e-04)
		};
		\addplot[
		mark=triangle,mark options={scale=2,solid},color=seabornblue,dotted,thick
		]
		coordinates {
			(3, 2.09292147e-07)
			(4, 1.67502971e-06)
			(5, 1.33940120e-05)
			(6, 1.06959464e-04)
			(7, 8.49949862e-04)
		};

		\legend{$\mathfrak{s}=6.20$, $\mathfrak{s}=4.05$, $\mathfrak{s}=3.83$, $\mathfrak{s}=2.04$,  $\mathfrak{s}=2.99$, $\mathfrak{s}=2.99$}
		\end{semilogyaxis}
		\end{tikzpicture}

	\end{subfigure}%
	\begin{subfigure}{.5\textwidth}
		\centering
		\begin{tikzpicture}[scale=.55]
		\begin{semilogyaxis}[
		title=Errors for $q$-independent solution,
		xlabel={$\log_2(q)$},
		ymin=1e-14, ymax=3,
		ylabel={$\log\left(Abs_e(\cdot,\frac{1}{2^7})\right)$},
		legend pos=outer north east
		]

		\addplot[
		mark=*,mark options={scale=1,solid},color=seaborngreen,thick				]
		coordinates {
			(3, 6.748989116856734e-11)
			(4, 1.993772376900343e-11)
			(5, 6.604797241594618e-12)
			(6, 4.75524753116571e-12)
			(7, 4.680552402189094e-12)
		};
		
		\addplot[
		mark=triangle*,mark options={scale=2,solid},color=seaborngreen,thick
		]
		coordinates {
			(3, 2.476935380358304e-07)
			(4, 6.15009442056466e-08)
			(5, 1.536351915220152e-08)
			(6, 3.841823693304671e-09)
			(7, 9.609627162658454e-10)
		};

		\addplot[
		mark=*,mark options={scale=1,solid},color=orange,dash pattern={on 7pt off 2pt on 1pt off 3pt},thick
		]
		coordinates {
			(3, 4.649448285759746e-07)
			(4, 2.768430298752057e-08)
			(5, 4.562446435217226e-09)
			(6,  9.142710136874816e-10)
			(7, 4.226980306757609e-10)
		};
		
		\addplot[
		mark=triangle*,mark options={scale=2,solid},color=orange,dash pattern={on 7pt off 2pt on 1pt off 3pt},thick
		]
		coordinates {
			(3, 9.26322331250278e-06)
			(4, 2.189940180534842e-06)
			(5, 5.244884774149189e-07)
			(6, 1.3593448693095967e-07)
			(7, 3.8427144991458196e-08)
		};	
		\addplot[
		mark=o,mark options={scale=1,solid},color=seabornblue,dotted,thick
		]
		coordinates {
			(3, 1.05054264e-08)
			(4, 1.05054283e-08)
			(5, 1.05054302e-08)
			(6, 1.05054360e-08)
			(7, 1.05055639e-08)
		};
		
		\addplot[
		mark=triangle,mark options={scale=2,solid},color=seabornblue,dotted,thick
		]
		coordinates {
			(3, 1.36715431e-08)
			(4, 1.36616098e-08)
			(5, 1.36608132e-08)
			(6, 1.36600056e-08)
			(7, 1.36555747e-08)
		};

		\legend{ $\mathfrak{s}=-0.02 $, $\mathfrak{s}=-2.00 $, $\mathfrak{s}=-1.11 $, $\mathfrak{s}=-1.82 $,  $\mathfrak{s}=0.00 $, $\mathfrak{s}=0.00 $}
		\end{semilogyaxis}
		\end{tikzpicture}
		
	\end{subfigure}
	
	\caption{The $L^2$ (discs) and weighted $H^2$ (triangles) absolute approximation errors at $1/h=2^7$ and different values of $q$, for the conforming method, with $u_h\in \text{ARG}_5(\Omega,\mathcal{T}_h)$ (solid green lines), the C0IP method with $u_h\in CG_3(\Omega,\mathcal{T}_h)$ (dash-dotted orange lines), and the mixed method with $(u_h,\vec{v}_h)\in DG_{2}(\Omega,\mathcal{T}_h)\times V_{4}(\Omega,\mathcal{T}_h)$ (dotted blue lines) for $B=q^{-4}$. Left: $u=\sin(q\left(\frac{3x}{5}+\frac{4y}{5}\right))$. Right: $u=100\sin\left(2\pi x+3\pi y\right)\left(xy(1-x)(1-y)\right)^3$.}\label{tab:exq2}
\end{figure}
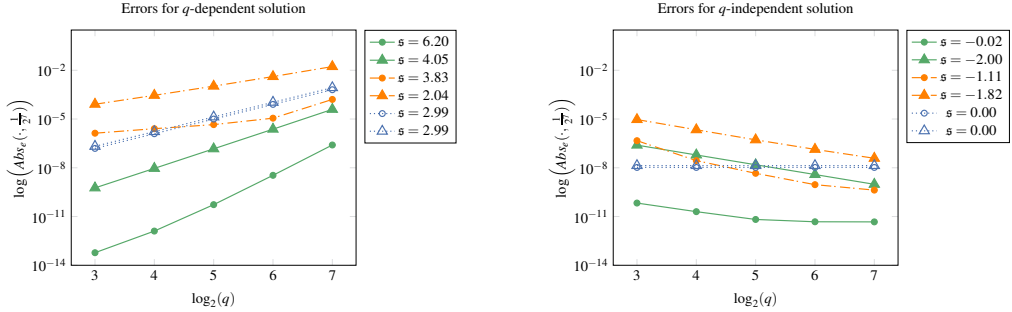
\subsection{3D experiments}
In these experiments, we consider a test on the unit cube domain, with right-hand side and boundary conditions chosen so that $\partial\Omega = \Gamma_{3,2}$ and the exact solution $u_{ex}=\sin\left(q\vec{\nu}\cdot[x, y, z]\right)$, with $q=10$, $B=q^{-4}$, $\boldsymbol{T}=\vec{\nu}\otimes\vec{\nu}$, $\vec{\nu}= [\frac{3}{13}, \frac{4}{13}, \frac{12}{13}]$, and $m=10$. For this problem it is not possible to conduct experiments with a conforming scheme in Firedrake, since the Zhang element~\cite{zhang2009} is not available. Table \ref{Timings-3d} presents results for the C0IP method with $u_h\in CG_{3}(\Omega,\mathcal{T}_h)$ and the mixed formulation with $(u_h, \vec{v}_h, \vec{\alpha}_h)\in DG_{1}(\Omega,\mathcal{T}_h)\times [CG_{3}(\Omega,\mathcal{T}_h)]^3\times RT_{2}(\Omega,\mathcal{T}_h)$ (see Remark~\ref{rem:3d_mixed}),
where we measure the error in the $L^2$ norm as well as the weighted $H^2$ norms, i.e.~the weighted $H^2$ norm defined in \eqref{C0-norm} and the $L^2\times H^1$ norm defined in \eqref{eq:wtd_product_norm}. The Table also reports the wall-clock time (in minutes) that the direct solver, MUMPS, required to solve the arising linear system. The experimental convergence rates are computed using the errors at $h^{-1}=2^4$ and $2^5$. We note that MUMPS
needed more time to solve the linear system arising from the mixed formulation than for the C0IP system at the same level.  
Finally, we point out that our analysis for the mixed formulation does not trivially generalize for the 3D case and finite-element enrichment for the approximation space of $\vec{v}$ is sometimes needed (\revise{see Remark~\ref{rem:3d_mixed}). While the discretization $DG_{1}(\Omega,\mathcal{T}_h)\times [CG_{3}(\Omega,\mathcal{T}_h)]^3\times RT_{2}(\Omega,\mathcal{T}_h)$ gives the optimal convergence rates when $\partial\Omega= \Gamma_{3,2}$, as is shown in Table \eqref{Timings-3d}, numerical experiments showed that the same discretization is unstable for $\partial\Omega = \Gamma_{0,1}$, but can be made stable by such enrichment. }
 
\begin{table}
	\centering
	\caption{Absolute approximation errors and wall-clock times in minutes for the C0IP method with $u_h\in CG_{3}(\Omega,\mathcal{T}_h)$ and the mixed formulation with $(u_h,\vec{v}_h,\vec{\alpha}_h)\in DG_{1}(\Omega,\mathcal{T}_h)\times [CG_{3}(\Omega,\mathcal{T}_h)]^3\times RT_{2}(\Omega,\mathcal{T}_h)$. \label{Timings-3d}}
	\begin{tabular}{ c| c c c |c c c}  
		\toprule
		\multicolumn{1}{c}{$h^{-1}$}& \multicolumn{3}{|c|}{Mixed}&\multicolumn{3}{c}{C0IP} \\ \midrule
		& $\|E_h\|_0$ & $\|E_h\|_{0,q,1}$ &Time&$\|E_h\|_0$&$\tnorm{E_h}{h}$&Time\\ \midrule
		$2^2$ &1.47$\times 10^{-1}$ &1.83$\times 10^{-1}$ &0.16& 3.86$\times 10^{-2}$&3.37$\times 10^{-1}$ & 0.01\\ 
		$2^3$&5.19 $\times 10^{-2}$&6.08$\times 10^{-2}$ & 0.06&4.17$\times 10^{-3}$ &1.06$\times 10^{-1}$ &0.02 \\ 
		$2^4$&1.41$\times 10^{-2}$ &1.55$\times 10^{-2}$ &1.45 & 3.79$\times 10^{-4}$& 2.95$\times 10^{-2}$&0.36 \\ 	
		$2^5$&3.60$\times 10^{-3}$&3.86$\times 10^{-3}$& 339.56& 4.56$\times 10^{-5}$&7.72 $\times 10^{-3}$& 14.40\\
		Rates&1.97& 2.00& &3.05 &1.93& \\
		\bottomrule
	\end{tabular}
\end{table}
\section{Conclusions}\label{sec:conclusions}
We consider and compare different finite-element techniques to discretize a fourth-order PDE describing the density variation of a smectic A liquid crystal. These models have two complications in comparison to classical biharmonic operators, as they are more akin to Helmholtz operators than elliptic ones, and involve Hessian-squared (div-div-grad-grad) operators rather than the classical biharmonic operator (div-grad-div-grad), with boundary conditions that preclude this potential simplification. We analyzed $H^2$-conforming, C0IP, and mixed finite-element methods.

In the $H^2$-conforming case, we use $C^1$ Argyris/Zhang elements.  In practice, these elements can be expensive, due to their high order (fifth-order piecewise polynomials in 2D and ninth-order in 3D), but they offer high-order approximation of smooth solutions. In this case, we implement essential boundary conditions using non-symmetric Nitsche-type penalty methods, which somewhat degrades the error estimates from the (impractical) case where essential BCs are imposed strongly. C0IP methods have the advantage over $H^2$-conforming elements that there is greater flexibility in choosing the order of approximation, at the cost of more complicated weak forms, where $C^1$-conformity is weakly enforced by penalizing inter-element jumps in the first derivative.  Our error estimates in this case match the dependence on $q$ from the conforming case, but with an $h$-dependence in line with the polynomial order used for the $CG$ space. Finally, we consider a three-field mixed ement formulation that explicitly introduces the gradient as an independent variable constrained using a Lagrange multiplier. The mixed formulation offers better robustness in $B$ than the other schemes, in exchange for higher degree of freedom counts. Numerical results confirm the theoretical expectations.

The mixed formulation proposed here was motivated by the observation that design of optimal linear solvers for the C0IP formulation is not straightforward (with direct solvers used in~\cite{xia2021structural}), coupled with the observed success of monolithic multigrid methods for a similar mixed discretization for the $H^2$-elliptic case of fourth-order operators in~\cite{farrell2021new}.  A natural step for future work is in extending these linear solvers to the mixed formulation proposed herein, in parallel to investigating effective linear solver strategies for the other discretizations.  As this work is motivated by considering the more complex models in~\cite{xia2021structural}, coupling the smectic density to a director field or tensor-valued order parameter, another natural direction for future research is to extend the analysis to mixed formulations of the energy minimization problem associated with Equation~\eqref{E}.

%\section*{Acknowledgments}
%The authors thank the anonymous reviewers for their valuable suggestions.
\section*{Funding}
The work of AH and SM was partially supported by an NSERC discovery grant. The work of PEF is supported by EPSRC grants EP/R029423/1 and EP/W026163/1.

\bibliographystyle{agsm}
\bibliography{IMAJNA-bibliography}

\end{document}